\documentclass{amsart}

\usepackage[T1]{fontenc}
\usepackage{enumerate, amsmath, amsfonts, amssymb, amsthm, mathrsfs, wasysym, graphics, graphicx, xcolor, url, hyperref, hypcap, xargs, multicol, pdflscape, multirow, hvfloat, array, ae, aecompl, pifont, mathtools, a4wide, float, blkarray, overpic, nicefrac}
\usepackage[shortlabels, inline]{enumitem}
\usepackage[noabbrev,capitalise]{cleveref}
\usepackage[normalem]{ulem}
\usepackage{marginnote}
\hypersetup{colorlinks=true, citecolor=darkblue, linkcolor=darkblue}
\usepackage[all]{xy}
\usepackage{tikz}
\usepackage{tikz-cd}
\usetikzlibrary{trees, decorations, decorations.pathmorphing, decorations.markings, decorations.shapes, shapes, arrows, matrix, calc, fit, intersections, patterns, angles}
\graphicspath{{figures/}{figures/diagonals/}{figures/walks/}{figures/tubes/}{figures/blocks/}}
\makeatletter\def\input@path{{figures/}}\makeatother
\usepackage{caption}
\usepackage[export]{adjustbox}

\newtheorem{theorem}{Theorem}[section]
\newtheorem{corollary}[theorem]{Corollary}
\newtheorem{proposition}[theorem]{Proposition}
\newtheorem{lemma}[theorem]{Lemma}

\newtheorem*{theorem*}{Theorem}%[section]

\theoremstyle{definition}
\newtheorem{definition}[theorem]{Definition}
\newtheorem{example}[theorem]{Example}
\newtheorem{remark}[theorem]{Remark}

\crefname{equation}{Equation}{Equations}

\newcommand{\R}{\mathbb{R}} % reals
\newcommand{\HH}{\mathbb{H}} % hyperplane
\renewcommand{\c}[1]{{\mathcal{#1}}} % call letters
\renewcommand{\b}[1]{{\boldsymbol{#1}}} % bold letters
\newcommand{\set}[2]{\left\{ #1 \;\middle|\; #2 \right\}} % set notation
\newcommand{\bigset}[2]{\big\{ #1 \;\big|\; #2 \big\}} % big set notation
\newcommand{\Bigset}[2]{\Big\{ #1 \;\Big|\; #2 \Big\}} % Big set notation
\newcommand{\ssm}{\smallsetminus} % small set minus
\newcommand{\dotprod}[2]{\langle \, #1 \; | \; #2 \, \rangle} % dot product
\newcommand{\eqdef}{\mbox{\,\raisebox{0.2ex}{\scriptsize\ensuremath{\mathrm:}}\ensuremath{=}\,}} % :=
\renewcommand{\implies}{\Rightarrow} % imply sign
\DeclareMathOperator{\conv}{conv} % convex hull
\newcommand{\ie}{\textit{i.e.}~} % id est
\newcommand{\eg}{\textit{e.g.}~} % exempli gratia
\newcommand{\aka}{\textit{a.k.a.}~} % also known as
\definecolor{darkblue}{rgb}{0,0,0.7} % darkblue color
\definecolor{green}{RGB}{57,181,74} % green color
\definecolor{violet}{RGB}{147,39,143} % violet color
\newcommand{\red}{\color{red}} % red command
\newcommand{\blue}{\color{blue}} % blue command
\newcommand{\orange}{\color{orange}} % orange command
\newcommand{\darkblue}{\color{darkblue}} % darkblue command
\newcommand{\defn}[1]{\textsl{\darkblue #1}} % emphasis of a definition
\newcommand{\para}[1]{\medskip\noindent\textsc{#1.}} % paragraph
\newcommand{\compactVectorT}[3]{\begin{bmatrix} #1 \\[-.1cm] #2 \\[-.1cm] #3 \end{bmatrix}}

\usepackage{todonotes}

\newcommandx{\Asso}[2][1=n,2={}]{\mathsf{Asso}^{#2}(#1)} % associahedron
\newcommandx{\Nest}[2][1=\building,2={}]{\mathsf{Nest}^{#2}(#1)} % associahedron
\newcommandx{\Zono}[2][1=n,2={}]{\mathsf{Zono}^{#2}(#1)} % zonotope
\newcommandx{\Fan}[1][1=F]{\mathcal{#1}} % fan
\newcommand{\gvector}[1]{\b{g}(#1)} % g-vector of #1
\newcommand{\gvectors}[1]{\b{g}(#1)} % g-vectors of #1
\newcommandx{\nestedFan}[1][1=\quiver]{\mathcal{F}(#1)} % g-vector fan
\newcommand{\typeCone}{\mathbb{TC}} % type cone
\newcommand{\ctypeCone}{\smash{\overline{\mathbb{TC}}}} % type cone
\newcommandx{\coefficient}[3][1={\b{s}}, 2=\b{r}, 3=\b{r}']{\alpha_{#2,#3}(#1)} % coefficient in linear dependence

\newcommand{\ground}{V} % ground set
\newcommandx{\graphG}[1][1=G]{#1} % graph
\newcommandx{\hypergraph}[1][1=H]{\graphG[#1]} % hypergraph
\newcommandx{\tube}[1][1=t]{\mathsf{#1}} % tube
\newcommandx{\tubes}[1][1=\graphG]{\building#1} % all tubes
\newcommandx{\tubing}[1][1=T]{\mathsf{#1}} % tubing
\newcommand{\connectedComponents}{\kappa} % connected components
\newcommand{\nestedComplex}{\mathcal{N}} % nested complex
\newcommand{\nonDisconnecting}{\mathrm{nd}} % non-disconnecting vertices
\newcommand{\building}{\mathcal{B}} % building set
\newcommand{\elementary}{\varepsilon} % elementary subsets of a building set
\newcommand{\maximalBlocks}{\mu} % maximal blocks
\newcommandx{\nested}[1][1=N]{\mathcal{#1}} % nested set
\newcommand{\rootset}[2]{\b{r}(#1,#2)} % root of element of nested set
\newcommand{\leaving}[3]{#1 \vdash #2 \subseteq #3} % leaving property

\makeatletter
\def\part{\@startsection{part}{1}%
\z@{.7\linespacing\@plus\linespacing}{.8\linespacing}%
{\LARGE\sffamily\centering}}
\makeatother

\makeatletter
\def\l@section{\@tocline{1}{5pt}{0pc}{}{}}
\makeatother
\let\oldtocpart=\tocpart
\renewcommand{\tocpart}[2]{\sc\large\oldtocpart{#1}{#2}}
\let\oldtocsection=\tocsection
\renewcommand{\tocsection}[2]{\bf\oldtocsection{#1}{#2}}
\let\oldtocsubsubsection=\tocsubsubsection
\renewcommand{\tocsubsubsection}[2]{\quad\oldtocsubsubsection{#1}{#2}}

\title{Deformation cones of graph associahedra and nestohedra}

\thanks{Partially supported by the French ANR grants CAPPS~17\,CE40\,0018, and CHARMS~19\,CE40\,0017.}

\author{Arnau Padrol}
\address[Arnau Padrol]{Institut de Math\'ematiques de Jussieu - Paris Rive Gauche, Sorbonne Universit\'e, Paris}
\email{arnau.padrol@imj-prg.fr}
\urladdr{\url{https://webusers.imj-prg.fr/~arnau.padrol/}}

\author{Vincent Pilaud}
\address[Vincent Pilaud]{CNRS \& LIX, \'Ecole Polytechnique, Palaiseau}
\email{vincent.pilaud@lix.polytechnique.fr}
\urladdr{\url{http://www.lix.polytechnique.fr/~pilaud/}}

\author{Germain Poullot}
\address[Germain Poullot]{Institut de Math\'ematiques de Jussieu - Paris Rive Gauche, Sorbonne Universit\'e,~Paris}
\email{germain.poullot@imj-prg.fr}
\urladdr{\url{https://webusers.imj-prg.fr/germain.poullot}}

\begin{document}

\begin{abstract}
We give the facet description of the deformation cones of graph associahedra and nestohedra, generalizing the classical parametrization of the family of deformed permutahedra by the cone of submodular functions.
When the underlying building set is made of intervals, this leads in particular to the construction of kinematic nestohedra generalizing the kinematic associahedra that recently appeared in the theory of scattering amplitudes.
\end{abstract}

\maketitle

\setcounter{tocdepth}{2}
\tableofcontents

%%%%%%%%%%%%%%%%%%%%%%%%%%%%%%%%%%%%%%%

\section*{Introduction}

A \defn{deformation} of a polytope~$P$ can be equivalently described as 
\begin{enumerate*}[(i)]
 \item a polytope whose normal fan coarsens the normal fan of~$P$~\cite{McMullen-typeCone}, 
 \item a Minkowski summand of a dilate of~$P$~\cite{Meyer,Shephard},
 \item a polytope obtained from~$P$ by perturbing the vertices so that the directions of all edges are preserved~\cite{Postnikov,PostnikovReinerWilliams}, 
 \item a polytope obtained from~$P$ by gliding its facets in the direction of their normal vectors without passing a vertex~\cite{Postnikov,PostnikovReinerWilliams}.
\end{enumerate*}
A sequence of deformations is illustrated in \cref{fig:deformations}.
The deformations of~$P$ form a polyhedral cone under dilation and Minkowski addition, called the \defn{deformation cone} of~$P$~\cite{Postnikov}.
The interior of the deformation cone of~$P$, called the \defn{type cone}~\cite{McMullen-typeCone}, contains those polytopes with the same normal fan as~$P$. 
When~$P$ is a rational polytope, it has an associated toric variety~\cite{CoxLittleSchenckToric}, and the type cone (here known as the \defn{numerically effective cone}, or shortly \defn{nef cone}) encodes its embeddings into projective space~\cite[Sect.~6.3]{CoxLittleSchenckToric}.
Among the different ways to parametrize and describe the deformation cone of a polytope~$P$ (see \eg \cite[App.~15]{PostnikovReinerWilliams}), we use the parametrization by the \defn{heights} corresponding to the facets of~$P$ and the description given by the \defn{wall-crossing inequalities} corresponding to the edges of~$P$~\cite{ChapotonFominZelevinsky}.
While this inequality description is immediately derived from the linear dependences among certain normal vectors of~$P$, it is in general more difficult to extract the irredundant facet inequality description of the deformation cone.

Fundamental examples of deformations of polytopes are the \defn{deformed permutahedra} (\aka generalized permutahedra or polymatroids) studied in~\cite{Edmonds, Postnikov, PostnikovReinerWilliams}, which are classically parametrized by \defn{submodular functions}.
Among the most famous deformed permutahedra are the classical \defn{associahedra} as constructed in~\cite{ShniderSternberg,Loday} (or even in~\cite{HohlwegLange}).
Associahedra appear in several mathematical contexts, from their original definition in topology~\cite{Stasheff} to some recent appearances in the theory of scattering amplitudes in mathematical physics~\cite{ArkaniHamedBaiHeYan}.

\begin{figure}
	\capstart
	\centerline{% to obtain each polytope:
% sage: attach("/Users/vinc/Documents/recherche/projets/codeAssociahedra/graphicalPolytopes.py") 
% sage: attach("/Users/vinc/Documents/recherche/projets/codeAssociahedra/hypergraphicalPolytopes.py") 
% sage: P = GraphicalPolytopes(graphs.PathGraph(4)).graphAssociahedron(essential=True)                                                                                                                      
% sage: Qx = HypergraphicalPolytopes([[1],[2],[3],[4],[1,4]]).hypergraphical_polytope_as_minkowski_sum(essential=True)                                                                                      
% sage: Qy = HypergraphicalPolytopes([[1],[2],[3],[4],[1,3],[2,4]]).hypergraphical_polytope_as_minkowski_sum(essential=True)                                                                                
% sage: Qxy = HypergraphicalPolytopes([[1],[2],[3],[4],[1,2,4],[1,3,4]]).hypergraphical_polytope_as_minkowski_sum(essential=True)                                                                           
% sage: for i in range(0,5): 
% ....:     x = min(i,2) 
% ....:     y = max(i,2)-2 
% ....:     f = open(f'deformation{i+1}.tex','w') 
% ....:     f.write(to_tikz(P.minkowski_sum(x/2*Qx).minkowski_sum(y/2*Qy).minkowski_sum((x/4+y/4)*Qxy), nb=True)) 
% ....:     f.close() 
% ....:                                                                                                                                                                                                     

\begin{tabular}{ccccc}
	\begin{tikzpicture}%
	[x={(-0.366215cm, -0.789554cm)},
	y={(0.235950cm, -0.590693cm)},
	z={(0.900119cm, -0.166391cm)},
	scale=.3,
	back/.style={very thin},
	edge/.style={color=blue, very thick},
	facet/.style={fill=blue,fill opacity=0},
	vertex/.style={inner sep=1pt,circle,fill=blue,thick},
	baseline=0]

%% Coordinate of the vertices:
%%
\coordinate (-3.66667, -3.77124, -3.26599) at (-3.66667, -3.77124, -3.26599);
\coordinate (5.66667, 2.35702, 0.81650) at (5.66667, 2.35702, 0.81650);
\coordinate (5.66667, 2.35702, -0.81650) at (5.66667, 2.35702, -0.81650);
\coordinate (-3.66667, -3.77124, 3.26599) at (-3.66667, -3.77124, 3.26599);
\coordinate (5.66667, -0.47140, 2.44949) at (5.66667, -0.47140, 2.44949);
\coordinate (5.66667, -0.47140, -2.44949) at (5.66667, -0.47140, -2.44949);
\coordinate (5.66667, -1.88562, 1.63299) at (5.66667, -1.88562, 1.63299);
\coordinate (5.66667, -1.88562, -1.63299) at (5.66667, -1.88562, -1.63299);
\coordinate (-3.66667, -0.94281, -4.89898) at (-3.66667, -0.94281, -4.89898);
\coordinate (0.33333, 6.12826, 0.81650) at (0.33333, 6.12826, 0.81650);
\coordinate (0.33333, 6.12826, -0.81650) at (0.33333, 6.12826, -0.81650);
\coordinate (-3.66667, -0.94281, 4.89898) at (-3.66667, -0.94281, 4.89898);
\coordinate (0.33333, -2.35702, 5.71548) at (0.33333, -2.35702, 5.71548);
\coordinate (0.33333, -2.35702, -5.71548) at (0.33333, -2.35702, -5.71548);
\coordinate (0.33333, -3.77124, 4.89898) at (0.33333, -3.77124, 4.89898);
\coordinate (0.33333, -3.77124, -4.89898) at (0.33333, -3.77124, -4.89898);
\coordinate (-2.33333, 5.18545, 2.44949) at (-2.33333, 5.18545, 2.44949);
\coordinate (-3.66667, 4.71405, -1.63299) at (-3.66667, 4.71405, -1.63299);
\coordinate (-2.33333, 5.18545, -2.44949) at (-2.33333, 5.18545, -2.44949);
\coordinate (-3.66667, 4.71405, 1.63299) at (-3.66667, 4.71405, 1.63299);
\coordinate (-2.33333, -4.71405, -3.26599) at (-2.33333, -4.71405, -3.26599);
\coordinate (-2.33333, -0.47140, 5.71548) at (-2.33333, -0.47140, 5.71548);
\coordinate (-2.33333, -0.47140, -5.71548) at (-2.33333, -0.47140, -5.71548);
\coordinate (-2.33333, -4.71405, 3.26599) at (-2.33333, -4.71405, 3.26599);
%%
%%
%% Drawing edges in the back
%%
\draw[edge,back] (-3.66667, -3.77124, -3.26599) -- (-3.66667, -0.94281, -4.89898);
\draw[edge,back] (5.66667, 2.35702, -0.81650) -- (0.33333, 6.12826, -0.81650);
\draw[edge,back] (-3.66667, -0.94281, -4.89898) -- (-3.66667, 4.71405, -1.63299);
\draw[edge,back] (-3.66667, -0.94281, -4.89898) -- (-2.33333, -0.47140, -5.71548);
\draw[edge,back] (0.33333, 6.12826, 0.81650) -- (0.33333, 6.12826, -0.81650);
\draw[edge,back] (0.33333, 6.12826, -0.81650) -- (-2.33333, 5.18545, -2.44949);
\draw[edge,back] (-3.66667, -0.94281, 4.89898) -- (-3.66667, 4.71405, 1.63299);
\draw[edge,back] (0.33333, -2.35702, -5.71548) -- (-2.33333, -0.47140, -5.71548);
\draw[edge,back] (-2.33333, 5.18545, 2.44949) -- (-3.66667, 4.71405, 1.63299);
\draw[edge,back] (-3.66667, 4.71405, -1.63299) -- (-2.33333, 5.18545, -2.44949);
\draw[edge,back] (-3.66667, 4.71405, -1.63299) -- (-3.66667, 4.71405, 1.63299);
\draw[edge,back] (-2.33333, 5.18545, -2.44949) -- (-2.33333, -0.47140, -5.71548);
%%
%%
%% Drawing vertices in the back
%%
\node[vertex] at (0.33333, 6.12826, -0.81650)     {};
\node[vertex] at (-2.33333, 5.18545, -2.44949)     {};
\node[vertex] at (-2.33333, -0.47140, -5.71548)     {};
\node[vertex] at (-3.66667, -0.94281, -4.89898)     {};
\node[vertex] at (-3.66667, 4.71405, -1.63299)     {};
\node[vertex] at (-3.66667, 4.71405, 1.63299)     {};
%%
%%
%% Drawing the facets
%%
\fill[facet] (0.33333, -3.77124, -4.89898) -- (5.66667, -1.88562, -1.63299) -- (5.66667, -0.47140, -2.44949) -- (0.33333, -2.35702, -5.71548) -- cycle {};
\fill[facet] (-2.33333, -4.71405, 3.26599) -- (-3.66667, -3.77124, 3.26599) -- (-3.66667, -3.77124, -3.26599) -- (-2.33333, -4.71405, -3.26599) -- cycle {};
\fill[facet] (0.33333, -2.35702, 5.71548) -- (0.33333, -3.77124, 4.89898) -- (-2.33333, -4.71405, 3.26599) -- (-3.66667, -3.77124, 3.26599) -- (-3.66667, -0.94281, 4.89898) -- (-2.33333, -0.47140, 5.71548) -- cycle {};
\fill[facet] (-2.33333, -4.71405, 3.26599) -- (0.33333, -3.77124, 4.89898) -- (5.66667, -1.88562, 1.63299) -- (5.66667, -1.88562, -1.63299) -- (0.33333, -3.77124, -4.89898) -- (-2.33333, -4.71405, -3.26599) -- cycle {};
\fill[facet] (0.33333, -3.77124, 4.89898) -- (5.66667, -1.88562, 1.63299) -- (5.66667, -0.47140, 2.44949) -- (0.33333, -2.35702, 5.71548) -- cycle {};
\fill[facet] (-2.33333, -0.47140, 5.71548) -- (0.33333, -2.35702, 5.71548) -- (5.66667, -0.47140, 2.44949) -- (5.66667, 2.35702, 0.81650) -- (0.33333, 6.12826, 0.81650) -- (-2.33333, 5.18545, 2.44949) -- cycle {};
\fill[facet] (5.66667, -1.88562, -1.63299) -- (5.66667, -0.47140, -2.44949) -- (5.66667, 2.35702, -0.81650) -- (5.66667, 2.35702, 0.81650) -- (5.66667, -0.47140, 2.44949) -- (5.66667, -1.88562, 1.63299) -- cycle {};
%%
%%
%% Drawing edges in the front
%%
\draw[edge] (-3.66667, -3.77124, -3.26599) -- (-3.66667, -3.77124, 3.26599);
\draw[edge] (-3.66667, -3.77124, -3.26599) -- (-2.33333, -4.71405, -3.26599);
\draw[edge] (5.66667, 2.35702, 0.81650) -- (5.66667, 2.35702, -0.81650);
\draw[edge] (5.66667, 2.35702, 0.81650) -- (5.66667, -0.47140, 2.44949);
\draw[edge] (5.66667, 2.35702, 0.81650) -- (0.33333, 6.12826, 0.81650);
\draw[edge] (5.66667, 2.35702, -0.81650) -- (5.66667, -0.47140, -2.44949);
\draw[edge] (-3.66667, -3.77124, 3.26599) -- (-3.66667, -0.94281, 4.89898);
\draw[edge] (-3.66667, -3.77124, 3.26599) -- (-2.33333, -4.71405, 3.26599);
\draw[edge] (5.66667, -0.47140, 2.44949) -- (5.66667, -1.88562, 1.63299);
\draw[edge] (5.66667, -0.47140, 2.44949) -- (0.33333, -2.35702, 5.71548);
\draw[edge] (5.66667, -0.47140, -2.44949) -- (5.66667, -1.88562, -1.63299);
\draw[edge] (5.66667, -0.47140, -2.44949) -- (0.33333, -2.35702, -5.71548);
\draw[edge] (5.66667, -1.88562, 1.63299) -- (5.66667, -1.88562, -1.63299);
\draw[edge] (5.66667, -1.88562, 1.63299) -- (0.33333, -3.77124, 4.89898);
\draw[edge] (5.66667, -1.88562, -1.63299) -- (0.33333, -3.77124, -4.89898);
\draw[edge] (0.33333, 6.12826, 0.81650) -- (-2.33333, 5.18545, 2.44949);
\draw[edge] (-3.66667, -0.94281, 4.89898) -- (-2.33333, -0.47140, 5.71548);
\draw[edge] (0.33333, -2.35702, 5.71548) -- (0.33333, -3.77124, 4.89898);
\draw[edge] (0.33333, -2.35702, 5.71548) -- (-2.33333, -0.47140, 5.71548);
\draw[edge] (0.33333, -2.35702, -5.71548) -- (0.33333, -3.77124, -4.89898);
\draw[edge] (0.33333, -3.77124, 4.89898) -- (-2.33333, -4.71405, 3.26599);
\draw[edge] (0.33333, -3.77124, -4.89898) -- (-2.33333, -4.71405, -3.26599);
\draw[edge] (-2.33333, 5.18545, 2.44949) -- (-2.33333, -0.47140, 5.71548);
\draw[edge] (-2.33333, -4.71405, -3.26599) -- (-2.33333, -4.71405, 3.26599);
%%
%%
%% Drawing the vertices in the front
%%
\node[vertex] at (-3.66667, -3.77124, -3.26599)     {};
\node[vertex] at (5.66667, 2.35702, 0.81650)     {};
\node[vertex] at (5.66667, 2.35702, -0.81650)     {};
\node[vertex] at (-3.66667, -3.77124, 3.26599)     {};
\node[vertex] at (5.66667, -0.47140, 2.44949)     {};
\node[vertex] at (5.66667, -0.47140, -2.44949)     {};
\node[vertex] at (5.66667, -1.88562, 1.63299)     {};
\node[vertex] at (5.66667, -1.88562, -1.63299)     {};
\node[vertex] at (0.33333, 6.12826, 0.81650)     {};
\node[vertex] at (-3.66667, -0.94281, 4.89898)     {};
\node[vertex] at (0.33333, -2.35702, 5.71548)     {};
\node[vertex] at (0.33333, -2.35702, -5.71548)     {};
\node[vertex] at (0.33333, -3.77124, 4.89898)     {};
\node[vertex] at (0.33333, -3.77124, -4.89898)     {};
\node[vertex] at (-2.33333, 5.18545, 2.44949)     {};
\node[vertex] at (-2.33333, -4.71405, -3.26599)     {};
\node[vertex] at (-2.33333, -0.47140, 5.71548)     {};
\node[vertex] at (-2.33333, -4.71405, 3.26599)     {};
\end{tikzpicture} &
	\begin{tikzpicture}%
	[x={(-0.366215cm, -0.789554cm)},
	y={(0.235950cm, -0.590693cm)},
	z={(0.900119cm, -0.166391cm)},
	scale=.3,
	back/.style={very thin},
	edge/.style={color=blue, very thick},
	facet/.style={fill=blue,fill opacity=0},
	vertex/.style={inner sep=1pt,circle,fill=blue,thick},
	baseline=0]

%% Coordinate of the vertices:
%%
\coordinate (-3.16667, -3.06413, -2.85774) at (-3.16667, -3.06413, -2.85774);
\coordinate (4.83333, 1.88562, 0.81650) at (4.83333, 1.88562, 0.81650);
\coordinate (4.83333, 1.88562, -0.81650) at (4.83333, 1.88562, -0.81650);
\coordinate (-3.16667, -3.06413, 3.26599) at (-3.16667, -3.06413, 3.26599);
\coordinate (4.83333, -0.23570, -2.04124) at (4.83333, -0.23570, -2.04124);
\coordinate (4.83333, -0.94281, 2.44949) at (4.83333, -0.94281, 2.44949);
\coordinate (4.83333, -1.64992, 2.04124) at (4.83333, -1.64992, 2.04124);
\coordinate (4.83333, -1.64992, -1.22474) at (4.83333, -1.64992, -1.22474);
\coordinate (0.50000, -2.47487, 5.10310) at (0.50000, -2.47487, 5.10310);
\coordinate (0.50000, -3.18198, 4.69486) at (0.50000, -3.18198, 4.69486);
\coordinate (0.16667, -1.88562, -4.89898) at (0.16667, -1.88562, -4.89898);
\coordinate (-3.16667, -1.29636, 4.28661) at (-3.16667, -1.29636, 4.28661);
\coordinate (-3.16667, -0.23570, -4.49073) at (-3.16667, -0.23570, -4.49073);
\coordinate (0.16667, -3.29983, -4.08248) at (0.16667, -3.29983, -4.08248);
\coordinate (-0.16667, 5.42115, 0.81650) at (-0.16667, 5.42115, 0.81650);
\coordinate (-0.16667, 5.42115, -0.81650) at (-0.16667, 5.42115, -0.81650);
\coordinate (-1.83333, 4.83190, 1.83712) at (-1.83333, 4.83190, 1.83712);
\coordinate (-1.83333, -0.82496, 5.10310) at (-1.83333, -0.82496, 5.10310);
\coordinate (-1.83333, -4.00694, 3.26599) at (-1.83333, -4.00694, 3.26599);
\coordinate (-1.83333, -4.00694, -2.85774) at (-1.83333, -4.00694, -2.85774);
\coordinate (-3.16667, 4.36049, -1.83712) at (-3.16667, 4.36049, -1.83712);
\coordinate (-2.50000, 4.59619, -2.24537) at (-2.50000, 4.59619, -2.24537);
\coordinate (-2.50000, 0.00000, -4.89898) at (-2.50000, 0.00000, -4.89898);
\coordinate (-3.16667, 4.36049, 1.02062) at (-3.16667, 4.36049, 1.02062);
%%
%%
%% Drawing edges in the back
%%
\draw[edge,back] (-3.16667, -3.06413, -2.85774) -- (-3.16667, -0.23570, -4.49073);
\draw[edge,back] (4.83333, 1.88562, -0.81650) -- (-0.16667, 5.42115, -0.81650);
\draw[edge,back] (0.16667, -1.88562, -4.89898) -- (-2.50000, 0.00000, -4.89898);
\draw[edge,back] (-3.16667, -1.29636, 4.28661) -- (-3.16667, 4.36049, 1.02062);
\draw[edge,back] (-3.16667, -0.23570, -4.49073) -- (-3.16667, 4.36049, -1.83712);
\draw[edge,back] (-3.16667, -0.23570, -4.49073) -- (-2.50000, 0.00000, -4.89898);
\draw[edge,back] (-0.16667, 5.42115, 0.81650) -- (-0.16667, 5.42115, -0.81650);
\draw[edge,back] (-0.16667, 5.42115, -0.81650) -- (-2.50000, 4.59619, -2.24537);
\draw[edge,back] (-1.83333, 4.83190, 1.83712) -- (-3.16667, 4.36049, 1.02062);
\draw[edge,back] (-3.16667, 4.36049, -1.83712) -- (-2.50000, 4.59619, -2.24537);
\draw[edge,back] (-3.16667, 4.36049, -1.83712) -- (-3.16667, 4.36049, 1.02062);
\draw[edge,back] (-2.50000, 4.59619, -2.24537) -- (-2.50000, 0.00000, -4.89898);
%%
%%
%% Drawing vertices in the back
%%
\node[vertex] at (-0.16667, 5.42115, -0.81650)     {};
\node[vertex] at (-2.50000, 4.59619, -2.24537)     {};
\node[vertex] at (-2.50000, 0.00000, -4.89898)     {};
\node[vertex] at (-3.16667, -0.23570, -4.49073)     {};
\node[vertex] at (-3.16667, 4.36049, -1.83712)     {};
\node[vertex] at (-3.16667, 4.36049, 1.02062)     {};
%%
%%
%% Drawing the facets
%%
\fill[facet] (0.50000, -3.18198, 4.69486) -- (4.83333, -1.64992, 2.04124) -- (4.83333, -0.94281, 2.44949) -- (0.50000, -2.47487, 5.10310) -- cycle {};
\fill[facet] (-1.83333, -4.00694, -2.85774) -- (-3.16667, -3.06413, -2.85774) -- (-3.16667, -3.06413, 3.26599) -- (-1.83333, -4.00694, 3.26599) -- cycle {};
\fill[facet] (0.50000, -2.47487, 5.10310) -- (0.50000, -3.18198, 4.69486) -- (-1.83333, -4.00694, 3.26599) -- (-3.16667, -3.06413, 3.26599) -- (-3.16667, -1.29636, 4.28661) -- (-1.83333, -0.82496, 5.10310) -- cycle {};
\fill[facet] (-1.83333, -4.00694, -2.85774) -- (0.16667, -3.29983, -4.08248) -- (4.83333, -1.64992, -1.22474) -- (4.83333, -1.64992, 2.04124) -- (0.50000, -3.18198, 4.69486) -- (-1.83333, -4.00694, 3.26599) -- cycle {};
\fill[facet] (0.16667, -3.29983, -4.08248) -- (4.83333, -1.64992, -1.22474) -- (4.83333, -0.23570, -2.04124) -- (0.16667, -1.88562, -4.89898) -- cycle {};
\fill[facet] (-1.83333, -0.82496, 5.10310) -- (0.50000, -2.47487, 5.10310) -- (4.83333, -0.94281, 2.44949) -- (4.83333, 1.88562, 0.81650) -- (-0.16667, 5.42115, 0.81650) -- (-1.83333, 4.83190, 1.83712) -- cycle {};
\fill[facet] (4.83333, -1.64992, -1.22474) -- (4.83333, -0.23570, -2.04124) -- (4.83333, 1.88562, -0.81650) -- (4.83333, 1.88562, 0.81650) -- (4.83333, -0.94281, 2.44949) -- (4.83333, -1.64992, 2.04124) -- cycle {};
%%
%%
%% Drawing edges in the front
%%
\draw[edge] (-3.16667, -3.06413, -2.85774) -- (-3.16667, -3.06413, 3.26599);
\draw[edge] (-3.16667, -3.06413, -2.85774) -- (-1.83333, -4.00694, -2.85774);
\draw[edge] (4.83333, 1.88562, 0.81650) -- (4.83333, 1.88562, -0.81650);
\draw[edge] (4.83333, 1.88562, 0.81650) -- (4.83333, -0.94281, 2.44949);
\draw[edge] (4.83333, 1.88562, 0.81650) -- (-0.16667, 5.42115, 0.81650);
\draw[edge] (4.83333, 1.88562, -0.81650) -- (4.83333, -0.23570, -2.04124);
\draw[edge] (-3.16667, -3.06413, 3.26599) -- (-3.16667, -1.29636, 4.28661);
\draw[edge] (-3.16667, -3.06413, 3.26599) -- (-1.83333, -4.00694, 3.26599);
\draw[edge] (4.83333, -0.23570, -2.04124) -- (4.83333, -1.64992, -1.22474);
\draw[edge] (4.83333, -0.23570, -2.04124) -- (0.16667, -1.88562, -4.89898);
\draw[edge] (4.83333, -0.94281, 2.44949) -- (4.83333, -1.64992, 2.04124);
\draw[edge] (4.83333, -0.94281, 2.44949) -- (0.50000, -2.47487, 5.10310);
\draw[edge] (4.83333, -1.64992, 2.04124) -- (4.83333, -1.64992, -1.22474);
\draw[edge] (4.83333, -1.64992, 2.04124) -- (0.50000, -3.18198, 4.69486);
\draw[edge] (4.83333, -1.64992, -1.22474) -- (0.16667, -3.29983, -4.08248);
\draw[edge] (0.50000, -2.47487, 5.10310) -- (0.50000, -3.18198, 4.69486);
\draw[edge] (0.50000, -2.47487, 5.10310) -- (-1.83333, -0.82496, 5.10310);
\draw[edge] (0.50000, -3.18198, 4.69486) -- (-1.83333, -4.00694, 3.26599);
\draw[edge] (0.16667, -1.88562, -4.89898) -- (0.16667, -3.29983, -4.08248);
\draw[edge] (-3.16667, -1.29636, 4.28661) -- (-1.83333, -0.82496, 5.10310);
\draw[edge] (0.16667, -3.29983, -4.08248) -- (-1.83333, -4.00694, -2.85774);
\draw[edge] (-0.16667, 5.42115, 0.81650) -- (-1.83333, 4.83190, 1.83712);
\draw[edge] (-1.83333, 4.83190, 1.83712) -- (-1.83333, -0.82496, 5.10310);
\draw[edge] (-1.83333, -4.00694, 3.26599) -- (-1.83333, -4.00694, -2.85774);
%%
%%
%% Drawing the vertices in the front
%%
\node[vertex] at (-3.16667, -3.06413, -2.85774)     {};
\node[vertex] at (4.83333, 1.88562, 0.81650)     {};
\node[vertex] at (4.83333, 1.88562, -0.81650)     {};
\node[vertex] at (-3.16667, -3.06413, 3.26599)     {};
\node[vertex] at (4.83333, -0.23570, -2.04124)     {};
\node[vertex] at (4.83333, -0.94281, 2.44949)     {};
\node[vertex] at (4.83333, -1.64992, 2.04124)     {};
\node[vertex] at (4.83333, -1.64992, -1.22474)     {};
\node[vertex] at (0.50000, -2.47487, 5.10310)     {};
\node[vertex] at (0.50000, -3.18198, 4.69486)     {};
\node[vertex] at (0.16667, -1.88562, -4.89898)     {};
\node[vertex] at (-3.16667, -1.29636, 4.28661)     {};
\node[vertex] at (0.16667, -3.29983, -4.08248)     {};
\node[vertex] at (-0.16667, 5.42115, 0.81650)     {};
\node[vertex] at (-1.83333, 4.83190, 1.83712)     {};
\node[vertex] at (-1.83333, -0.82496, 5.10310)     {};
\node[vertex] at (-1.83333, -4.00694, 3.26599)     {};
\node[vertex] at (-1.83333, -4.00694, -2.85774)     {};
\end{tikzpicture} &
	\begin{tikzpicture}%
	[x={(-0.366215cm, -0.789554cm)},
	y={(0.235950cm, -0.590693cm)},
	z={(0.900119cm, -0.166391cm)},
	scale=.3,
	back/.style={very thin},
	edge/.style={color=blue, very thick},
	facet/.style={fill=blue,fill opacity=0},
	vertex/.style={inner sep=1pt,circle,fill=blue,thick},
	baseline=0]

%% Coordinate of the vertices:
%%
\coordinate (4.00000, 1.41421, 0.81650) at (4.00000, 1.41421, 0.81650);
\coordinate (-2.66667, -2.35702, 3.26599) at (-2.66667, -2.35702, 3.26599);
\coordinate (4.00000, 1.41421, -0.81650) at (4.00000, 1.41421, -0.81650);
\coordinate (4.00000, 0.00000, -1.63299) at (4.00000, 0.00000, -1.63299);
\coordinate (4.00000, -1.41421, 2.44949) at (4.00000, -1.41421, 2.44949);
\coordinate (-2.66667, -1.64992, 3.67423) at (-2.66667, -1.64992, 3.67423);
\coordinate (4.00000, -1.41421, -0.81650) at (4.00000, -1.41421, -0.81650);
\coordinate (0.66667, -2.59272, 4.49073) at (0.66667, -2.59272, 4.49073);
\coordinate (0.00000, -1.41421, -4.08248) at (0.00000, -1.41421, -4.08248);
\coordinate (0.00000, -2.82843, -3.26599) at (0.00000, -2.82843, -3.26599);
\coordinate (-0.66667, 4.71405, 0.81650) at (-0.66667, 4.71405, 0.81650);
\coordinate (-2.66667, 4.00694, -2.04124) at (-2.66667, 4.00694, -2.04124);
\coordinate (-0.66667, 4.71405, -0.81650) at (-0.66667, 4.71405, -0.81650);
\coordinate (-1.33333, -3.29983, -2.44949) at (-1.33333, -3.29983, -2.44949);
\coordinate (-2.66667, 4.00694, 0.40825) at (-2.66667, 4.00694, 0.40825);
\coordinate (-1.33333, 4.47834, 1.22474) at (-1.33333, 4.47834, 1.22474);
\coordinate (-2.66667, 0.47140, -4.08248) at (-2.66667, 0.47140, -4.08248);
\coordinate (-1.33333, -1.17851, 4.49073) at (-1.33333, -1.17851, 4.49073);
\coordinate (-2.66667, -2.35702, -2.44949) at (-2.66667, -2.35702, -2.44949);
\coordinate (-1.33333, -3.29983, 3.26599) at (-1.33333, -3.29983, 3.26599);
%%
%%
%% Drawing edges in the back
%%
\draw[edge,back] (4.00000, 1.41421, -0.81650) -- (-0.66667, 4.71405, -0.81650);
\draw[edge,back] (-2.66667, -1.64992, 3.67423) -- (-2.66667, 4.00694, 0.40825);
\draw[edge,back] (0.00000, -1.41421, -4.08248) -- (-2.66667, 0.47140, -4.08248);
\draw[edge,back] (-0.66667, 4.71405, 0.81650) -- (-0.66667, 4.71405, -0.81650);
\draw[edge,back] (-2.66667, 4.00694, -2.04124) -- (-0.66667, 4.71405, -0.81650);
\draw[edge,back] (-2.66667, 4.00694, -2.04124) -- (-2.66667, 4.00694, 0.40825);
\draw[edge,back] (-2.66667, 4.00694, -2.04124) -- (-2.66667, 0.47140, -4.08248);
\draw[edge,back] (-2.66667, 4.00694, 0.40825) -- (-1.33333, 4.47834, 1.22474);
\draw[edge,back] (-2.66667, 0.47140, -4.08248) -- (-2.66667, -2.35702, -2.44949);
%%
%%
%% Drawing vertices in the back
%%
\node[vertex] at (-2.66667, 4.00694, -2.04124)     {};
\node[vertex] at (-0.66667, 4.71405, -0.81650)     {};
\node[vertex] at (-2.66667, 0.47140, -4.08248)     {};
\node[vertex] at (-2.66667, 4.00694, 0.40825)     {};
%%
%%
%% Drawing the facets
%%
\fill[facet] (-1.33333, -3.29983, 3.26599) -- (0.66667, -2.59272, 4.49073) -- (4.00000, -1.41421, 2.44949) -- (4.00000, -1.41421, -0.81650) -- (0.00000, -2.82843, -3.26599) -- (-1.33333, -3.29983, -2.44949) -- cycle {};
\fill[facet] (0.66667, -2.59272, 4.49073) -- (-1.33333, -3.29983, 3.26599) -- (-2.66667, -2.35702, 3.26599) -- (-2.66667, -1.64992, 3.67423) -- (-1.33333, -1.17851, 4.49073) -- cycle {};
\fill[facet] (-1.33333, -3.29983, -2.44949) -- (-1.33333, -3.29983, 3.26599) -- (-2.66667, -2.35702, 3.26599) -- (-2.66667, -2.35702, -2.44949) -- cycle {};
\fill[facet] (0.00000, -2.82843, -3.26599) -- (4.00000, -1.41421, -0.81650) -- (4.00000, 0.00000, -1.63299) -- (0.00000, -1.41421, -4.08248) -- cycle {};
\fill[facet] (-1.33333, -1.17851, 4.49073) -- (0.66667, -2.59272, 4.49073) -- (4.00000, -1.41421, 2.44949) -- (4.00000, 1.41421, 0.81650) -- (-0.66667, 4.71405, 0.81650) -- (-1.33333, 4.47834, 1.22474) -- cycle {};
\fill[facet] (4.00000, -1.41421, -0.81650) -- (4.00000, 0.00000, -1.63299) -- (4.00000, 1.41421, -0.81650) -- (4.00000, 1.41421, 0.81650) -- (4.00000, -1.41421, 2.44949) -- cycle {};
%%
%%
%% Drawing edges in the front
%%
\draw[edge] (4.00000, 1.41421, 0.81650) -- (4.00000, 1.41421, -0.81650);
\draw[edge] (4.00000, 1.41421, 0.81650) -- (4.00000, -1.41421, 2.44949);
\draw[edge] (4.00000, 1.41421, 0.81650) -- (-0.66667, 4.71405, 0.81650);
\draw[edge] (-2.66667, -2.35702, 3.26599) -- (-2.66667, -1.64992, 3.67423);
\draw[edge] (-2.66667, -2.35702, 3.26599) -- (-2.66667, -2.35702, -2.44949);
\draw[edge] (-2.66667, -2.35702, 3.26599) -- (-1.33333, -3.29983, 3.26599);
\draw[edge] (4.00000, 1.41421, -0.81650) -- (4.00000, 0.00000, -1.63299);
\draw[edge] (4.00000, 0.00000, -1.63299) -- (4.00000, -1.41421, -0.81650);
\draw[edge] (4.00000, 0.00000, -1.63299) -- (0.00000, -1.41421, -4.08248);
\draw[edge] (4.00000, -1.41421, 2.44949) -- (4.00000, -1.41421, -0.81650);
\draw[edge] (4.00000, -1.41421, 2.44949) -- (0.66667, -2.59272, 4.49073);
\draw[edge] (-2.66667, -1.64992, 3.67423) -- (-1.33333, -1.17851, 4.49073);
\draw[edge] (4.00000, -1.41421, -0.81650) -- (0.00000, -2.82843, -3.26599);
\draw[edge] (0.66667, -2.59272, 4.49073) -- (-1.33333, -1.17851, 4.49073);
\draw[edge] (0.66667, -2.59272, 4.49073) -- (-1.33333, -3.29983, 3.26599);
\draw[edge] (0.00000, -1.41421, -4.08248) -- (0.00000, -2.82843, -3.26599);
\draw[edge] (0.00000, -2.82843, -3.26599) -- (-1.33333, -3.29983, -2.44949);
\draw[edge] (-0.66667, 4.71405, 0.81650) -- (-1.33333, 4.47834, 1.22474);
\draw[edge] (-1.33333, -3.29983, -2.44949) -- (-2.66667, -2.35702, -2.44949);
\draw[edge] (-1.33333, -3.29983, -2.44949) -- (-1.33333, -3.29983, 3.26599);
\draw[edge] (-1.33333, 4.47834, 1.22474) -- (-1.33333, -1.17851, 4.49073);
%%
%%
%% Drawing the vertices in the front
%%
\node[vertex] at (4.00000, 1.41421, 0.81650)     {};
\node[vertex] at (-2.66667, -2.35702, 3.26599)     {};
\node[vertex] at (4.00000, 1.41421, -0.81650)     {};
\node[vertex] at (4.00000, 0.00000, -1.63299)     {};
\node[vertex] at (4.00000, -1.41421, 2.44949)     {};
\node[vertex] at (-2.66667, -1.64992, 3.67423)     {};
\node[vertex] at (4.00000, -1.41421, -0.81650)     {};
\node[vertex] at (0.66667, -2.59272, 4.49073)     {};
\node[vertex] at (0.00000, -1.41421, -4.08248)     {};
\node[vertex] at (0.00000, -2.82843, -3.26599)     {};
\node[vertex] at (-0.66667, 4.71405, 0.81650)     {};
\node[vertex] at (-1.33333, -3.29983, -2.44949)     {};
\node[vertex] at (-1.33333, 4.47834, 1.22474)     {};
\node[vertex] at (-1.33333, -1.17851, 4.49073)     {};
\node[vertex] at (-2.66667, -2.35702, -2.44949)     {};
\node[vertex] at (-1.33333, -3.29983, 3.26599)     {};
\end{tikzpicture} &
	\begin{tikzpicture}%
	[x={(-0.366215cm, -0.789554cm)},
	y={(0.235950cm, -0.590693cm)},
	z={(0.900119cm, -0.166391cm)},
	scale=.3,
	back/.style={very thin},
	edge/.style={color=blue, very thick},
	facet/.style={fill=blue,fill opacity=0},
	vertex/.style={inner sep=1pt,circle,fill=blue,thick},
	baseline=0]

%% Coordinate of the vertices:
%%
\coordinate (-2.33333, -1.88562, -1.63299) at (-2.33333, -1.88562, -1.63299);
\coordinate (3.00000, 1.41421, 0.81650) at (3.00000, 1.41421, 0.81650);
\coordinate (3.00000, 1.41421, -0.81650) at (3.00000, 1.41421, -0.81650);
\coordinate (-2.33333, -1.88562, 3.67423) at (-2.33333, -1.88562, 3.67423);
\coordinate (3.00000, 0.00000, -1.63299) at (3.00000, 0.00000, -1.63299);
\coordinate (3.00000, -1.41421, 2.44949) at (3.00000, -1.41421, 2.44949);
\coordinate (3.00000, -1.41421, -0.81650) at (3.00000, -1.41421, -0.81650);
\coordinate (-2.33333, -1.53206, 3.87836) at (-2.33333, -1.53206, 3.87836);
\coordinate (-2.33333, 0.94281, -3.26599) at (-2.33333, 0.94281, -3.26599);
\coordinate (0.33333, -0.94281, -3.26599) at (0.33333, -0.94281, -3.26599);
\coordinate (0.33333, -2.35702, -2.44949) at (0.33333, -2.35702, -2.44949);
\coordinate (0.00000, -2.47487, 4.28661) at (0.00000, -2.47487, 4.28661);
\coordinate (-1.00000, -2.82843, 3.67423) at (-1.00000, -2.82843, 3.67423);
\coordinate (-1.00000, -2.82843, -1.63299) at (-1.00000, -2.82843, -1.63299);
\coordinate (-1.33333, 4.47834, 0.81650) at (-1.33333, 4.47834, 0.81650);
\coordinate (-1.33333, 4.47834, -0.81650) at (-1.33333, 4.47834, -0.81650);
\coordinate (-2.33333, 4.12479, -1.42887) at (-2.33333, 4.12479, -1.42887);
\coordinate (-1.66667, 4.36049, 1.02062) at (-1.66667, 4.36049, 1.02062);
\coordinate (-1.66667, -1.29636, 4.28661) at (-1.66667, -1.29636, 4.28661);
\coordinate (-2.33333, 4.12479, 0.61237) at (-2.33333, 4.12479, 0.61237);
%%
%%
%% Drawing edges in the back
%%
\draw[edge,back] (-2.33333, -1.88562, -1.63299) -- (-2.33333, 0.94281, -3.26599);
\draw[edge,back] (3.00000, 1.41421, -0.81650) -- (-1.33333, 4.47834, -0.81650);
\draw[edge,back] (-2.33333, -1.53206, 3.87836) -- (-2.33333, 4.12479, 0.61237);
\draw[edge,back] (-2.33333, 0.94281, -3.26599) -- (0.33333, -0.94281, -3.26599);
\draw[edge,back] (-2.33333, 0.94281, -3.26599) -- (-2.33333, 4.12479, -1.42887);
\draw[edge,back] (-1.33333, 4.47834, 0.81650) -- (-1.33333, 4.47834, -0.81650);
\draw[edge,back] (-1.33333, 4.47834, -0.81650) -- (-2.33333, 4.12479, -1.42887);
\draw[edge,back] (-2.33333, 4.12479, -1.42887) -- (-2.33333, 4.12479, 0.61237);
\draw[edge,back] (-1.66667, 4.36049, 1.02062) -- (-2.33333, 4.12479, 0.61237);
%%
%%
%% Drawing vertices in the back
%%
\node[vertex] at (-2.33333, 0.94281, -3.26599)     {};
\node[vertex] at (-1.33333, 4.47834, -0.81650)     {};
\node[vertex] at (-2.33333, 4.12479, -1.42887)     {};
\node[vertex] at (-2.33333, 4.12479, 0.61237)     {};
%%
%%
%% Drawing the facets
%%
\fill[facet] (-1.00000, -2.82843, -1.63299) -- (0.33333, -2.35702, -2.44949) -- (3.00000, -1.41421, -0.81650) -- (3.00000, -1.41421, 2.44949) -- (0.00000, -2.47487, 4.28661) -- (-1.00000, -2.82843, 3.67423) -- cycle {};
\fill[facet] (0.00000, -2.47487, 4.28661) -- (-1.66667, -1.29636, 4.28661) -- (-2.33333, -1.53206, 3.87836) -- (-2.33333, -1.88562, 3.67423) -- (-1.00000, -2.82843, 3.67423) -- cycle {};
\fill[facet] (-1.00000, -2.82843, -1.63299) -- (-2.33333, -1.88562, -1.63299) -- (-2.33333, -1.88562, 3.67423) -- (-1.00000, -2.82843, 3.67423) -- cycle {};
\fill[facet] (0.33333, -2.35702, -2.44949) -- (3.00000, -1.41421, -0.81650) -- (3.00000, 0.00000, -1.63299) -- (0.33333, -0.94281, -3.26599) -- cycle {};
\fill[facet] (-1.66667, -1.29636, 4.28661) -- (0.00000, -2.47487, 4.28661) -- (3.00000, -1.41421, 2.44949) -- (3.00000, 1.41421, 0.81650) -- (-1.33333, 4.47834, 0.81650) -- (-1.66667, 4.36049, 1.02062) -- cycle {};
\fill[facet] (3.00000, -1.41421, -0.81650) -- (3.00000, 0.00000, -1.63299) -- (3.00000, 1.41421, -0.81650) -- (3.00000, 1.41421, 0.81650) -- (3.00000, -1.41421, 2.44949) -- cycle {};
%%
%%
%% Drawing edges in the front
%%
\draw[edge] (-2.33333, -1.88562, -1.63299) -- (-2.33333, -1.88562, 3.67423);
\draw[edge] (-2.33333, -1.88562, -1.63299) -- (-1.00000, -2.82843, -1.63299);
\draw[edge] (3.00000, 1.41421, 0.81650) -- (3.00000, 1.41421, -0.81650);
\draw[edge] (3.00000, 1.41421, 0.81650) -- (3.00000, -1.41421, 2.44949);
\draw[edge] (3.00000, 1.41421, 0.81650) -- (-1.33333, 4.47834, 0.81650);
\draw[edge] (3.00000, 1.41421, -0.81650) -- (3.00000, 0.00000, -1.63299);
\draw[edge] (-2.33333, -1.88562, 3.67423) -- (-2.33333, -1.53206, 3.87836);
\draw[edge] (-2.33333, -1.88562, 3.67423) -- (-1.00000, -2.82843, 3.67423);
\draw[edge] (3.00000, 0.00000, -1.63299) -- (3.00000, -1.41421, -0.81650);
\draw[edge] (3.00000, 0.00000, -1.63299) -- (0.33333, -0.94281, -3.26599);
\draw[edge] (3.00000, -1.41421, 2.44949) -- (3.00000, -1.41421, -0.81650);
\draw[edge] (3.00000, -1.41421, 2.44949) -- (0.00000, -2.47487, 4.28661);
\draw[edge] (3.00000, -1.41421, -0.81650) -- (0.33333, -2.35702, -2.44949);
\draw[edge] (-2.33333, -1.53206, 3.87836) -- (-1.66667, -1.29636, 4.28661);
\draw[edge] (0.33333, -0.94281, -3.26599) -- (0.33333, -2.35702, -2.44949);
\draw[edge] (0.33333, -2.35702, -2.44949) -- (-1.00000, -2.82843, -1.63299);
\draw[edge] (0.00000, -2.47487, 4.28661) -- (-1.00000, -2.82843, 3.67423);
\draw[edge] (0.00000, -2.47487, 4.28661) -- (-1.66667, -1.29636, 4.28661);
\draw[edge] (-1.00000, -2.82843, 3.67423) -- (-1.00000, -2.82843, -1.63299);
\draw[edge] (-1.33333, 4.47834, 0.81650) -- (-1.66667, 4.36049, 1.02062);
\draw[edge] (-1.66667, 4.36049, 1.02062) -- (-1.66667, -1.29636, 4.28661);
%%
%%
%% Drawing the vertices in the front
%%
\node[vertex] at (-2.33333, -1.88562, -1.63299)     {};
\node[vertex] at (3.00000, 1.41421, 0.81650)     {};
\node[vertex] at (3.00000, 1.41421, -0.81650)     {};
\node[vertex] at (-2.33333, -1.88562, 3.67423)     {};
\node[vertex] at (3.00000, 0.00000, -1.63299)     {};
\node[vertex] at (3.00000, -1.41421, 2.44949)     {};
\node[vertex] at (3.00000, -1.41421, -0.81650)     {};
\node[vertex] at (-2.33333, -1.53206, 3.87836)     {};
\node[vertex] at (0.33333, -0.94281, -3.26599)     {};
\node[vertex] at (0.33333, -2.35702, -2.44949)     {};
\node[vertex] at (0.00000, -2.47487, 4.28661)     {};
\node[vertex] at (-1.00000, -2.82843, 3.67423)     {};
\node[vertex] at (-1.00000, -2.82843, -1.63299)     {};
\node[vertex] at (-1.33333, 4.47834, 0.81650)     {};
\node[vertex] at (-1.66667, 4.36049, 1.02062)     {};
\node[vertex] at (-1.66667, -1.29636, 4.28661)     {};
\end{tikzpicture} &
	\begin{tikzpicture}%
	[x={(-0.366215cm, -0.789554cm)},
	y={(0.235950cm, -0.590693cm)},
	z={(0.900119cm, -0.166391cm)},
	scale=.3,
	back/.style={very thin},
	edge/.style={color=blue, very thick},
	facet/.style={fill=blue,fill opacity=0},
	vertex/.style={inner sep=1pt,circle,fill=blue,thick},
	baseline=0]

%% Coordinate of the vertices:
%%
\coordinate (-2.00000, -1.41421, -0.81650) at (-2.00000, -1.41421, -0.81650);
\coordinate (-2.00000, -1.41421, 4.08248) at (-2.00000, -1.41421, 4.08248);
\coordinate (-2.00000, 1.41421, -2.44949) at (-2.00000, 1.41421, -2.44949);
\coordinate (-2.00000, 4.24264, -0.81650) at (-2.00000, 4.24264, -0.81650);
\coordinate (-2.00000, 4.24264, 0.81650) at (-2.00000, 4.24264, 0.81650);
\coordinate (-0.66667, -2.35702, -0.81650) at (-0.66667, -2.35702, -0.81650);
\coordinate (-0.66667, -2.35702, 4.08248) at (-0.66667, -2.35702, 4.08248);
\coordinate (0.66667, -1.88562, -1.63299) at (0.66667, -1.88562, -1.63299);
\coordinate (0.66667, -0.47140, -2.44949) at (0.66667, -0.47140, -2.44949);
\coordinate (2.00000, -1.41421, -0.81650) at (2.00000, -1.41421, -0.81650);
\coordinate (2.00000, -1.41421, 2.44949) at (2.00000, -1.41421, 2.44949);
\coordinate (2.00000, 0.00000, -1.63299) at (2.00000, 0.00000, -1.63299);
\coordinate (2.00000, 1.41421, -0.81650) at (2.00000, 1.41421, -0.81650);
\coordinate (2.00000, 1.41421, 0.81650) at (2.00000, 1.41421, 0.81650);
%%
%%
%% Drawing edges in the back
%%
\draw[edge,back] (-2.00000, -1.41421, -0.81650) -- (-2.00000, 1.41421, -2.44949);
\draw[edge,back] (-2.00000, 1.41421, -2.44949) -- (-2.00000, 4.24264, -0.81650);
\draw[edge,back] (-2.00000, 1.41421, -2.44949) -- (0.66667, -0.47140, -2.44949);
\draw[edge,back] (-2.00000, 4.24264, -0.81650) -- (-2.00000, 4.24264, 0.81650);
\draw[edge,back] (-2.00000, 4.24264, -0.81650) -- (2.00000, 1.41421, -0.81650);
%%
%%
%% Drawing vertices in the back
%%
\node[vertex] at (-2.00000, 1.41421, -2.44949)     {};
\node[vertex] at (-2.00000, 4.24264, -0.81650)     {};
%%
%%
%% Drawing the facets
%%
\fill[facet] (2.00000, -1.41421, 2.44949) -- (-0.66667, -2.35702, 4.08248) -- (-0.66667, -2.35702, -0.81650) -- (0.66667, -1.88562, -1.63299) -- (2.00000, -1.41421, -0.81650) -- cycle {};
\fill[facet] (-0.66667, -2.35702, 4.08248) -- (-2.00000, -1.41421, 4.08248) -- (-2.00000, -1.41421, -0.81650) -- (-0.66667, -2.35702, -0.81650) -- cycle {};
\fill[facet] (2.00000, 0.00000, -1.63299) -- (0.66667, -0.47140, -2.44949) -- (0.66667, -1.88562, -1.63299) -- (2.00000, -1.41421, -0.81650) -- cycle {};
\fill[facet] (2.00000, 1.41421, 0.81650) -- (-2.00000, 4.24264, 0.81650) -- (-2.00000, -1.41421, 4.08248) -- (-0.66667, -2.35702, 4.08248) -- (2.00000, -1.41421, 2.44949) -- cycle {};
\fill[facet] (2.00000, 1.41421, 0.81650) -- (2.00000, -1.41421, 2.44949) -- (2.00000, -1.41421, -0.81650) -- (2.00000, 0.00000, -1.63299) -- (2.00000, 1.41421, -0.81650) -- cycle {};
%%
%%
%% Drawing edges in the front
%%
\draw[edge] (-2.00000, -1.41421, -0.81650) -- (-2.00000, -1.41421, 4.08248);
\draw[edge] (-2.00000, -1.41421, -0.81650) -- (-0.66667, -2.35702, -0.81650);
\draw[edge] (-2.00000, -1.41421, 4.08248) -- (-2.00000, 4.24264, 0.81650);
\draw[edge] (-2.00000, -1.41421, 4.08248) -- (-0.66667, -2.35702, 4.08248);
\draw[edge] (-2.00000, 4.24264, 0.81650) -- (2.00000, 1.41421, 0.81650);
\draw[edge] (-0.66667, -2.35702, -0.81650) -- (-0.66667, -2.35702, 4.08248);
\draw[edge] (-0.66667, -2.35702, -0.81650) -- (0.66667, -1.88562, -1.63299);
\draw[edge] (-0.66667, -2.35702, 4.08248) -- (2.00000, -1.41421, 2.44949);
\draw[edge] (0.66667, -1.88562, -1.63299) -- (0.66667, -0.47140, -2.44949);
\draw[edge] (0.66667, -1.88562, -1.63299) -- (2.00000, -1.41421, -0.81650);
\draw[edge] (0.66667, -0.47140, -2.44949) -- (2.00000, 0.00000, -1.63299);
\draw[edge] (2.00000, -1.41421, -0.81650) -- (2.00000, -1.41421, 2.44949);
\draw[edge] (2.00000, -1.41421, -0.81650) -- (2.00000, 0.00000, -1.63299);
\draw[edge] (2.00000, -1.41421, 2.44949) -- (2.00000, 1.41421, 0.81650);
\draw[edge] (2.00000, 0.00000, -1.63299) -- (2.00000, 1.41421, -0.81650);
\draw[edge] (2.00000, 1.41421, -0.81650) -- (2.00000, 1.41421, 0.81650);
%%
%%
%% Drawing the vertices in the front
%%
\node[vertex] at (-2.00000, -1.41421, -0.81650)     {};
\node[vertex] at (-2.00000, -1.41421, 4.08248)     {};
\node[vertex] at (-2.00000, 4.24264, 0.81650)     {};
\node[vertex] at (-0.66667, -2.35702, -0.81650)     {};
\node[vertex] at (-0.66667, -2.35702, 4.08248)     {};
\node[vertex] at (0.66667, -1.88562, -1.63299)     {};
\node[vertex] at (0.66667, -0.47140, -2.44949)     {};
\node[vertex] at (2.00000, -1.41421, -0.81650)     {};
\node[vertex] at (2.00000, -1.41421, 2.44949)     {};
\node[vertex] at (2.00000, 0.00000, -1.63299)     {};
\node[vertex] at (2.00000, 1.41421, -0.81650)     {};
\node[vertex] at (2.00000, 1.41421, 0.81650)     {};
\end{tikzpicture} \\[2cm]
	permutahedron &
	&
	cyclohedron &
	&
	associahedron
\end{tabular}}
	\caption{A sequence of polytope deformations, from the permutahedron, through the cyclohedron, to the associahedron.}
	\label{fig:deformations}
\end{figure}
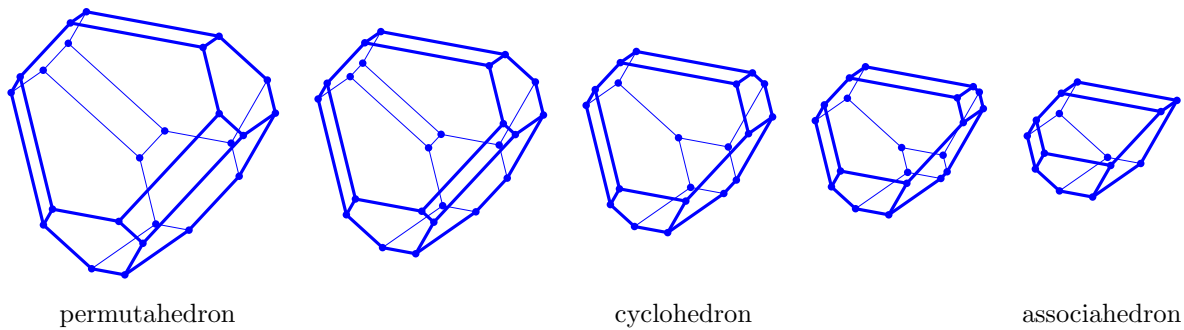

This paper focuses on some specific deformed permutahedra generalizing the associahedra, namely the graph associahedra and nestohedra.
Graph associahedra were defined by M.~Carr and S.~Devadoss~\cite{CarrDevadoss} in connection to C.~De Concini and C.~Procesi's wonderful arrangements~\cite{DeConciniProcesi}.
For a given graph~$\graphG$, the \defn{$\graphG$-associahedron}~$\Asso[\graphG]$ is a simple polytope whose combinatorial structure encodes the connected induced subgraphs of~$\graphG$ and their nested structure.
More precisely, the $\graphG$-associahedron is a polytopal realization of the \defn{nested complex} of~$\graphG$, defined as the simplicial complex of all collections of \defn{tubes} (connected induced subgraphs) of~$\graphG$ which are pairwise \defn{compatible} (either nested, or disjoint and non-adjacent).
As illustrated in \cref{fig:specialGraphAssociahedra}, the graph associahedra of certain special families of graphs coincide with well-known families of polytopes: complete graph associahedra are permutahedra, path associahedra are classical associahedra, cycle associahedra are cyclohedra, and star associahedra are stellohedra.
Graph associahedra were extended to \defn{nestohedra}, which are simple polytopes realizing the nested complex of arbitrary \defn{building sets}~\cite{Postnikov, FeichtnerSturmfels}.
Graph associahedra and nestohedra have been constructed in different ways: by successive truncations of faces of the standard simplex~\cite{CarrDevadoss}, as Minkowski sums of faces of the standard simplex~\cite{Postnikov, FeichtnerSturmfels}, or from their normal fans by exhibiting explicit inequality descriptions~\cite{Devadoss, Zelevinsky}.
For a given building set, the resulting polytopes all have the same normal fan, called \defn{nested fan}, whose rays are given by the characteristic vectors of the building blocks, and whose cones are given by the nested sets.
As all nested fans coarsen the braid fan, all graph associahedra and nestohedra are deformed permutahedra, and hence they can be obtained by gliding facets of the permutahedron.
However, in contrast to the classical associahedron~\cite{ShniderSternberg,Loday,HohlwegLange}, note that some graph associahedra and nestohedra cannot be obtained by deleting inequalities in the facet description of the permutahedron~\cite{Pilaud-removahedra}.

\hvFloat[floatPos=p, capWidth=h, capPos=right, capAngle=90, objectAngle=90, capVPos=c, objectPos=c]{figure}
{
		\begin{tabular}{c@{\;}c@{\;}c@{\;}c}
			permutahedron &
			associahedron &
			cyclohedron &
			stellohedron \\[.2cm]
			\includegraphics[scale=.93]{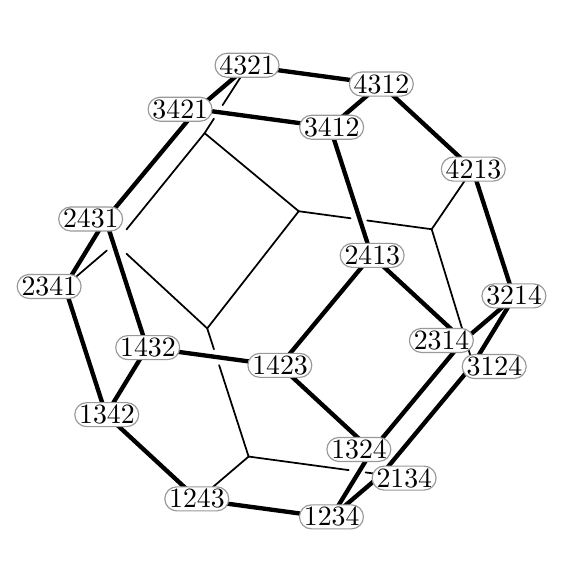} &
			\includegraphics[scale=.93]{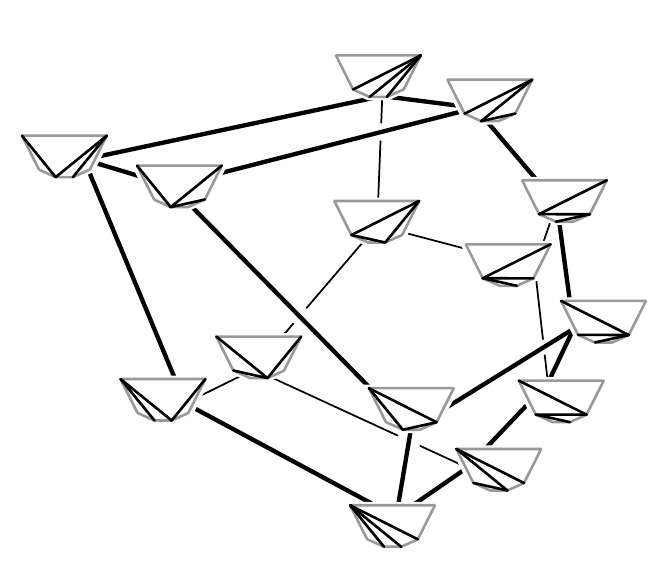} &
			\includegraphics[scale=.93]{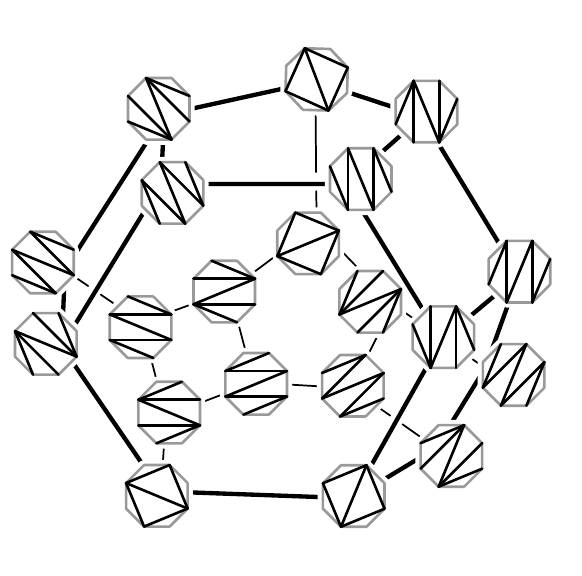} &
			\includegraphics[scale=.93]{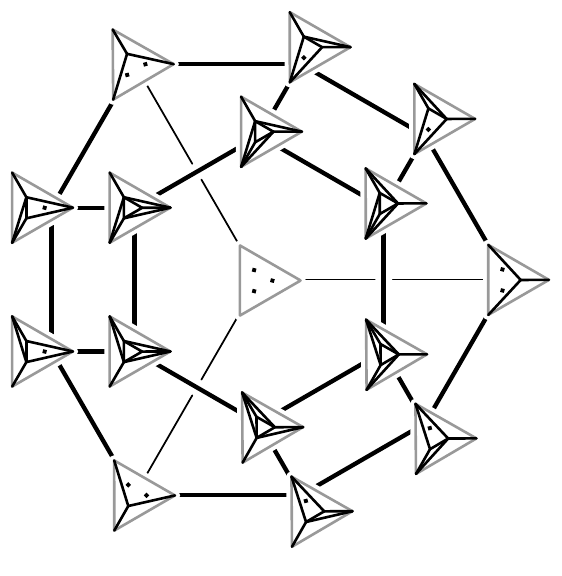} \\[.1cm]
			\includegraphics[scale=.93]{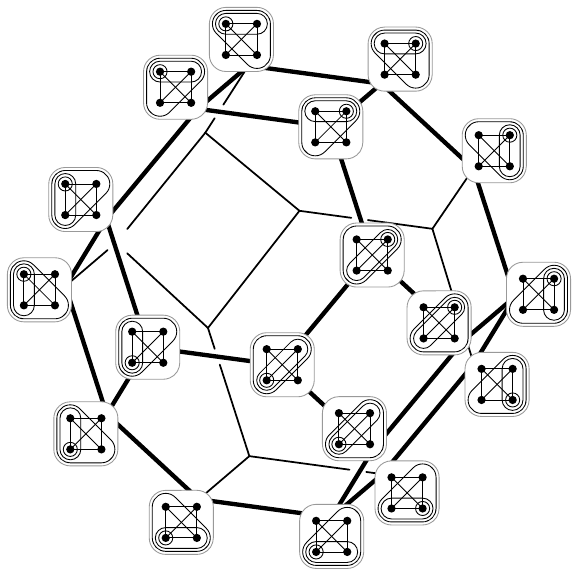} &
			\includegraphics[scale=.93]{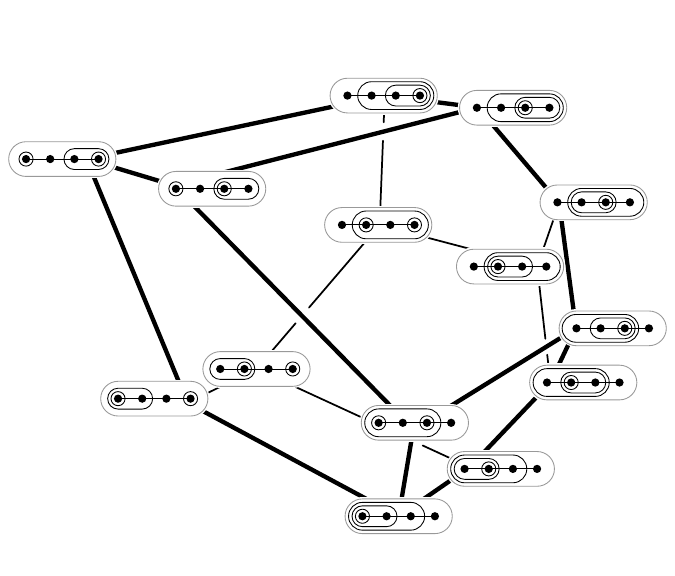} &
			\includegraphics[scale=.93]{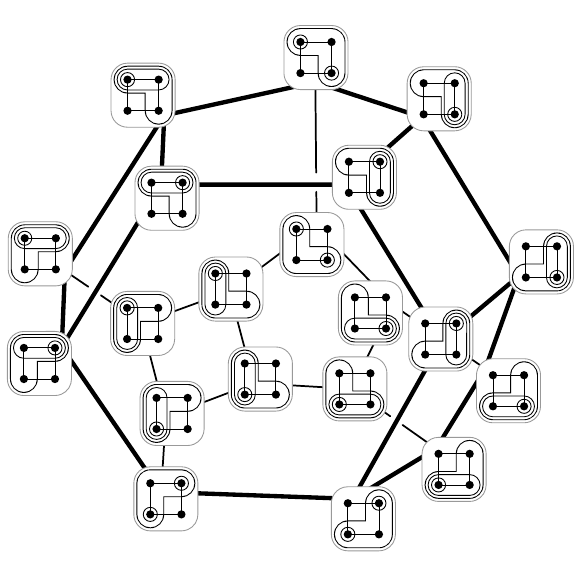} &
			\includegraphics[scale=.93]{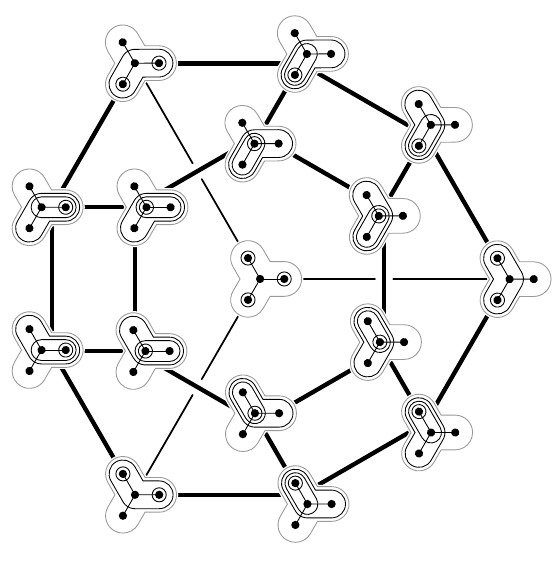} \\[.2cm]
			complete graph &
			path &
			cycle &
			star \\[.3cm]
		\end{tabular}
}
{Some classical families of polytopes as graph associahedra. Illustration from~\cite{MannevillePilaud-compatibilityFans}.}
{fig:specialGraphAssociahedra}

%\begin{figure}
%	\capstart
%	\begin{adjustbox}{center}
%		\begin{tabular}{c@{\;}c@{\;}c@{\;}c}
%			permutahedron &
%			associahedron &
%			cyclohedron &
%			stellohedron \\
%			\includegraphics[scale=.6]{permutahedron} &
%			\includegraphics[scale=.6]{associahedron} &
%			\includegraphics[scale=.6]{cyclohedron} &
%			\includegraphics[scale=.6]{stellohedron} \\
%			\includegraphics[scale=.6]{permutahedronTubings} &
%			\includegraphics[scale=.6]{associahedronTubings} &
%			\includegraphics[scale=.6]{cyclohedronTubings} &
%			\includegraphics[scale=.6]{stellohedronTubings} \\[.1cm]
%			complete graph &
%			path &
%			cycle &
%			star
%		\end{tabular}
%	\end{adjustbox}
%	\caption{Some classical families of polytopes as graph associahedra. Illustration from~\cite{MannevillePilaud-compatibilityFans}.}
%	\label{fig:specialGraphAssociahedra}
%\end{figure}

In this paper, we describe all realizations of the nested fans by studying the deformation cone of the $\graphG$-associahedron for any graph~$\graphG$ (\cref{sec:typeConeGraphicalNestedFan}) and of the $\building$-nestohedron of any building set~$\building$ (\cref{sec:typeConeNestedFan}).
Our main contribution is an irredundant facet description of these deformation cones, characterizing which of the wall-crossing inequalities are irreplaceable (\cref{thm:extremalExchangeablePairsGraphicalNestedFan,thm:extremalExchangeFramesNestedFan}).
Even though the graphical case is a specialization of the general case, we present it first separately since it admits a much simpler description that serves as an introduction for the general case.
This simplification relies on two pleasant properties (\cref{prop:exchangeablePairsGraphicalNestedFan}): first, the classical simple characterization of the pairs of exchangeable tubes, and second, the fact that the wall-crossing inequalities only depend on their exchanged tubes.
This yields the following statement, where~$\tubes$ denotes the set of tubes of~$\graphG$ and~$\connectedComponents(G)$ denotes the set containing~$\varnothing$ and all connected components~of~$G$.

\begin{theorem*}[\cref{thm:extremalExchangeablePairsGraphicalNestedFan}]
\label{thm:facetsDeformationConeGraphAsso}
For any graph~$\graphG$, the deformation cone of the graph associahedron~$\Asso[\graphG]$ is the set of polytopes
\(
\set{\b{x} \in \R^V}{-\sum_{v \in V} \b{x}_v \le \b{h}_\varnothing \text{ and } \sum_{v \in \tube} \b{x}_v \le \b{h}_{\tube} \text{ for all } \tube \in \tubes \ssm \{\varnothing\}}
%\bigset{\b{x} \in \R^V}{\dotprod{\gvector{\tube}}{\b{x}} \le \b{h}_{\tube} \text{ for all } \tube \in \tubes}
\)
for all~$\b{h}$ in the cone of~$\R^{\tubes}$ defined by the following irredundant facet description:
\begin{itemize}
\item $\sum_{K \in \connectedComponents(G)} \b{h}_K = 0$,
\item $\b{h}_{\tube} + \b{h}_{\tube'} \ge \b{h}_{\tube \cup \tube'} + \sum_{\tube[s] \in \connectedComponents(\tube \cap \tube')} \b{h}_{\tube[s]}$ for any tubes~$\tube, \tube'$ of~$G$ such that~$\tube \ssm \{v\} = \tube' \ssm \{v'\}$ for some neighbor~$v$ of~$\tube'$ in~$\tube \ssm \tube'$ and some neighbor~$v'$ of~$\tube$ in~$\tube' \ssm \tube$.
\end{itemize}
\end{theorem*}

The non-graphical case is much more involved.
First, we need a characterization of the pairs of exchangeable blocks (\cref{prop:exchangeablePairsNestedFan}), which was surprisingly missing for arbitrary \mbox{building sets (\cref{rem:Zelevinsky}).}
Second, the wall-crossing inequalities do not anymore correspond to the pairs of exchangeable blocks.
Namely, the wall-crossing inequalities do not only depend on the exchanged blocks, but also on an additional structure that we call the \defn{frame} of the exchange.
Moreover, some distinct exchange frames actually yield the same wall-crossing inequalities.
Taming these technical difficulties, we obtain the following statement, where~$\connectedComponents(\building)$ denotes the set containing~$\varnothing$ and all connected components of~$\building$, where $\mu(P)$ denotes the set of maximal blocks of~$\building$ strictly contained in a block~$P \in \building$, and where a block~$P$ is called \defn{elementary} when the blocks of~$\maximalBlocks(P)$ are disjoint (see \cref{sec:typeConeNestedFan} for details on these notations).

\begin{theorem*}[\cref{thm:extremalExchangeFramesNestedFan}]
\label{thm:facetsDeformationConeNesto}
For any building set~$\building$, the deformation cone of the nestohedron~$\Nest$ is the set of polytopes
\(
\set{\b{x} \in \R^V}{-\sum_{v \in V} \b{x}_v \le \b{h}_\varnothing \text{ and } \sum_{v \in B} \b{x}_v \le \b{h}_B \text{ for all } B \in \building \ssm \{\varnothing\}}
%\bigset{\b{x} \in \R^V}{\dotprod{\gvector{B}}{\b{x}} \le \b{h}_B \text{ for all } B \in \building}
\)
for all~$\b{h}$ in the cone of~$\R^\building$ defined by the following irredundant facet description:
\begin{itemize}
\item $\sum_{K \in \connectedComponents(\building)} \b{h}_K = 0$,
\item $\sum_{B \in\maximalBlocks(P)} \b{h}_{B} \ge \b{h}_P$ for any elementary block~$P$ of~$\building$,
\item $\b{h}_B + \b{h}_{B'} + \sum_{K \in \connectedComponents(P \ssm (B \cup B'))} \b{h}_K \ge \b{h}_P + \sum_{K \in \connectedComponents(B \cap B')} \b{h}_K$ for any block~$P$ of~$\building$ neither singleton nor elementary, and any two blocks~$B \ne B'$ in~$\maximalBlocks(P)$.
\end{itemize}
\end{theorem*}

These irredundant inequality descriptions enable us to count the facets of these deformation cones and thus to determine when these deformation cones are simplicial.
It turns out that the deformation cone of the $\graphG$-associahedron is simplicial if and only if~$\graphG$ is a disjoint union of paths (\ie the $\graphG$-associahedron is a Cartesian product of classical associahedra).
In contrast, there is much more freedom for nestohedra of arbitrary building sets, and we show that the deformation cone of the nestohedron is always simplicial for an \defn{interval building set}, that is a building set whose blocks are some intervals of~$[n]$ (\cref{prop:simplicialTypeConeInterval}).
As advocated in~\cite{PadrolPaluPilaudPlamondon}, the simpliciality of the deformation cone leads to an elegant description of all deformations of the polytope in the so-called kinematic space~\cite{ArkaniHamedBaiHeYan}.
Generalizing the kinematic associahedra of~\cite{ArkaniHamedBaiHeYan}, we thus define the \defn{kinematic nestohedra} of arbitrary interval building sets (\cref{prop:kinematicNestohedraInterval}).

%%%%%%%%%%%%%%%%%%%%%%%%%%%%%%%%%%%%%%%

\section{Geometric preliminaries}
\label{sec:preliminaries}

We first briefly recall basic notions of polytopes and fans, and the definition of the type cone of~\cite{McMullen-typeCone}, following the presentation of~\cite{PadrolPaluPilaudPlamondon}.

%%%%%%%%%%%

\subsection{Fans and polytopes}
\label{subsec:fansPolytopes}

A (polyhedral) \defn{cone} is the positive span of finitely many vectors or equivalently, the intersection of finitely many closed linear half-spaces. 
The \defn{faces} of a cone are its intersections with its supporting hyperplanes. 
The \defn{rays} (resp.~\defn{facets}) are the faces of dimension~$1$ (resp.~ codimension~$1$).
A cone is \defn{simplicial} if its rays are linearly independent.
A (polyhedral) \defn{fan}~$\Fan$ is a set of cones such that any face of a cone of~$\Fan$ belongs to~$\Fan$, and any two cones of~$\Fan$ intersect along a face of both. 
A fan is \defn{essential} if the intersection of its cones is the origin, \defn{complete} if the union of its cones covers~$\R^n$, and \defn{simplicial} if all its cones are simplicial.
In a simplicial fan, we say that two maximal cones are \defn{adjacent} if they share a facet, and that two rays are \defn{exchangeable} if they belong to two adjacent cones but not to their common facet.
We will say that the fan \defn{realizes} the simplicial complex consisting of the subsets of rays spanning cones.

A \defn{polytope} is the convex hull of finitely many points or equivalently, a bounded intersection of finitely many closed affine half-spaces.
The \defn{faces} of a polytope are its intersections with its supporting hyperplanes.
The \defn{vertices} (resp.~\defn{edges}, resp.~\defn{facets}) are the faces of dimension~$0$ (resp.~dimension~$1$, resp.~codimension~$1$).

The \defn{normal cone} of a face~$F$ of a polytope~$P$ is the cone generated by the normal vectors of the facets of~$P$ containing~$F$.
Said differently, it is the cone of vectors~$\b{c}$ such that the linear form~$\b{x} \mapsto \dotprod{\b{c}}{\b{x}}$ on~$P$ is maximized by all points of the face~$F$.
The \defn{normal fan} of~$P$ is the set of normal cones of all its faces.

%%%%%%%%%%%

\subsection{Type cone}
\label{subsec:typeCone}

Fix an essential complete simplicial fan~$\Fan$ in~$\R^n$. Let~$\b{G}$ be the $N \times n$-matrix whose rows are (representative vectors of) the rays of~$\Fan$.
For any height vector~$\b{h} \in \R^N$, we define the polytope
\(
P_\b{h} \eqdef \set{\b{x} \in \R^n}{\b{G}\b{x} \le \b{h}}.
\)
The following classical statement characterizes the height vectors~$\b{h}$ for which the fan~$\Fan$ is the normal fan of this polytope~$P_\b{h}$.

\begin{proposition}[\cite{GelfandKapranovZelevinsky,ChapotonFominZelevinsky}]
\label{prop:characterizationPolytopalFan}
Let~$\Fan$ be an essential complete simplicial fan in~$\R^n$. Then the following are equivalent for any height vector~$\b{h} \in \R^N$:
\begin{enumerate}
\item The fan~$\Fan$ is the normal fan of the polytope~$P_\b{h} \eqdef \set{\b{x} \in \R^n}{\b{G}\b{x} \le \b{h}}$.
\item For any two adjacent maximal cones~$\R_{\ge0}\b{R}$ and~$\R_{\ge0}\b{R}'$ of~$\Fan$ with~${\b{R} \ssm \{\b{r}\} = \b{R}' \ssm \{\b{r}'\}}$, the height vector~$\b{h}$ satisfies the \defn{wall-crossing inequality}
\(
\sum_{\b{s} \in \b{R} \cup \b{R}'} \coefficient[{\b{s}}][\b{R}][\b{R}'] \, \b{h}_{\b{s}} > 0,
\)
where
\(
\sum_{\b{s} \in \b{R} \cup \b{R}'} \coefficient[{\b{s}}][\b{R}][\b{R}'] \, \b{s} = \b{0}
\)
is the unique linear dependence among the rays of~$\b{R} \cup \b{R}'$ such that~$\coefficient[{\b{r}}][\b{R}][\b{R}'] + \coefficient[{\b{r}'}][\b{R}][\b{R}'] = 2$. 
\end{enumerate}
\end{proposition}

Following \cite{McMullen-typeCone}, we define the \defn{type cone} of~$\Fan$ as the cone~$\typeCone(\Fan)$ of polytopal realizations of~$\Fan$.
Based on \cref{prop:characterizationPolytopalFan} it can be parametrized as
\begin{align*}
\typeCone(\Fan) & \eqdef \set{\b{h} \in \R^N}{\Fan \text{ is the normal fan of } P_\b{h}} \\
& = \Bigset{\b{h} \in \R^N}{\sum_{\b{s} \in \b{R} \cup \b{R}'} \coefficient[{\b{s}}][\b{R}][\b{R}'] \, \b{h}_{\b{s}} > 0 \; \begin{array}{l} \text{for any adjacent maximal} \\ \text{cones~$\R_{\ge0}\b{R}$ and~$\R_{\ge0}\b{R}'$ of~$\Fan$} \end{array}}.
\end{align*}
Note that~$\typeCone(\Fan)$ is an open polyhedral cone (dilations preserve normal fans) and contains a lineality subspace of dimension~$n$ (translations preserve normal fans). Its closure $\ctypeCone(\Fan)$ consists of all polytopes whose normal fan coarsens~$\Fan$, and is called the \defn{deformation cone}. Its faces are the deformation cones of the coarsenings of~$\Fan$. Having into account the lineality, we will say that the type cone is \defn{simplicial} when it has precisely $N-n$ facets.

When the type cone is simplicial, it naturally defines alternative polytopal realizations of the fan~$\Fan$ in the non-negative orthant parametrized by positive vectors, akin to the realizations of the kinematic associahedra introduced in~\cite{ArkaniHamedBaiHeYan}. See~\cite[Sect.~1.4]{PadrolPaluPilaudPlamondon} for details.

\begin{proposition}[{\cite{PadrolPaluPilaudPlamondon}}]
\label{prop:simplicialTypeCone}
Assume that the type cone~$\typeCone(\Fan)$ is simplicial and let~$\b{K}$ be the $(N-n) \times N$-matrix whose rows are the inner normal vectors of the facets of~$\typeCone(\Fan)$. Then, for any positive vector~${\b{p} \in \R_{>0}^{N-n}}$, the polytope
\(
R_\b{p} \eqdef \set{\b{z} \in \R^N}{\b{K}\b{z} = \b{p} \text{ and } \b{z} \ge 0}
\)
is a realization of the fan~$\Fan$.
Moreover, the polytopes~$R_\b{p}$ for~$\b{p} \in \R_{>0}^{N-n}$ describe all polytopal realizations of~$\Fan$ (up to translation).
\end{proposition}

%%%%%%%%%%%%%%%%%%%%%%%%%%%%%%%%%%%%%%%

\section{Type cones of graphical nested fans}
\label{sec:typeConeGraphicalNestedFan}

In this section, we study graphical nested fans, postponing the study of arbitrary nested fans to \cref{sec:typeConeNestedFan}.
While the graphical case is significantly simpler than the general case, some proof ideas presented here will be transported to \cref{sec:typeConeNestedFan}.
This section is thus useful both to the readers only interested in the graphical case and as a prototype for the general case.

%%%%%%%%%%%

\subsection{Graphical nested complex, graphical nested fan, and graph associahedron}
\label{subsec:graphicalNestedComplexGraphicalNestedFan}

We start with the definitions and properties of the nested complex of a graph, using material from~\cite{CarrDevadoss, Postnikov, FeichtnerSturmfels, Zelevinsky, MannevillePilaud-compatibilityFans}.

\subsubsection{Graphical nested complex}
Let~$\graphG$ be a graph with vertex set~$\ground$.
A \defn{tube} of~$\graphG$ is a non-empty subset of vertices of~$\graphG$ whose induced subgraph is connected.
The set of tubes of~$\graphG$ is denoted by~$\tubes$.
The (inclusion) maximal tubes of~$\graphG$ are its \defn{connected components}~$\connectedComponents(\graphG)$.
Two tubes~$\tube, \tube'$ of~$\graphG$ are \defn{compatible} if they are either nested (\ie $\tube \subseteq \tube'$ or~$\tube' \subseteq \tube$), or disjoint and non-adjacent~(\ie~${\tube \cup \tube' \notin \tubes}$).
Note that any connected component of~$\graphG$ is compatible with any other tube of~$\graphG$.
A \defn{tubing} on~$\graphG$ is a set~$\tubing$ of pairwise compatible tubes of~$\graphG$ containing all connected components~$\connectedComponents(\graphG)$.
Examples are illustrated in \cref{fig:exmNestedGraphical}.
The \defn{nested complex} of~$\graphG$ is the simplicial complex~$\nestedComplex(\graphG)$ whose faces are~$\tubing \ssm \connectedComponents(\graphG)$ for all tubings~$\tubing$~on~$\graphG$.
If~$\tubing \ssm \{\tube\} = \tubing' \ssm \{\tube'\}$ for two maximal tubings~$\tubing$ and~$\tubing'$ and two tubes~$\tube$ and~$\tube'$, we say that~$\tubing$ and $\tubing'$ are \defn{adjacent} and that~$\tube$ and~$\tube'$ are~\defn{exchangeable}.

\begin{figure}[h]
	\capstart
	\centerline{\includegraphics[scale=.5]{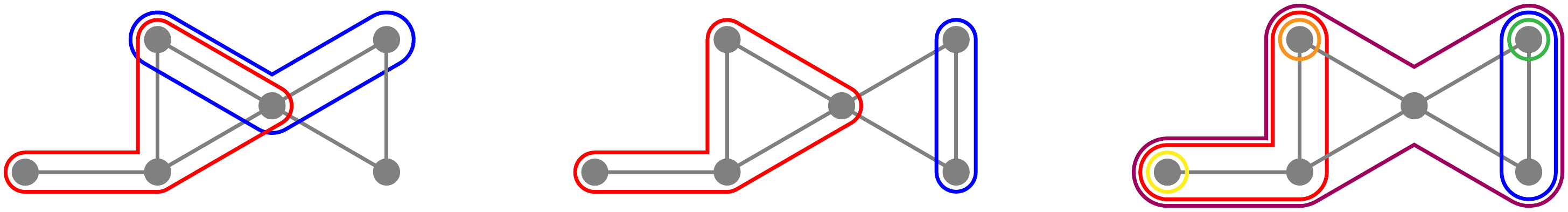}}
	\caption{Some incompatible tubes (left and middle), and a maximal tubing (right).}
	\label{fig:exmNestedGraphical}
\end{figure}

\subsubsection{Graphical nested fan}
Let~$(\b{e}_v)_{v \in \ground}$ be the canonical basis of~$\R^\ground$.
We consider the subspace~${\HH \eqdef \bigset{\b{x} \in \R^\ground}{\sum_{v \in K} x_v = 0 \text{ for all } K \in \connectedComponents(\graphG)}}$ and let~$\pi : \R^\ground \to \HH$ denote the orthogonal projection onto~$\HH$.
The \defn{$\b{g}$-vector} of a tube~$\tube$ of~$\graphG$ is the projection~$\gvector{\tube} \eqdef \pi \big( \sum_{v \in \tube} \b{e}_v \big)$ of the characteristic vector of~$\tube$.
We set~$\gvectors{\tubing} \eqdef \set{\gvector{\tube}}{\tube \in \tubing}$ for a tubing~$\tubing$ on~$\graphG$.
Note that by definition, $\gvector{\varnothing} = \b{0}$ and $\gvector{K} = \b{0}$ for all connected components~$K \in \connectedComponents(\graphG)$.
The vectors~$\gvector{\tube}$ with~$\tube \in \tubes$ support a complete simplicial fan realization of the nested complex.
See~\cref{fig:graphicalNestedFans}.

\begin{theorem}[\cite{CarrDevadoss, Postnikov, FeichtnerSturmfels, Zelevinsky}]
\label{thm:graphicalNestedFan}
For any graph~$\graphG$, the set of cones
\[
\nestedFan[\graphG] \eqdef \set{\R_{\ge 0} \, \gvectors{\tubing}}{\tubing \text{ tubing on } \graphG}
\]
is a complete simplicial fan of~$\HH$, called the \defn{nested fan} of~$\graphG$, realizing the nested complex~$\nestedComplex(\graphG)$.
\end{theorem}

\begin{figure}[h]
	\capstart
	\centerline{\includegraphics[scale=.55]{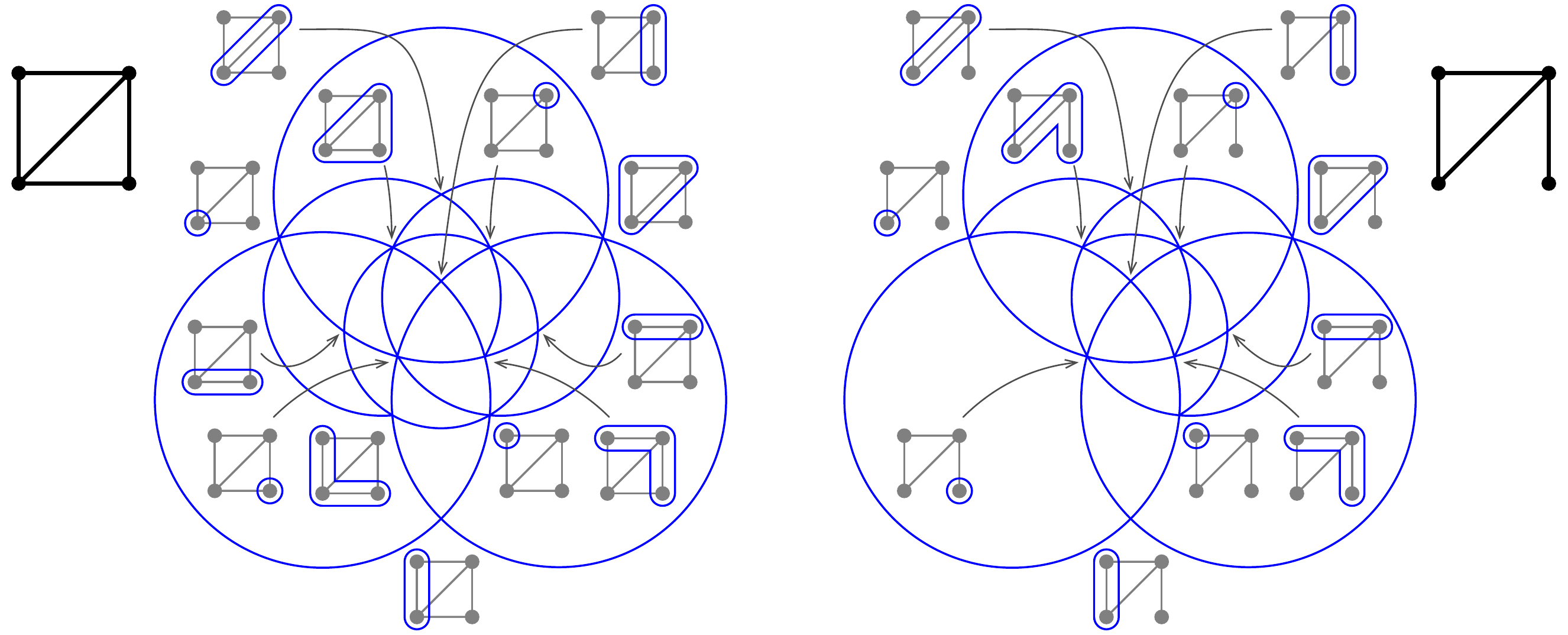}}
	\caption{Two graphical nested fans. The rays are labeled by the corresponding tubes. As the fans are $3$-dimensional, we intersect them with the sphere and stereographically project them from the direction~$(-1,-1,-1)$.}
	\label{fig:graphicalNestedFans}
\end{figure}

\subsubsection{Graph associahedron}
The following statement is proved in~\cite{CarrDevadoss, Devadoss, Postnikov, FeichtnerSturmfels, Zelevinsky}.
For a subset~$U \subseteq \ground$, denote by~$\triangle_U \eqdef \conv\set{\b{e}_u}{u \in U}$ the face of the standard simplex~$\triangle_\ground$ corresponding to~$U$.

\begin{figure}
	\capstart
	\centerline{\includegraphics[scale=.55]{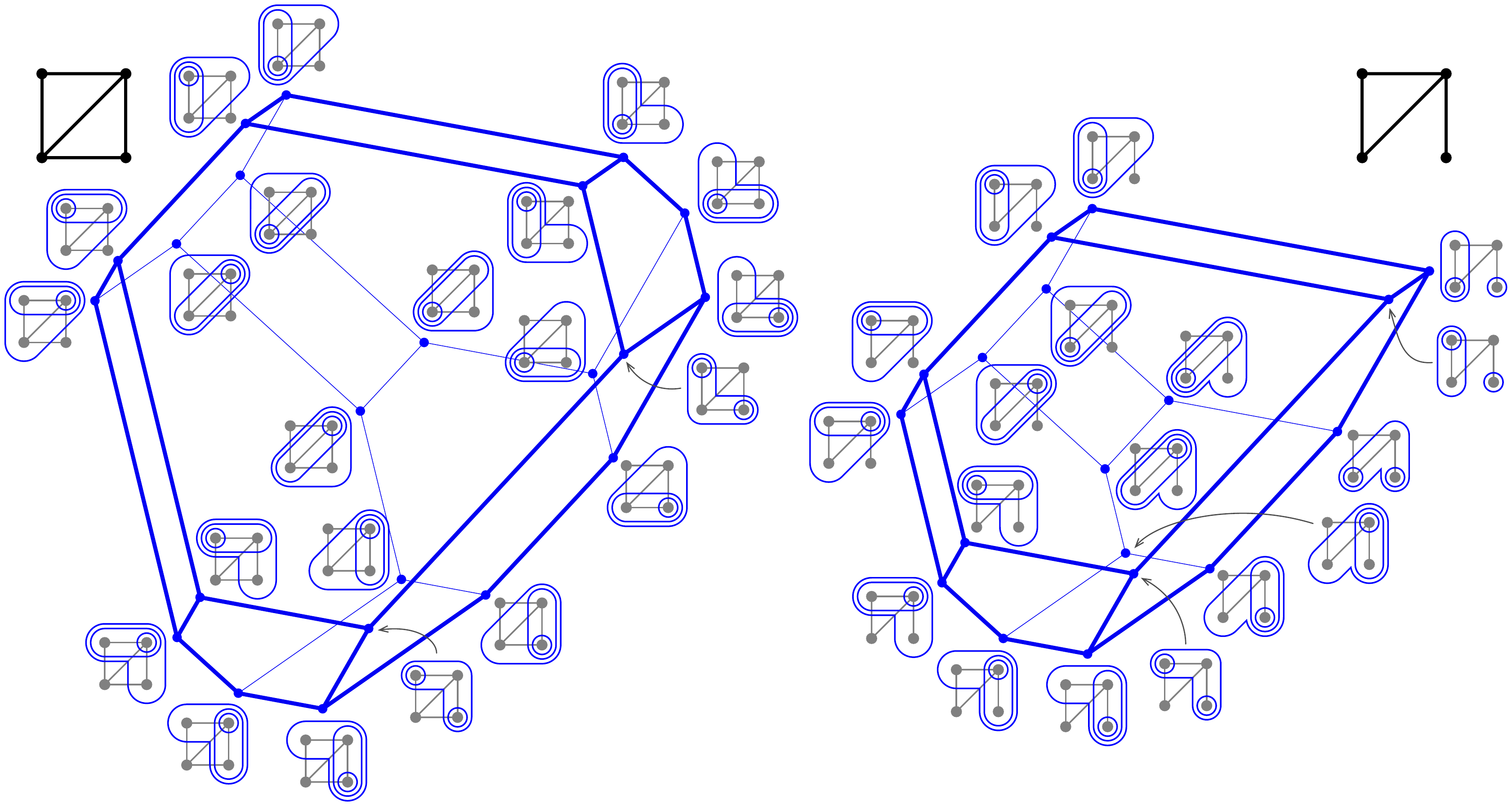}}
	\caption{Two graph associahedra, realizing the graphical nested fans of \cref{fig:graphicalNestedFans}. The vertices are labeled by the corresponding maximal tubings.}
	\label{fig:graphAssociahedra}
\end{figure}

\begin{theorem}[\cite{CarrDevadoss, Devadoss, Postnikov, FeichtnerSturmfels, Zelevinsky}]
\label{thm:graphAssociahedron}
For any graph~$\graphG$, the nested fan~$\nestedFan[\graphG]$ is the normal fan of a polytope.
For instance, $\nestedFan[\graphG]$ is the normal fan of
\begin{enumerate}[(i)]
\item the intersection of~$\HH$ with the hyperplanes~$\dotprod{\gvector{\tube}}{\b{x}} \le -3^{|\tube|}$ for all tubes~$\tube \in \tubes$~\cite{Devadoss},
\item the Minkowski sum~$\sum_{\tube \in \tubes} \triangle_{\tube}$ of the faces of the standard simplex given by all tubes of~$\graphG$~\cite{Postnikov}.
\end{enumerate}
\end{theorem}

\begin{definition}
Any polytope whose normal fan is the nested fan~$\nestedFan[\graphG]$ is called \defn{graph associahedron} and denoted by~$\Asso[\graphG]$.
\end{definition}

For instance, \cref{fig:graphAssociahedra} represents the graph associahedra realizing the graphical nested fans of \cref{fig:graphicalNestedFans} and obtained using the construction (ii) of \cref{thm:graphAssociahedron}.

\begin{remark}
For instance,
\begin{enumerate}[(i)]
\item for the complete graph~$\graphG[K]_n$, the tubes are all non-empty subsets of~$[n]$, the tubings correspond to ordered partitions of~$[n]$, the maximal tubings correspond to permutations of~$[n]$, the graphical nested fan~$\nestedFan[{\graphG[K]_n}]$ is the classical \defn{braid fan}, and the graph associahedron~$\Asso[{\graphG[K]_n}]$ is the classical \defn{permutahedron} (see \eg \cite{Ziegler-polytopes, Hohlweg}),
\item for the path~$\graphG[P]_n$, the tubes are all non-empty intervals of~$[n]$, the tubings correspond to Schr\"oder trees on with $n+1$ leaves, the maximal tubings correspond to binary trees with $n+1$ leaves, the graphical nested fan~$\nestedFan[{\graphG[P]_n}]$ is the classical \defn{sylvester fan}, and the graph associahedron~$\Asso[{\graphG[P]_n}]$ is the classical \defn{associahedron} (see~\cite{ShniderSternberg,Loday}).
\end{enumerate}
\end{remark}

%%%%%%%%%%%

\pagebreak
\subsection{Exchangeable tubes and $\b{g}$-vector dependences}
\label{subsec:exchangeableTubes}

The next statement \mbox{follows from~\cite{MannevillePilaud-compatibilityFans, Zelevinsky}.}

\begin{proposition}
\label{prop:exchangeablePairsGraphicalNestedFan}
Let~$\tube, \tube'$ be two tubes of~$\graphG$. Then
\begin{enumerate}[(i)]
\item The tubes~$\tube$ and~$\tube'$ are exchangeable in~$\nestedFan[\graphG]$ if and only if $\tube'$ has a unique neighbor~$v$ in~$\tube \ssm \tube'$ and $\tube$ has a unique neighbor~$v'$ in~$\tube' \ssm \tube$.
\item For any adjacent maximal tubings~$\tubing, \tubing'$ on~$\graphG$ with $\tubing \ssm \{\tube\} = \tubing' \ssm \{\tube'\}$, both~$\tubing$ and~$\tubing'$ contain the tube~$\tube \cup \tube'$ and the connected components of~$\tube \cap \tube'$.
\item The linear dependence between the $\b{g}$-vectors of~$\tubing \cup \tubing'$ is given by
\[
\gvector{\tube} + \gvector{\tube'} = \gvector{\tube \cup \tube'} + \sum_{\tube[s] \in \connectedComponents(\tube \cap \tube')} \gvector{\tube[s]}.
\]
In particular, it only depends on the exchanged tubes~$\tube$ and~$\tube'$, not on the tubings~$\tubing$ and~$\tubing'$.
%\item The $\b{c}$-vector orthogonal to all $\b{g}$-vectors~$\gvector{\tube[s]}$ for~$\tube[s] \in \tubing \cap \tubing'$ is~$\b{e}_v - \b{e}_{v'}$.
\end{enumerate}
\end{proposition}

\begin{proof}
Points~(i) and~(ii) were proved in~\cite{MannevillePilaud-compatibilityFans}. Point~(iii) follows from the fact that
\[
\sum_{v \in \tube} \b{e}_v + \sum_{v \in \tube'} \b{e}_v = \sum_{v \in \tube \cup \tube'} \b{e}_v + \sum_{v \in \tube \cap \tube'} \b{e}_v = \sum_{v \in \tube \cup \tube'} \b{e}_v + \sum_{\tube[s] \in \connectedComponents(\tube \cap \tube')} \sum_{v \in \tube[s]} \b{e}_v.
\qedhere
\]
%Finally for~(iv), any tube~$\tube[s] \in \tubing \cap \tubing'$ that contains~$v$ or~$v'$ actually contains both (to be compatible with~$\tube$ and~$\tube'$). Therefore, $\gvector{\tube[s]}$ is orthogonal to~$\b{e}_v - \b{e}_{v'}$ for any tube~$\tube[s] \in \tubing \cap \tubing'$.
\end{proof}

For instance, the two tubes on the left of \cref{fig:exmNestedGraphical} are exchangeable, while the two tubes on the middle of \cref{fig:exmNestedGraphical} are not.

%%%%%%%%%%%

\subsection{Type cone of graphical nested fans}
\label{subsec:typeConeGraphicalNestedFans}

As a direct consequence of \cref{prop:exchangeablePairsGraphicalNestedFan}, we obtain the following (possibly redundant) description of the type cone of the graphical nested fan~$\nestedFan[\graphG]$.

\begin{corollary}
\label{coro:typeConeGraphicalNestedFan}
For any graph~$\graphG$, the type cone of the nested fan~$\nestedFan[\graphG]$ is given by
\[
\typeCone(\nestedFan[\graphG]) = \set{\b{h} \in \R^{\tubes}}{\begin{array}{l} \b{h}_{K} = 0 \text{ for any connected component } K \in \connectedComponents(\graphG) \text{ and}\\ \b{h}_{\tube} + \b{h}_{\tube'} > \b{h}_{\tube \cup \tube'} + \sum_{\tube[s] \in \connectedComponents(\tube \cap \tube')} \b{h}_{\tube[s]} \text{ for any exchangeable tubes } \tube, \tube' \end{array}}.
\]
\end{corollary}

We denote by~$\b{f}_{\tube}$ for~$\tube \in \tubes$ the canonical basis of~$\R^{\tubes}$ and by 
\[
\b{n}(\tube, \tube') \eqdef \b{f}_{\tube} + \b{f}_{\tube'} - \b{f}_{\tube\cup \tube'} - \sum_{\tube[s] \in \connectedComponents(\tube \cap \tube')} \b{f}_{\tube[s]}
\]
the inner normal vector of the inequality of the type cone~$\typeCone(\nestedFan[\graphG])$ corresponding to an exchangeable pair~$\{\tube, \tube'\}$ of tubes of~$\graphG$. Thus $\b{h} \in \typeCone(\nestedFan[\graphG])$ if and only if $
\dotprod{\b{n}(\tube, \tube')}{\b{h}} > 0$ for all exchangeable tubes~$\tube,\tube' \in \tubes$.

\begin{remark}
For instance,
\begin{enumerate}[(i)]
\item for the complete graph~$\graphG[K]_n$, the type cone~$\typeCone(\nestedFan[{\graphG[K]_n}])$ is formed by all strict submodular functions, \ie functions~$\b{h} : 2^{[n]} \to \R$ such that~$\b{h}_{\varnothing} = 0 = \b{h}_{[n]}$ and~${\b{h}_{A} + \b{h}_{B} > \b{h}_{A \cap B} + \b{h}_{A \cup B}}$ for any~$A, B \subseteq [n]$. The inequalities~$\b{h}_{U \ssm \{v\}} + \b{h}_{U \ssm \{v'\}} > \b{h}_{U} + \b{h}_{U \ssm \{v,v'\}}$ for~$v,v' \in \ground$ and~$\{v,v'\} \subseteq U \subseteq \ground$ clearly imply all submodular inequalities. The closure of the type cone~$\typeCone(\nestedFan[{\graphG[K]_n}])$ is the set of deformed permutahedra (or generalized permutahedra) studied by A.~Postnikov in~\cite{Postnikov} and E.-M.~Feichtner and B.~Sturmfels in~\cite{FeichtnerSturmfels}.
\item for the path~$\graphG[P]_n$, the type cone~$\typeCone(\nestedFan[{\graphG[P]_n}])$ is formed by the functions ${\b{h} : \set{[i,j]}{1 \le i \le j \le n} \mapsto \R}$ such that~$\b{h}_{[1,n]} = 0 = \b{h}_{\{i\}}$ for all~$i \in [n]$ and~$\b{h}_{[i,j]} + \b{h}_{[k,\ell]} > \b{h}_{[i,\ell]} + \b{h}_{[k,j]}$ for all~$1 \le i \le j \le n$ and~$1 \le k \le \ell \le n$ such that~$i < k$, $j < \ell$ and~$k \le j+1$ (where~$\b{h}_{[k,j]} = 0$ if~$k = j+1$).
\end{enumerate}
\end{remark}

\pagebreak

\begin{example}
\label{exm:typeConeGraphicalNestedFan}
Consider the graphical nested fans illustrated in \cref{fig:graphicalNestedFans}.
The type cone of the left fan lives in~$\R^{13}$, has a lineality space of dimension~$3$ and $19$ facet-defining inequalities (given below). In particular, it is not simplicial.
Note that as in \cref{fig:graphicalNestedFans}, we express the $\b{g}$-vectors in the basis given by the maximal tubing containing the first three tubes below.

\bigskip
\centerline{$
\begin{array}{r|c@{\;\,}c@{\;\,}c@{\;\,}c@{\;\,}c@{\;\,}c@{\;\,}c@{\;\,}c@{\;\,}c@{\;\,}c@{\;\,}c@{\;\,}c@{\;\,}c}
\text{tubes} & \raisebox{-.4cm}{\includegraphics[scale=.6]{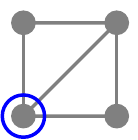}} & \raisebox{-.4cm}{\includegraphics[scale=.6]{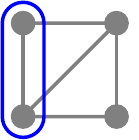}} & \raisebox{-.4cm}{\includegraphics[scale=.6]{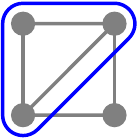}} & \raisebox{-.4cm}{\includegraphics[scale=.6]{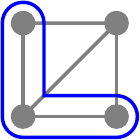}} & \raisebox{-.4cm}{\includegraphics[scale=.6]{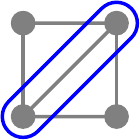}} & \raisebox{-.4cm}{\includegraphics[scale=.6]{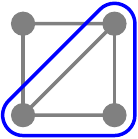}} & \raisebox{-.4cm}{\includegraphics[scale=.6]{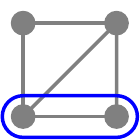}} & \raisebox{-.4cm}{\includegraphics[scale=.6]{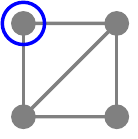}} & \raisebox{-.4cm}{\includegraphics[scale=.6]{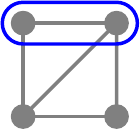}} & \raisebox{-.4cm}{\includegraphics[scale=.6]{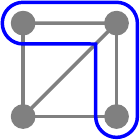}} & \raisebox{-.4cm}{\includegraphics[scale=.6]{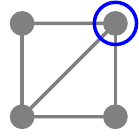}} & \raisebox{-.4cm}{\includegraphics[scale=.6]{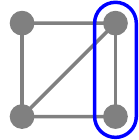}} & \raisebox{-.4cm}{\includegraphics[scale=.6]{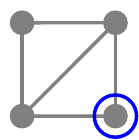}}  \\[.6cm]
\text{$\b{g}$-vectors} & \compactVectorT{1}{0}{0} & \compactVectorT{0}{1}{0} & \compactVectorT{0}{0}{1} & \compactVectorT{0}{1}{-1} & \compactVectorT{1}{-1}{1} & \compactVectorT{1}{-1}{0} & \compactVectorT{1}{0}{-1} & \compactVectorT{-1}{1}{0} & \compactVectorT{-1}{0}{1} & \compactVectorT{-1}{0}{0} & \compactVectorT{0}{-1}{1} & \compactVectorT{0}{-1}{0} & \compactVectorT{0}{0}{-1} \\[.6cm]
\text{facet}		& 0 & 0 & 0 & 0 & 0 & 0 & 0 & 0 & 0 & 0 & 1 & -1 & 1 \\
\text{defining}		& 0 & 0 & 0 & 0 & 0 & 1 & 0 & 0 & 0 & 1 & 0 & -1 & 0 \\
\text{inequalities}	& -1 & 1 & 0 & -1 & 0 & 0 & 1 & 0 & 0 & 0 & 0 & 0 & 0 \\
				& 0 & 0 & 0 & 0 & 0 & 0 & 0 & 0 & 1 & -1 & -1 & 1 & 0 \\
				& 0 & 0 & 0 & 0 & 0 & -1 & 1 & 0 & 0 & 0 & 0 & 1 & -1 \\
				& -1 & 1 & -1 & 0 & 1 & 0 & 0 & 0 & 0 & 0 & 0 & 0 & 0 \\
				& -1 & 0 & 0 & 0 & 1 & -1 & 1 & 0 & 0 & 0 & 0 & 0 & 0 \\
				& 0 & 1 & -1 & 0 & 0 & 0 & 0 & -1 & 1 & 0 & 0 & 0 & 0 \\
				& 0 & 0 & -1 & 0 & 1 & 0 & 0 & 0 & 1 & 0 & -1 & 0 & 0 \\
				& 0 & 0 & 0 & 0 & 1 & -1 & 0 & 0 & 0 & 0 & -1 & 1 & 0 \\
				& 0 & 0 & 1 & 0 & 0 & 0 & 0 & 0 & -1 & 1 & 0 & 0 & 0 \\
				& 0 & 0 & 1 & 0 & -1 & 1 & 0 & 0 & 0 & 0 & 0 & 0 & 0 \\
				& 0 & 0 & 0 & 0 & 0 & 0 & 0 & 1 & -1 & 0 & 1 & 0 & 0 \\
				& 1 & 0 & 0 & 0 & 0 & 0 & -1 & 0 & 0 & 0 & 0 & 0 & 1 \\
				& 1 & 0 & 0 & 0 & -1 & 0 & 0 & 0 & 0 & 0 & 1 & 0 & 0 \\
				& 1 & -1 & 0 & 0 & 0 & 0 & 0 & 1 & 0 & 0 & 0 & 0 & 0 \\
				& 0 & 0 & 0 & 1 & 0 & 0 & 0 & -1 & 0 & 1 & 0 & 0 & -1 \\
				& 0 & 0 & 0 & 1 & 0 & 1 & -1 & 0 & 0 & 0 & 0 & 0 & 0 \\
				& 0 & -1 & 1 & 1 & 0 & 0 & 0 & 0 & 0 & 0 & 0 & 0 & 0 \\[.2cm]
\end{array}
$}

\bigskip
\noindent
The type cone of the right fan lives in~$\R^{11}$, has a lineality space of dimension~$3$ and $12$ facet-defining inequalities (given below). In particular, it is not simplicial.
Note that as in \cref{fig:graphicalNestedFans}, we express the $\b{g}$-vectors in the basis given by the maximal tubing containing the first three tubes~below.

\[
\begin{array}{r|c@{\;\,}c@{\;\,}c@{\;\,}c@{\;\,}c@{\;\,}c@{\;\,}c@{\;\,}c@{\;\,}c@{\;\,}c@{\;\,}c@{\;\,}c@{\;\,}c}
\text{tubes} & \raisebox{-.4cm}{\includegraphics[scale=.6]{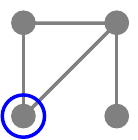}} & \raisebox{-.4cm}{\includegraphics[scale=.6]{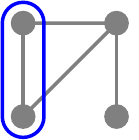}} & \raisebox{-.4cm}{\includegraphics[scale=.6]{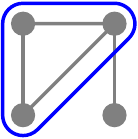}} & \raisebox{-.4cm}{\includegraphics[scale=.6]{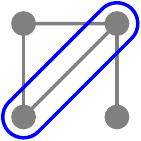}} & \raisebox{-.4cm}{\includegraphics[scale=.6]{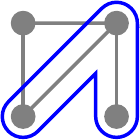}} & \raisebox{-.4cm}{\includegraphics[scale=.6]{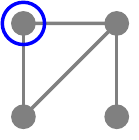}} & \raisebox{-.4cm}{\includegraphics[scale=.6]{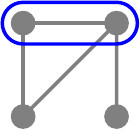}} & \raisebox{-.4cm}{\includegraphics[scale=.6]{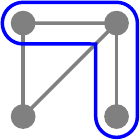}} & \raisebox{-.4cm}{\includegraphics[scale=.6]{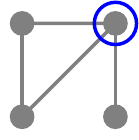}} & \raisebox{-.4cm}{\includegraphics[scale=.6]{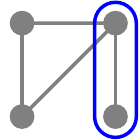}} & \raisebox{-.4cm}{\includegraphics[scale=.6]{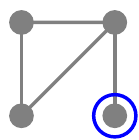}} \\[.6cm]
\text{$\b{g}$-vectors} & \compactVectorT{1}{0}{0} & \compactVectorT{0}{1}{0} & \compactVectorT{0}{0}{1} & \compactVectorT{1}{-1}{1} & \compactVectorT{1}{-1}{0} & \compactVectorT{-1}{1}{0} & \compactVectorT{-1}{0}{1} & \compactVectorT{-1}{0}{0} & \compactVectorT{0}{-1}{1} & \compactVectorT{0}{-1}{0} & \compactVectorT{0}{0}{-1} \\[.6cm]
\text{facet}		& -1 & 1 & -1 & 1 & 0 & 0 & 0 & 0 & 0 & 0 & 0 \\
\text{defining}		& 1 & -1 & 0 & 0 & 0 & 1 & 0 & 0 & 0 & 0 & 0 \\
\text{inequalities}	& 0 & 1 & -1 & 0 & 0 & -1 & 1 & 0 & 0 & 0 & 0 \\
				& 1 & 0 & 0 & -1 & 0 & 0 & 0 & 0 & 1 & 0 & 0 \\
				& 0 & 0 & -1 & 1 & 0 & 0 & 1 & 0 & -1 & 0 & 0 \\
				& 0 & 0 & 0 & 1 & -1 & 0 & 0 & 0 & -1 & 1 & 0 \\
				& 0 & 0 & 0 & 0 & 0 & 1 & -1 & 0 & 1 & 0 & 0 \\
				& 0 & 0 & 0 & 0 & 0 & 0 & 1 & -1 & -1 & 1 & 0 \\
				& 0 & 0 & 0 & 0 & 0 & 0 & 0 & 0 & 1 & -1 & 1 \\
				& 0 & 0 & 0 & 0 & 1 & 0 & 0 & 1 & 0 & -1 & 0 \\
				& 0 & 0 & 1 & 0 & 0 & 0 & -1 & 1 & 0 & 0 & 0 \\
				& 0 & 0 & 1 & -1 & 1 & 0 & 0 & 0 & 0 & 0 & 0 \\[.2cm]
\end{array}
\]

\end{example}

\begin{example}
\label{exm:constructionsGraphicalNestedFan}
We can exploit \cref{coro:typeConeGraphicalNestedFan} to show that certain height functions belong to the type cone of~$\nestedFan[\graphG]$ and recover some classical constructions of the graph associahedron.

\begin{enumerate}[(i)]
\item Consider the height function~$\b{h} \in \R^{\tubes}$ given by~$\b{h}_{\tube} \eqdef -3^{|\tube|}$. Then for any exchangeable tubes~$\tube$ and~$\tube'$, we have
\[
\qquad\qquad
\dotprod{\b{n}(\tube, \tube')}{\b{h}} = - 3^{|\tube|} - 3^{|\tube'|} + 3^{|\tube \cup \tube'|} + \sum_{\tube[s] \in \connectedComponents(\tube \cap \tube')} 3^{|\tube[s]|} \ge - 2 \cdot 3^{|\tube \cup \tube'|-1} + 3^{|\tube \cup \tube'|} > 0.
\]
Therefore, the height function~$\b{h}$ belongs to the type cone~$\typeCone(\nestedFan[\graphG])$.
The corresponding polytope~$P_\b{h} \eqdef \bigset{\b{x} \in \R^\ground}{\dotprod{\gvector{\tube}}{\b{x}} \le \b{h}_{\tube} \text{ for } \tube \in \tubes}$ is the graph associahedron constructed by S.~Devadoss's in~\cite{Devadoss}.
\item Consider the height function~$\b{h} \in \R^{\tubes}$ given by~$\b{h}_{\tube} \eqdef -|\bigset{\tube[s] \in \tubes}{\tube[s] \subseteq \tube}|$. Then for any exchangeable tubes~$\tube$ and~$\tube'$, we have
\[
\qquad\qquad
\dotprod{\b{n}(\tube, \tube')}{\b{h}} = |\bigset{\tube[s]\in\tubes}{\tube[s] \not\subseteq \tube \text{ and } \tube[s] \not\subseteq \tube' \text{ but } \tube[s] \subseteq \tube \cup \tube'}| > 0
\]
since~$\tube \cup \tube'$ fulfills the conditions on~$\tube[s]$.
Therefore, the height function~$\b{h}$ belongs to the type cone~$\typeCone(\nestedFan[\graphG])$.
The corresponding polytope~${P_\b{h} \eqdef \bigset{\b{x} \in \R^\ground}{\dotprod{\gvector{\tube}}{\b{x}} \le \b{h}_{\tube} \text{ for } \tube \in \tubes}}$ is the graph associahedron constructed by A.~Postnikov's in~\cite{Postnikov}.
\end{enumerate}
\end{example}

Note that many inequalities of \cref{coro:typeConeGraphicalNestedFan} are redundant.
In the remaining of this section, we describe the facet-defining inequalities of the type cone of the graphical nested fans.
We say that an exchangeable pair~$\{\tube, \tube'\}$ of tubes of~$\graphG$ is
\begin{itemize}
\item \defn{extremal} if its  corresponding inequality in \cref{coro:typeConeGraphicalNestedFan} defines a facet of~$\typeCone(\nestedFan[\graphG])$,
\item \defn{maximal} if~${\tube \ssm \{v\} = \tube' \ssm \{v'\}}$ for some neighbor~$v$ of~$\tube'$ and some neighbor~$v'$ of~$\tube$.
\end{itemize}
We can now state our main result on graphical nested complexes.

\begin{theorem}
\label{thm:extremalExchangeablePairsGraphicalNestedFan}
An exchangeable pair is extremal if and only if it is maximal.
\end{theorem}

\begin{proof}
We treat separately the two implications:

\para{Extremal $\implies$ maximal}
Consider an exchangeable pair~$\{\tube, \tube'\}$ of tubes of~$\graphG$.
By \cref{prop:exchangeablePairsGraphicalNestedFan}, $\tube'$ has a unique neighbor~$v$ in~$\tube \ssm \tube'$ and $\tube$ has a unique neighbor~$v'$ in~$\tube' \ssm \tube$.
Therefore, $\tube \ssm \tube'$ and~$\tube' \ssm \tube$ are both connected.
Assume that $\{\tube, \tube'\}$ is not maximal, for instance that~$\tube \ssm \tube' \ne \{v\}$, and let $w\ne v$ be a non-disconnecting node of $\tube \ssm \tube'$.
By \cref{prop:exchangeablePairsGraphicalNestedFan}, $\tilde \tube \eqdef \tube \ssm \{w\}$ and $\tube'$ are exchangeable, and $\tilde \tube' \eqdef (\tube \cup \tube') \ssm \{w\}$ and $\tube$ are exchangeable as well.
Moreover, we have
\begin{align*}
\b{n}(\tilde \tube, \tube') + \b{n}( \tube, \tilde \tube')
& = \Big( \b{f}_{\tilde \tube} + \b{f}_{\tube'} - \b{f}_{\tilde \tube\cup \tube'} - \sum_{\tube[s] \in \connectedComponents(\tilde \tube \cap \tube')} \b{f}_{\tube[s]} \Big)
+ \Big( \b{f}_{\tube} + \b{f}_{\tilde \tube'} - \b{f}_{\tube\cup \tilde \tube'} - \sum_{\tube[s] \in \connectedComponents(\tube \cap \tilde \tube')} \b{f}_{\tube[s]} \Big) \\
& = \b{f}_{\tube} + \b{f}_{\tube'} - \b{f}_{\tube\cup \tube'} - \sum_{\tube[s] \in \connectedComponents(\tube \cap \tube')} \b{f}_{\tube[s]} =\b{n}(\tube, \tube'),
\end{align*}
as $\tilde \tube \cup \tube' = \tilde \tube'$, $\tilde \tube \cap \tube' = \tube \cap \tube'$, $\tube \cup \tilde \tube' = \tube \cup \tube'$ and $\connectedComponents(\tube \cap \tilde \tube') = \connectedComponents(\tilde \tube) = \tilde \tube$.
Therefore $\b{n}(\tube, \tube')$ defines a redundant inequality and $\{\tube, \tube'\}$ is not an extremal exchangeable pair.
The proof is symmetric if~$\tube' \ssm \tube \ne \{v'\}$.

\para{Maximal $\implies$ extremal}
Let $\{\tube, \tube'\}$ be a maximal exchangeable pair.
To prove that $\{\tube, \tube'\}$ is extremal, we will construct a vector $\b{w} \in \R^{\tubes}$ such that $\dotprod{\b{n}(\tube, \tube')}{\b{w}}< 0$, but $\dotprod{\b{n}(\tilde \tube, \tilde \tube')}{\b{w}}>0$ for any other maximal exchangeable pair $\{\tilde \tube, \tilde \tube'\}$. 
This will show that the inequality induced by~$\{\tube, \tube'\}$ is not redundant.

Define $\alpha(\tube,\tube') \eqdef \set{\tube[s] \in \tubes}{\tube[s] \not\subseteq \tube \text{ and } \tube[s] \not\subseteq \tube' \text{ but } \tube[s] \subseteq \tube \cup \tube'}$.
Note that~$\alpha(\tube,\tube')$ is non-empty since it contains~$\tube \cup \tube'$.
Define three vectors~$\b{x}, \b{y}, \b{z} \in \R^{\tubes}$~by
\begin{align*}
\b{x}_{\tube[s]} & \eqdef -|\bigset{\tube[r] \in \tubes \ssm \alpha(\tube, \tube')}{\tube[r] \subseteq \tube[s]}|,
\\
\b{y}_{\tube[s]} & \eqdef -|\bigset{\tube[r] \in \alpha(\tube, \tube')}{\tube[r] \subseteq \tube[s]}|,
\\
\b{z}_{\tube[s]} & \eqdef \begin{cases}
	-1 & \text{if } \tube \subseteq \tube[s] \text{ or } \tube' \subseteq \tube[s], \\
	0 & \text{otherwise},
\end{cases}
\end{align*}
for each tube~$\tube[s] \in \tubes$.

We will prove below that their scalar products with~$\b{n}(\tilde \tube, \tilde \tube')$ for any maximal exchangeable pair $\{\tilde \tube, \tilde \tube'\}$ satisfy the following inequalities
\[
\renewcommand{\arraystretch}{1.3}
\begin{array}{l|ccc}
& \dotprod{\b{n}(\tilde \tube, \tilde \tube')}{\b{x}} & \dotprod{\b{n}(\tilde \tube, \tilde \tube')}{\b{y}} & \dotprod{\b{n}(\tilde \tube, \tilde \tube')}{\b{z}} \\
\hline
\text{if } \{\tube, \tube'\} = \{\tilde \tube, \tilde \tube'\} & = 0 & = |\alpha(\tube, \tube')| & = -1 \\
\text{if } \alpha(\tilde \tube, \tilde \tube') \not\subseteq \alpha(\tube,\tube') & \ge 1 & \ge 0 & \ge -1 \\
\text{otherwise} & = 0 & \ge 1 & \ge 0
\end{array}
\renewcommand{\arraystretch}{1}
\]
It immediately follows from this table that the vector $\b{w} \eqdef \b{x} + \delta \b{y} + \varepsilon \b{z}$ fulfills the desired properties for any~$\delta, \varepsilon$ such that~$0 < \delta \cdot |\alpha(\tube, \tube')| < \varepsilon < 1$.

To prove the inequalities of the table, observe that for any maximal exchangeable pair~$\{\tilde \tube, \tilde \tube'\}$,
\begin{itemize}
\item $\dotprod{\b{n}(\tilde \tube, \tilde \tube')}{\b{x}} = |\alpha(\tilde \tube, \tilde \tube') \ssm \alpha(\tube, \tube')|$,
\item $\dotprod{\b{n}(\tilde \tube, \tilde \tube')}{\b{y}} = |\alpha(\tilde \tube, \tilde \tube') \cap \alpha(\tube, \tube')|$,
\item $\dotprod{\b{n}(\tilde \tube, \tilde \tube')}{\b{z}} \ge -1$ since~$\b{z}_{\tilde \tube} = -1$ or~$\b{z}_{\tilde \tube'} = -1$ implies~$\b{z}_{\tilde \tube \cup \tilde \tube'} = -1$,
\item $\dotprod{\b{n}(\tilde \tube, \tilde \tube')}{\b{z}} \ge 0$ when $\{\tube, \tube'\} \ne \{\tilde \tube, \tilde \tube'\}$ but $\alpha(\tilde \tube, \tilde \tube') \subseteq \alpha(\tube,\tube')$. Indeed $\alpha(\tilde \tube, \tilde \tube') \subseteq \alpha(\tube,\tube')$ implies~$\tilde \tube \cup \tilde \tube' \subseteq \tube \cup \tube'$. If~$\tube \subseteq \tilde \tube$, then~$\tube \subseteq \tilde \tube \subsetneq \tilde \tube \cup \tilde \tube' \subseteq \tube \cup \tube'$, which implies that~${\tube = \tilde \tube}$ by maximality of~$\tube$ in~$\tube \cup \tube'$. Similarly, $\tube' \subseteq \tilde \tube$ implies $\tube' = \tilde \tube$. Hence, if~$\b{z}_{\tilde \tube} = -1$, then by definition~$\tube \subseteq \tilde \tube$ or~$\tube' \subseteq \tilde \tube$, which implies that~$\tilde \tube \in \{\tube, \tube'\}$. Similarly~$\b{z}_{\tilde \tube'} = -1$ implies~${\tilde \tube' \in \{\tube, \tube'\}}$. Hence, $\b{z}_{\tilde \tube} = -1 = \b{z}_{\tilde \tube'}$ implies~$\tilde \tube = \tilde \tube'$ (impossible since~$\tilde \tube$ and~$\tilde \tube'$ are exchangeable) or~$\{\tube, \tube'\} = \{\tilde \tube, \tilde \tube'\}$ (contradicting our assumption). Therefore, at most one of~$\b{z}_{\tilde \tube}$ and~$\b{z}_{\tilde \tube'}$ equals to~$-1$, and if exactly one does, then~$\b{z}_{\tilde \tube \cup \tilde \tube'} = -1$. We conclude that~$\dotprod{\b{n}(\tilde \tube, \tilde \tube')}{\b{z}} \ge 0$.
\qedhere
\end{itemize}
\end{proof}

The following statement reformulates \cref{thm:extremalExchangeablePairsGraphicalNestedFan}.

\begin{corollary}
\label{coro:extremalExchangeablePairsGraphicalNestedFan}
The extremal exchangeable pairs for the nested fan of~$\graphG$ are precisely the pairs of tubes~${\tube[s] \ssm \{v'\}}$ and~${\tube[s] \ssm \{v\}}$ for any tube~$\tube[s] \in \tubes$ and distinct non-disconnecting vertices~$v,v'$~of~$\tube[s]$.
\end{corollary}

We derive from \cref{thm:extremalExchangeablePairsGraphicalNestedFan,coro:extremalExchangeablePairsGraphicalNestedFan} the irredundant facet description of the type cone~$\typeCone(\nestedFan[\graphG])$.

%\begin{corollary}
%\label{coro:facetDescriptionTypeConeGraphicalNestedFan}
%Consider a graph~$\graphG$ on~$\ground$ with tubes~$\tubes$. 
%Then, for a height vector~$\b{h} \in \R^{\tubes}$, the nested fan~$\nestedFan[\graphG]$ is the normal fan of the graph associahedron
%\[
%\set{\b{x} \in \R^\ground}{\dotprod{\gvector{\tube}}{\b{x}} \le \b{h}_{\tube} \text{ for any tube } \tube \in \tubes}
%\]
%if and only if~$\b{h}_{\varnothing} = 0 = \b{h}_{K}$ for any~$K \in \connectedComponents(\graphG)$ and $\b{h}_{\tube[s] \ssm \{v'\}} + \b{h}_{\tube[s] \ssm \{v\}} > \b{h}_{\tube[s]} + \b{h}_{\tube[s] \ssm \{v, v'\}}$ for any tube~$\tube[s] \in \tubes$ and distinct non-disconnecting vertices~$v,v'$ of~$\tube[s]$.
%\end{corollary}

% \begin{corollary}
% \label{coro:facetDescriptionTypeConeGraphicalNestedFan}
% The inequalities $\b{h}_{\tube[s] \ssm \{v'\}} + \b{h}_{\tube[s] \ssm \{v\}} > \b{h}_{\tube[s]} + \b{h}_{\tube[s] \ssm \{v, v'\}}$, for any tube~$\tube[s] \in \tubes$ and any distinct non-disconnecting vertices~$v,v'$ of~$\tube[s]$, provide an irredundant facet description of~$\typeCone(\nestedFan[\graphG])$.
% \end{corollary}

\begin{corollary}
\label{coro:facetDescriptionTypeConeGraphicalNestedFan}
For any graph~$\graphG$, the type cone of the nested fan~$\nestedFan[\graphG]$ is given by the following irredundant facet description
\[
\typeCone(\nestedFan[\graphG]) = \set{\b{h} \in \R^{\tubes}}{\begin{array}{l} \b{h}_{K} = 0 \text{ for any connected component } K \in \connectedComponents(\graphG) \text{, and}\\ 
\b{h}_{\tube[s] \ssm \{v'\}} + \b{h}_{\tube[s] \ssm \{v\}} > \b{h}_{\tube[s]} + \b{h}_{\tube[s] \ssm \{v, v'\}} \text{ for any tube } \tube[s] \in \tubes \\ \text{and distinct non-disconnecting vertices } v, v' \in \tube[s] \end{array}}.
\]\end{corollary}

\begin{remark}
For instance,
\begin{enumerate}[(i)]
\item for the complete graph~$\graphG[K]_n$, all the inequalities~$\b{h}_{U \ssm \{v\}} + \b{h}_{U \ssm \{v'\}} > \b{h}_{U} + \b{h}_{U \ssm \{v,v'\}}$ for~${v,v' \in \ground}$ and~$\{v,v'\} \subseteq U \subseteq \ground$ are facet defining inequalities of~$\typeCone(\nestedFan[{\graphG[K]_n}])$.
\item for the path~$\graphG[P]_n$, only the inequalities $\b{h}_{[i,j-1]} + \b{h}_{[i+1,j]} > \b{h}_{[i,j]} + \b{h}_{[i+1,j-1]}$ for~$1 \le i < j \le n$ are facet defining inequalities of~$\typeCone(\nestedFan[{\graphG[P]_n}])$ (where~$\b{h}_\varnothing = 0$ by convention).
\end{enumerate}
\end{remark}

We derive from \cref{coro:extremalExchangeablePairsGraphicalNestedFan} the number of facets of the type cone~$\typeCone(\nestedFan[\graphG])$.
For a tube~$\tube$ of~$\graphG$, we denote by~$\nonDisconnecting(\tube)$ the number of non-disconnecting vertices of~$\tube$.
In other words, $\nonDisconnecting(\tube)$ is the number of tubes covered by~$\tube$ in the inclusion poset of all tubes of~$\graphG$.

\begin{corollary}
\label{coro:numberFacetsTypeConeGraphicalNestedFan}
The type cone~$\typeCone(\nestedFan[\graphG])$ has~$\sum\limits_{\tube[s] \in \tubes} \binom{\nonDisconnecting(\tube[s])}{2}$ facets.
\end{corollary}

The formula of \cref{coro:numberFacetsTypeConeGraphicalNestedFan} can be made more explicit for specific families of graph associahedra discussed in the introduction and illustrated in \cref{fig:specialGraphAssociahedra}.

\begin{proposition}
The number of facets of the type cone~$\typeCone(\nestedFan[\graphG])$ is:
\begin{itemize}
\item $2^{n-2}\binom{n}{2}$ for the permutahedron (complete graph associahedron),
\item $\binom{n}{2}$ for the associahedron (path associahedron),
\item $3\binom{n}{2} - n$ for the cyclohedron (cycle associahedron),
\item $n-1+2^{n-3}\binom{n-1}{2}$ for the stellohedron (star associahedron).
\end{itemize}
\end{proposition}

\begin{proof}
For the permutahedron, choose any two vertices~$v,v'$, and complete them into a tube by selecting any subset of the~$n-2$ remaining vertices.
For the associahedron, choose any two vertices~$v,v'$, and complete them into a tube by taking the path between them.
For the cyclohedron, choose the two vertices~$v,v'$, and complete them into a tube by taking either all the cycle, or one of the two paths between~$v$ and~$v'$ (this gives three options in general, but only two when~$v,v'$ are neighbors).
For the stellohedron, choose either~$v$ as the center of the star and~$v'$ as one of the $n-1$ leaves, or $v$~and~$v'$ as leaves of the star and complete them into a tube by taking the center and any subset of the $n-3$ remaining leaves.
\end{proof}

%%%%%%%%%%%

\subsection{Simplicial type cone}
\label{subsec:simplicialTypeConeGraphicalNestedFans}

To conclude on graphical nested fans, we characterize the graphs~$\graphG$ whose nested fan~$\nestedFan[\graphG]$ has a simplicial type cone.

\begin{proposition}
\label{prop:simplicialTypeConeGraphicalNestedFan}
The type cone~$\typeCone(\nestedFan[\graphG])$ is simplicial if and only if~$\graphG$ is a disjoint union~of~paths.
\end{proposition}

\begin{proof}
Observe first that the graphical nested fan~$\nestedFan[\graphG]$ has~$N = |\tubes| - |\connectedComponents(\graphG)|$ rays and dimension~$n = |\ground|-|\connectedComponents(\graphG)|$.
Moreover, any tube~$\tube$ with~$|\tube| \ge 2$ has two non-disconnecting vertices when it is a path, and at least three non-disconnecting vertices otherwise (the leaves of an arbitrary spanning tree of~$\tube$, or any vertex if it is a cycle).
Therefore, each tube of~$\tubes$ which is not a singleton contributes to at least one extremal exchangeable pair.
We conclude that the number of extremal exchangeable pairs is at least
\[
|\tubes| - |\ground| = (|\tubes| - |\connectedComponents(\graphG)|) - (|\ground|-|\connectedComponents(\graphG)|) = N - n,
\]
with equality if and only if all tubes of~$\graphG$ are paths, \ie if and only if~$\graphG$ is a collection of paths.
Hence, $\typeCone(\nestedFan[\graphG])$ is simplicial if and only if~$\graphG$ is a disjoint union of paths.
\end{proof}

The motivation to study the simpliciality of the type cone~$\typeCone(\nestedFan[\graphG])$ stems from the \defn{kinematic associahedra} of~\cite[Sect.~3.2]{ArkaniHamedBaiHeYan}.
These polytopes are alternative realizations of the associahedron obtained as sections of the kinematic space (the positive orthant in~$\smash{\R^{\binom{n}{2}}}$) by a well-chosen affine subspace parametrized by positive vectors.
While these polytopes are just affinely equivalent to the realizations in~$\R^\ground$, they have the advantage of being more natural from the scattering amplitudes perspective~\cite{ArkaniHamedBaiHeYan}.
As observed in~\cite{PadrolPaluPilaudPlamondon}, such realizations can be directly obtained from the facet description of the type cone, when the latter is simplicial.
Hence, combining \cref{prop:simplicialTypeCone,prop:simplicialTypeConeGraphicalNestedFan,coro:facetDescriptionTypeConeGraphicalNestedFan} produces kinematic realizations of all graph associahedra of disjoint union of paths (\ie all Cartesian products of associahedra).
Our next statement only recalls the construction of the kinematic associahedron as it serves as a prototype for \cref{prop:kinematicNestohedraInterval} that will describe new families of kinematic nestohedra.

\begin{proposition}
\label{prop:kinematicAssociahedra}
For any~$\b{p} \in \R_{>0}^{\binom{[n]}{2}}$, the polytope~$R_\b{p}(n)$ defined as the intersection of the positive orthant~$ \set{\b{z} \in \R^{\set{[i,j]}{1 \le i \le j \le n}}}{\b{z} \ge 0}$ with the hyperplanes
\begin{itemize}
\item $\b{z}_{[1,n]} = 0$ and $\b{z}_{[i,i]} = 0$ for~$i \in [n]$,
\item $\b{z}_{[i,j-1]} + \b{z}_{[i+1,j]} - \b{z}_{[i,j]} + \b{z}_{[i+1,j-1]} = \b{p}_{[i,j]}$ for all~$1 \le i < j \le n$,
\end{itemize}
is an associahedron whose normal fan is~$\nestedFan[{\graphG[P]_n}]$.
Moreover, the polytopes~$R_\b{p}(n)$ for~${\b{p} \in \R_{>0}^{\binom{[n]}{2}}}$ describe all polytopal realizations of~$\nestedFan[{\graphG[P]_n}]$ (up to translations).
\end{proposition}

%We also obtain from \cref{prop:maximalImpliesExtremalGraphical} the following surprising property.
%
%\begin{corollary}
%Any $\b{c}$-vector supports at least one extremal exchangeable pair.
%\end{corollary}
%
%\begin{proof}
%Consider a $\b{c}$-vector~$\b{e}_v - \b{e}_{v'}$ for two distinct vertices~$v, v'$ in a common connected component of~$\graphG$. Let~$\tube[r]$ be a path from~$v$ to~$v'$ in~$\graphG$ and let~$\tube \eqdef \tube[r] \cup \{v\}$ and~$\tube' \eqdef \tube[r] \cup \{v'\}$. Then $\{\tube, \tube'\}$ is an extremal exchangeable pair with $\b{c}$-vector~$\b{e}_v - \b{e}_{v'}$.
%\end{proof}

%%%%%%%%%%%%%%%%%%%%%%%%%%%%%%%%%%%%%%%

\section{Type cones of arbitrary nested fans}
\label{sec:typeConeNestedFan}

We now extend our results from graph associahedra to nestohedra.
In the general situation, the set of tubes is replaced by a building set~$\building$, and the tubings are replaced by $\building$-nested sets (this generalization can equivalently be interpreted as replacing the graph by an arbitrary hypergraph). 
As in the graphical case, the nested sets define a nested complex and a nested fan, which is the normal fan of the nestohedron.
In this section, we describe the type cones of arbitrary nested fans.
We follow the same scheme as in \cref{sec:typeConeGraphicalNestedFan}, even if the general situation is significantly more intricate (\cref{rem:differencesGraphicalExchangeables,rem:differencesGraphicalExchangeRelation} highlight some of the complications of the general case).

%%%%%%%%%%%

\pagebreak
\subsection{Nested complex, nested fan, and nestohedron}
\label{subsec:nestedComplexNestedFan}

We first recall the definitions of arbitrary building sets, nested complexes, nested fans and nestohedra, following~\cite{Postnikov, FeichtnerSturmfels, Zelevinsky, Pilaud-removahedra}.

\subsubsection{Building sets}
A \defn{building set}~$\building$ on a ground set~$\ground$ is a set of non-empty subsets of~$\ground$ such that
\begin{itemize}
\item if~$B,B' \in \building$ and~$B \cap B' \ne \varnothing$, then~$B \cup B' \in \building$, and
\item $\building$ contains all singletons~$\{v\}$ for~$v \in \ground$.
\end{itemize} 
We denote by~$\connectedComponents(\building)$ the set of \defn{connected components} of~$\building$, defined as the (inclusion) maximal elements of~$\building$.
We denote by~$\elementary(\building)$ the set of \defn{elementary blocks} of~$\building$, defined as the blocks~$B \in \building$ such that~$|B| > 1$, and $B = B' \cup B''$ implies~$B' \cap B'' = \varnothing$ for any~$B', B'' \in \building \ssm \{B\}$.
For instance, consider the building set~$\building_\circ$ on~$[9]$ defined by
\[
\building_\circ \eqdef \{1, 2, 3, 4, 5, 6, 7, 8, 9, 14, 25, 123, 456, 789, 1234, 1235, 1456, 2456, 12345, 12456, 123456\}
\]
(since all labels have a single digit, we can abuse notation and write $123$ for $\{1,2,3\}$).
Its connected components are $\connectedComponents(\building_\circ) = \{123456, 789\}$, and its elementary blocks are $\elementary(\building_\circ) = \{14, 25, 123, 456, 789\}$, which are represented in \cref{fig:exmNested}\,(left).

\begin{remark}
\label{rem:elementary}
If~$B \in \building$ is elementary, then the maximal blocks of~$\building$ strictly contained in~$B$ are disjoint.
Conversely, if there exist two disjoint maximal blocks~$M, N \in \building$ strictly contained in~$B \in \building$, then~$B$ is elementary.
Otherwise, there would be~$B', B'' \in \building \ssm \{B\}$ such that $B = B' \cup B''$ and $B' \cap B'' \ne \varnothing$.
By maximality, $M$ and~$N$ are not strict subsets of~$B'$ and~$B''$, hence $M$ and~$N$ intersect both~$B'$ and~$B''$.
Since~$M \cap B' \ne \varnothing$, we have~$M \cup B' \in \building$.
As $M \subseteq M \cup B' \subseteq B$, we obtain again by maximality of~$M$ that~$M = M \cup B'$ or $M \cup B' = B$.
In the former case, we have~$\varnothing \ne B' \cap N \subseteq M \cap N$ contradicting our assumption on~$M$ and~$N$.
In the latter case, we have~$N \subseteq B' \ssm M$ contradicting the maximality of~$N$.
\end{remark}

\begin{remark}
\label{rem:graphicalBuildingSet}
For a graph~$\graphG$ with vertex set~$\ground$, the set~$\tubes$ of all tubes of~$\graphG$ is a \defn{graphical building set}.
The blocks of~$\connectedComponents(\tubes)$ are the vertex sets of the connected components~$\connectedComponents(\graphG)$ of~$\graphG$, and the blocks of~$\elementary(\tubes)$ are the edges of~$\graphG$.
\end{remark}

\begin{remark}
\label{rem:hypergraphs}
Note that not all building sets are graphical building sets.
It was in fact proved in~\cite[Prop.~7.3]{Zelevinsky} that a building set is graphical if and only if for any~$B \in \building$ and~$\c{C} \subset \building$, if~$B \cup \bigcup \c{C} \in \building$, then there is~$C \in \c{C}$ such that~$B \cup C \in \building$.
However, arbitrary building sets can be interpreted using hypergraphs~\cite{Berge} instead of graphs.
More precisely, a hypergraph~$\hypergraph$ on~$\ground$ defines a building set~$\building\hypergraph$ on~$\ground$ given by all non-empty subsets of~$\ground$ which induce connected subhypergraphs of~$\hypergraph$ (a path in~$\hypergraph$ is a sequence of vertices where any two consecutive ones belong to a common hyperedge of~$\hypergraph$).
Conversely, a building set~$\building$ on~$\ground$ is the building set of various hypergraphs on~$\ground$, all containing the hypergraph with hyperedge set~$\elementary(\building)$.
See \cite{DosenPetric} for details.
\end{remark}

\subsubsection{Nested complex}
Given a building set~$\building$, a \defn{$\building$-nested set}~$\nested$ is a subset of~$\building$ such that
\begin{itemize}
\item for any~$B,B' \in \nested$, either~$B \subseteq B'$ or~$B' \subseteq B$ or~$B \cap B' = \varnothing$,
\item for any~$k \ge 2$ pairwise disjoint~$B_1,\dots,B_k \in \nested$, the union~$B_1 \cup \dots \cup B_k$ is not in~$\building$, and
\item $\nested$ contains~$\connectedComponents(\building)$.
\end{itemize}
These are the original conditions that appeared for instance in~\cite{Postnikov}.
In this paper, we prefer to use the following convenient reformulation, similar to that of~\cite{Zelevinsky}:  $\nested \subseteq \building$ is a $\building$-nested set if and only if~$\connectedComponents(\building) \subseteq \nested$ and the union~$\bigcup \nested[X]$ of any subset~$\nested[X] \subseteq \nested$ does not belong to~$\building \ssm \nested[X]$.
% any~$B_1, \dots, B_k \in \nested$ such that~$B_1 \cup \dots \cup B_k \in \building$, there is~$i \in [k]$ such that~$B_j \subseteq B_i$ for all~$j \in [k]$.
%
It is known that all inclusion maximal nested sets have~$|V|$ blocks.
The \defn{$\building$-nested complex} is the simplicial complex~$\nestedComplex(\building)$ whose faces are $\nested \ssm \connectedComponents(\building)$ for all $\building$-nested sets~$\nested$.
It is a simplicial sphere of dimension~$|\ground| - |\connectedComponents(\building)|$.
Note that it is convenient to include~$\connectedComponents(\building)$ in all $\building$-nested sets as in~\cite{Postnikov} for certain combinatorial manipulations, but to remove $\connectedComponents(\building)$ from all $\building$-nested sets as in~\cite{Zelevinsky} when defining the $\building$-nested complex.
If~$\nested \ssm \{B\} = \nested' \ssm \{B'\}$ for two maximal $\building$-nested sets~$\nested$ and~$\nested'$ and two building blocks~$B$ and~$B'$, we say that~$\nested$ and $\nested'$ are \defn{adjacent} and that~$B$ and~$B'$ are~\defn{exchangeable}.

For instance, \cref{fig:exmNested}\,(middle) represents the two adjacent maximal $\building_\circ$-nested sets
\[
\nested_\circ \eqdef \{3, 4, 5, 7, 8, 14, 789, 12345, 123456\}
\quad\text{and}\quad
\nested'_\circ \eqdef \{3, 4, 5, 7, 8, 25, 789, 12345, 123456\}.
\]

\begin{figure}
	\capstart
	\centerline{\includegraphics[scale=.9]{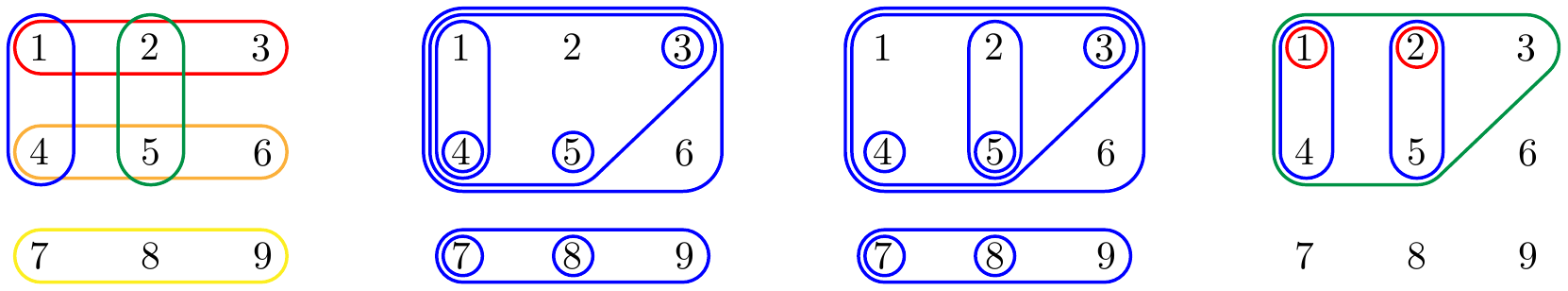}}
	\caption{The elementary blocks of a building set~$\building_\circ$ (left), two adjacent maximal $\building_\circ$-nested sets (middle), and the corresponding frame (right).}
	\label{fig:exmNested}
\end{figure}

\begin{remark}
For a graph~$\graphG$, a set of tubes of~$\building\graphG$ is nested if and only if its tubes are pairwise compatible in the sense of \cref{subsec:graphicalNestedComplexGraphicalNestedFan} (either nested or non-adjacent).
The nested complex~$\nestedComplex(\building\graphG)$ thus coincides with the graphical nested complex~$\nestedComplex(\graphG)$ introduced in \cref{subsec:graphicalNestedComplexGraphicalNestedFan} (which justifies our notation there).
Note that, in contrast to the graphical nested complexes, not all nested complexes are flag (\ie clique complexes of their graphs).
\end{remark}

For a $\building$-nested set~$\nested$ and~$B \in \nested$, we call \defn{root} of~$B$ in~$\nested$ the set~$\rootset{B}{\nested} \eqdef B \ssm \bigcup_{C} C$ where the union runs over~$C \in \nested$ such that~$C \subsetneq B$.
The $\building$-nested set~$\nested$ is maximal if and only if all~$\rootset{B}{\nested}$ are singletons for~$B \in \nested$.
In that case, we abuse notation writing~$\rootset{B}{\nested}$ for the only element of this singleton.
For instance, in the maximal $\building_\circ$-nested sets~$\nested_\circ$ and~$\nested'_\circ$ represented in \cref{fig:exmNested}\,(middle), we have $\rootset{14}{\nested_\circ} = 1 = \rootset{12345}{\nested'_\circ}$ and $\rootset{12345}{\nested_\circ} = 2 = \rootset{25}{\nested'_\circ}$.

\subsubsection{Nested fan}
We still denote by~$(\b{e}_v)_{v \in \ground}$ the canonical basis of~$\R^\ground$.
We consider the subspace~${\HH \eqdef \bigset{\b{x} \in \R^\ground}{\sum_{v \in B} x_v = 0 \text{ for all } B \in \connectedComponents(\building)}}$ and let~$\pi : \R^\ground \to \HH$ denote the orthogonal projection onto~$\HH$.
The \defn{$\b{g}$-vector} of a building block~$B$ of~$\building$ is the projection~$\gvector{B} \eqdef \pi \big( \sum_{v \in B} \b{e}_v \big)$ of the characteristic vector of~$B$.
We set~$\gvectors{\nested} \eqdef \set{\gvector{B}}{B \in \nested}$ for a $\building$-nested set~$\nested$.
Note that by definition, $\gvector{K} = \b{0}$ for all connected components~$K \in \connectedComponents(\building)$.
The vectors~$\gvector{B}$ with~$B \in \building$ support a complete simplicial fan realization of the nested complex.
See~\cref{fig:nestedFans}.

\begin{theorem}[\cite{Postnikov, FeichtnerSturmfels, Zelevinsky}]
\label{thm:nestedFan}
For any building set~$\building$, the set of cones
\[
\nestedFan[\building] \eqdef \set{\R_{\ge 0} \, \gvectors{\nested}}{\nested \text{ nested set of } \building}
\]
is a complete simplicial fan of~$\HH$, called the \defn{nested fan} of~$\building$, which realizes the nested complex~$\nestedComplex(\building)$.
\end{theorem}

\begin{remark}
For a graph~$\graphG$, the nested fan~$\nestedFan[\building\graphG]$ coincides with the graphical nested fan~$\nestedFan[\graphG]$ introduced in \cref{subsec:graphicalNestedComplexGraphicalNestedFan} (which justifies our notation there).
\end{remark}

\begin{figure}[h]
	\capstart
	\centerline{\includegraphics[scale=.55]{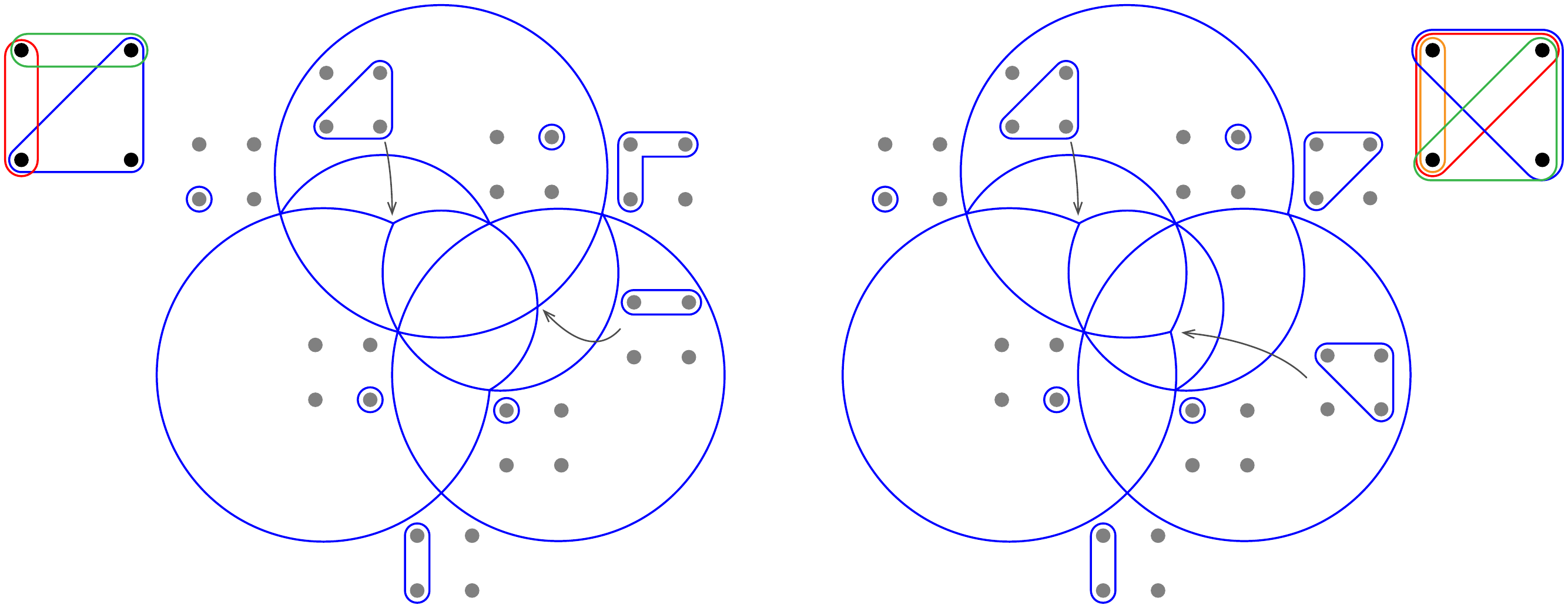}}
	\caption{Two nested fans. The rays are labeled by the corresponding blocks. As the fans are $3$-dimensional, we intersect them with the sphere and stereographically project them from the direction~$(-1,-1,-1)$.}
	\label{fig:nestedFans}
\end{figure}

\subsubsection{Nestohedron}
Again, the $\building$-nested fan is always the normal fan of a polytope, as shown in~\cite{Postnikov, FeichtnerSturmfels, Zelevinsky}.
We still denote by ~$\triangle_U \eqdef \conv\set{\b{e}_u}{u \in U}$ the face of the standard simplex~$\triangle_\ground$ corresponding to a subset~$U$ of~$\ground$.

\begin{figure}
	\capstart
	\centerline{\includegraphics[scale=.55]{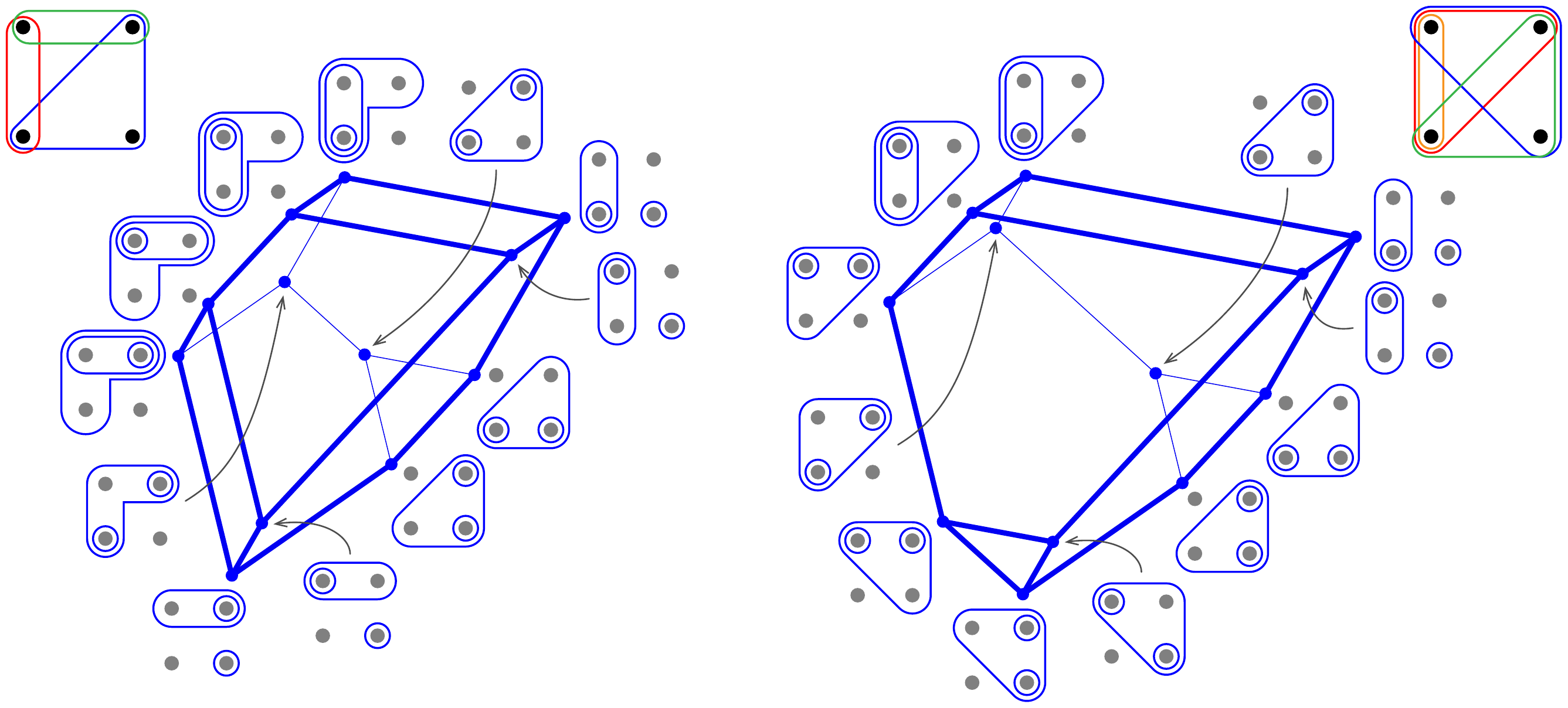}}
	\caption{Two nestohedra, realizing the nested fans of \cref{fig:nestedFans}. The vertices are labeled by the corresponding maximal nested sets.}
	\label{fig:nestohedra}
\end{figure}

\begin{theorem}[\cite{Postnikov, FeichtnerSturmfels, Zelevinsky}]
\label{thm:nestohedron}
For any building set~$\building$, the nested fan~$\nestedFan[\building]$ is the normal fan of a polytope.
For instance, $\nestedFan[\building]$ is the normal fan of
\begin{enumerate}[(i)]
\item the intersection of~$\HH$ with the hyperplanes~$\dotprod{\gvector{B}}{\b{x}} \le -3^{|B|}$ for all~$B \in \building$~\mbox{\cite{Devadoss, Pilaud-removahedra},}
\item the Minkowski sum~$\sum_{B \in \building} \triangle_B$ of the faces of the standard simplex given by all blocks of~$\building$~\cite{Postnikov}.
\end{enumerate}
\end{theorem}

\begin{definition}
Any polytope whose normal fan is the nested fan~$\nestedFan[\building]$ is called a \defn{nestohedron} of~$\building$ and denoted by~$\Nest$.
\end{definition}

For instance, \cref{fig:nestohedra} represents the nestohedra realizing the nested fans of \cref{fig:nestedFans} and obtained using the construction (ii) of \cref{thm:nestohedron}.

\begin{remark}
For a graph~$\graphG$, the nestohedra of~$\building\graphG$ are the graph associahedra of~$\graphG$.
\end{remark}

\subsubsection{Restrictions and contractions}
Following~\cite{Zelevinsky}, we describe a structural decomposition of links in nested complexes.
For any~$U \subseteq \ground$, define
\begin{itemize}
\item the \defn{restriction} of~$\building$ to~$U$ as the building set~$\building_{|U} \eqdef \set{B \in \building}{B \subseteq U}$,
\item the \defn{contraction} of~$U$ in~$\building$ as the building set~$\building_{/U} \eqdef \set{C \subseteq \ground \ssm U}{C \in \building \text{ or } C \cup U \in \building}$.
\end{itemize}
% Note that $\building_{|U}$ (resp.~$\building_{/U}$) is a building set, and that the $\building_{|U}$-nested sets (resp.~the $\building_{/U}$-nested sets) are precisely the $\building$-sets included in~$\building_{|U}$ (resp.~in~$\building_{/U}$).
% These restriction and contraction operations enable to describe links in nested complexes.

\begin{proposition}[{\cite[Prop.~3.2]{Zelevinsky}}]
\label{prop:links}
For $U \in \building \ssm \connectedComponents(\building)$, the link~$\set{C \subseteq \building \ssm \{U\}}{C \cup \{U\} \in \nestedComplex(\building)}$ is isomorphic to the Cartesian product~$\nestedComplex(\building_{|U}) \times \nestedComplex(\building_{/U})$.
\end{proposition}

In particular, two building blocks $B$ and~$B'$ in~$U$ (resp.~in~$\ground \ssm U$) are exchangeable in~$\nestedComplex(\building)$ if and only if they are exchangeable in~$\nestedComplex(\building_{|U})$ (resp.~in~$\nestedComplex(\building_{/U})$).

Slightly abusing notation when~$\building$ is clear from the context, we define the \defn{connected components} of~$U$ as~$\connectedComponents(U) \eqdef \connectedComponents(\building_{|U})$.
For instance, for the building set~$\building_\circ$ whose elementary blocks are represented in \cref{fig:exmNested}\,(left) and $U = \{1, 2, 4, 5, 7, 8\}$, we have~${\building_\circ}_{|U} = \{1, 2, 4, 5, 7, 8, 14, 25\}$ so that~$\connectedComponents(U) = \{14, 25, 7, 8\}$.
Note that the definition of building sets implies that
\begin{itemize}
\item for any~$U \subseteq \ground$, the connected components $\connectedComponents(U)$ define a partition of~$U$,
\item for any~$U, U' \subseteq \ground$ such that~$U \cap U' = \varnothing$ and there is no~$B \in \building$ with~$B \subseteq U \sqcup U'$ and~$U \cap B \ne \varnothing \ne U' \cap B$, we have~$\connectedComponents(U \sqcup U') = \connectedComponents(U) \sqcup \connectedComponents(U')$.
\end{itemize}

%%%%%%%%%%%

\subsection{Exchangeable building blocks and exchange frames}
\label{subsec:exchangeableBuildingBlocks}

We now provide an analogue of \cref{prop:exchangeablePairsGraphicalNestedFan} characterizing the exchangeable blocks for arbitrary building sets.
The situation is however much more technical, as highlighted in \cref{rem:differencesGraphicalExchangeables,rem:differencesGraphicalExchangeRelation}.
We start with two useful lemmas.

\begin{lemma}
\label{lem:minimalElementContainingBlock}
For any $\building$-nested set~$\nested$ and any block~$B \in \building \ssm \connectedComponents(\building)$, the set~$\bigset{C \in \nested}{B \subsetneq C}$ admits a unique (inclusion) minimal element~$M$.
Moreover, if $B \notin \nested$, then~$M$ is also the unique (inclusion) maximal element of~$\bigset{C \in \nested}{\rootset{C}{\nested} \cap B \ne \varnothing}$.
\end{lemma}

\begin{proof}
Let~$\nested[X] \eqdef \bigset{C \in \nested}{B \subsetneq C}$ and~$\nested[Y] \eqdef \bigset{C \in \nested}{\rootset{C}{\nested} \cap B \ne \varnothing}$.
Note first that neither~$\nested[X]$ nor~$\nested[Y]$ are empty since $\varnothing \ne B \notin \connectedComponents(\building)$. 
Since all elements of~$\nested[X]$ contain~$B$ and~$\nested$ is a \mbox{$\building$-nested} set, $\nested[X]$ forms a chain by inclusion, and thus admits a unique inclusion minimal element~$M$.
% If~$B \in \nested$, we have~$\rootset{M}{\nested} \subsetneq M \ssm B$ by definition.
Moreover, any building block in~$\nested[Y]$ intersects~$B$ so that~$\bigcup \nested[Y] = B \cup \bigcup \nested[Y]$ is in~$\building$.
Hence, $\nested[Y]$ admits a unique maximal element~$M' \eqdef \bigcup \nested[Y]$.
By definition, $B \subseteq M'$.
If~$B \notin \nested$, then~$B \ne M'$ since~$M' \in \nested[Y] \subseteq \nested$.
Hence, $M' \in \nested[X]$.
Moreover, for any~$C \in \nested$ such that~$C \subsetneq M'$, we have~$C \cap \rootset{M'}{\nested} = \varnothing$ so that~$B \not\subseteq C$ and~$C \notin \nested[X]$.
We conclude that~$M' = M$.
\end{proof}

\begin{lemma}
\label{lem:upperBlocksExchange}
If~$\nested$ and~$\nested'$ are two adjacent maximal $\building$-nested sets with~$\nested \ssm \{B\} = \nested \ssm \{B'\}$, then $\set{C \in \nested}{B \subsetneq C} = \set{C' \in \nested'}{B' \subsetneq C'}$.
\end{lemma}

\begin{proof}
Assume for instance that there is~$C \in \nested \cap \nested'$ such that~$B \subsetneq C$ but~$B' \not\subseteq C$.
We then claim that~$\nested \cup \nested'$ would be a $\building$-nested set, contradicting the maximality of~$\nested$ and~$\nested'$.
Consider a subset~$\nested[X]$ of~$\nested \cup \nested'$ whose union~$\bigcup \nested[X]$ is in~$\building$.
If~$B \notin \nested[X]$, then~$\nested[X] \subseteq \nested'$, hence $\bigcup \nested[X]$ is in~$\nested[X]$ as~$\nested'$ is a $\building$-nested set.
Similarly, if~$B' \notin \nested[X]$, then~$\bigcup \nested[X]$ is in~$\nested[X]$.
Assume now that both~$B$ and~$B'$ belong to~$\nested[X]$.
Define~$\nested[Y] \eqdef \{C\} \cup \nested[X] \ssm \{B\}$.
Note that $\nested[Y] \subseteq \nested'$ since~$B \notin \nested[Y]$.
Moreover, $\bigcup \nested[Y] = C \cup \bigcup \nested[X]$ belongs to~$\building$ since~$C$ and~$\bigcup \nested[X]$ both belong to~$\building$ and intersect~$B$.
Hence, $\bigcup \nested[Y]$ is in~$\nested[Y]$ since~$\nested'$ is a $\building$-nested set.
Note that~$\bigcup \nested[Y] \ne C$ since~$B' \not\subseteq C$ and~$B' \in \nested[Y]$.
Therefore~$\bigcup \nested[Y]$ is in~$\nested[X]$, and thus~$\bigcup \nested[X] = \bigcup \nested[Y]$ is in~$\nested[X]$.
\end{proof}

For two adjacent maximal $\building$-nested sets~$\nested$ and~$\nested'$ with~$\nested \ssm \{B\} = \nested \ssm \{B'\}$, we say that
\begin{itemize}
\item the unique minimal element~$P$ of $\bigset{C \in \nested}{B \subsetneq C}  = \bigset{C' \in \nested'}{B' \subsetneq C'}$ is the \defn{parent},
\item the vertices~$v \eqdef \rootset{P}{\nested'}$ and~$v' \eqdef \rootset{P}{\nested}$ are the  \defn{pivots}, and
\item the triple $(B, B', P)$ is the  \defn{frame}
\end{itemize}
of the exchange between~$\nested$ and~$\nested'$. Note that the parent is well-defined by \cref{lem:minimalElementContainingBlock,lem:upperBlocksExchange}.
We call an \defn{exchange frame} a triple~$(B, B', P)$ which is the frame of an exchange between two adjacent maximal $\building$-nested sets.
For instance, for the two adjacent maximal $\building_\circ$-nested sets~$\nested_\circ$ and~$\nested'_\circ$ represented in \cref{fig:exmNested}\,(middle), we have~$B = 14$, $B' = 25$, $P = 12345$, $v = 1$ and $v' = 2$.
The corresponding exchange frame is illustrated in \cref{fig:exmNested}\,(right).

We are now ready to characterize the pairs of exchangeable building blocks for arbitrary building sets.
For three blocks~$B, C, P \in \building$, we abbreviate the conditions~$B \cap C \ne \varnothing$ and $C \subseteq P$ but~$C \not\subseteq B$ into the short notation~$\leaving{B}{C}{P}$.
The following statement generalizes \cref{prop:exchangeablePairsGraphicalNestedFan}\,(i).

\begin{proposition}
\label{prop:exchangeablePairsNestedFan}
Two blocks~$B, B' \in \building$ are exchangeable in~$\nestedFan[\building]$ if and only if there exist a block~$P \in \building$, and some vertices~$v \in B \ssm B'$ and~$v' \in B' \ssm B$ such that
\begin{itemize}
\item $B \subsetneq P$ and~$B' \subsetneq P$, and
\item $v' \in C$ for any~$\leaving{B}{C}{P}$ while $v \in C'$ for any~$\leaving{B'}{C'}{P}$.
\end{itemize}
\end{proposition}

%\begin{proposition}
%A triple~$(B, B', P)$ of blocks of~$\building$ with $B \subsetneq P$ and~$B' \subsetneq P$ is an exchange frame in~$\nestedFan[\building]$ if and only if there exist~$v \in B \ssm B'$ and~$v' \in B' \ssm B$ such that $v' \in C$ for any~$\leaving{B}{C}{P}$ while $v \in C'$ for any~$\leaving{B'}{C'}{P}$.
%\end{proposition}

\begin{proof}
Assume first that~$B$ and~$B'$ are exchangeable.
Let~$\nested$ and~$\nested'$ be two adjacent maximal $\building$-nested sets such that~$\nested \ssm \{B\} = \nested \ssm \{B'\}$.
Let~$P$ be the parent and $v,v'$ be the pivots of this exchange. % as in \cref{def:parentBlockPivotVertices}.
%By \cref{lem:upperBlocksExchange}, we have
%\[
%\nested[P] \eqdef \bigset{C \in \nested}{B \subsetneq C}  = \bigset{C' \in \nested'}{B' \subsetneq C'}.
%\]
%By \cref{lem:minimalElementContainingBlock},~$\nested[P]$ admits a unique inclusion minimal element~$P$.
%Moreover, the root~$v$ of~$P$ in~$\nested'$ 
Note that~$v \in B$ (by \cref{lem:minimalElementContainingBlock}) but $v \notin B'$ (by definition, since~$B' \in \nested'$ and~$B' \subsetneq P$).
Similarly, $v' \in B' \ssm B$.
Consider now a building block~$C$ such that~$\leaving{B}{C}{P}$.
By definition, $B \subsetneq B \cup C \subseteq P$ and~$B \cup C \in \building$.
If~$B \cup C = P$, then~$v' = \rootset{P}{\nested}$ belongs to~$B \cup C$ and thus to~$C$.
If~$B \cup C \ne P$, then~$P$ is the inclusion minimal element of~$\set{D \in \nested}{B \cup C \subsetneq D}$.
Since~$B \cup C \notin \nested$ by minimality of~$P$ in~$\set{D \in \nested}{B \subsetneq D}$, we obtain by \cref{lem:minimalElementContainingBlock} that~$v' = \rootset{P}{\nested}$ belongs to~$B \cup C$ and thus to~$C$.
Similarly, $v \in C'$ for any~$\leaving{B'}{C'}{P}$.

Conversely, consider~$B, B' \in \building$ so that there is~$P \in \building$, $v \in B \ssm B'$ and~${v' \in B' \ssm B}$ satisfying the conditions of \cref{prop:exchangeablePairsNestedFan}.
Let~$U \eqdef P \ssm \{v,v'\}$, and~$\nested[M]$ denote an arbitrary maximal $\building_{|U}$-nested set.
Let~$\nested \eqdef \nested[M] \cup \{B\}$ and~$\nested' \eqdef \nested[M] \cup \{B'\}$.
Consider a subset~$\nested[X]$ of~$\nested$ whose union~$\bigcup \nested[X]$ is in~$\building$.
If~$B \notin \nested[X]$, then~$\nested[X] \subseteq \nested[M]$, hence~$\bigcup \nested[X]$ is in~$\nested[X]$ since~$\nested[M]$ is a $\building_{|U}$-nested set.
If~$B \in \nested[X]$, since $B \cap \bigcup \nested[X] \ne \varnothing$ and~$\bigcup \nested[X] \subseteq P$ but~${v' \notin \bigcup \nested[X]}$, the conditions of \cref{prop:exchangeablePairsNestedFan} ensure that~$\bigcup \nested[X] \subseteq B$, so that~$\bigcup \nested[X] = B$ is in~$\nested[X]$.
Hence, $\nested$ is a $\building_{|P}$-nested set.
It is moreover maximal since~${|\nested| = |\nested[M] \cup \{B, P\}| = |\nested[M]| + 2 = |U| + 2 = |P|}$.
By symmetry, $\nested'$ is a maximal $\building_{|P}$-nested set.
Since~$\nested \ssm \{B\} = \nested' \ssm \{B'\}$, we obtain that~$B$ and~$B'$ are exchangeable in~$\nestedComplex(\building_{|P})$, hence in~$\nestedComplex(\building)$ by \cref{prop:links}.
The parent of this exchange is~$P$ and the pivots are $v$ and~$v'$.
\end{proof}

\begin{remark}
\label{rem:differencesGraphicalExchangeables}
For the graphical nested fans, \cref{prop:exchangeablePairsGraphicalNestedFan}\,(i) ensures that if~$B$ and~$B'$ are exchangeable, then $B \cup B'$ is always a block and is the only possible parent (note however that $B$ and $B'$ are not necessarily exchangeable when~$B \cup B'$ is a block).
In contrast to the graphical case, for a general building set,
\begin{itemize}
\item the same exchangeable blocks may admit several possible parents and pivots,
\item the set of parents does not necessarily admit a unique (inclusion) minimal element,
\item $B \cup B'$ is not always a block when~$B$ and~$B'$ are exchangeable. In other words, $B$ and~$B'$ can be exchangeable even if~$\{B,B'\} \cup \connectedComponents(\building)$ is a $\building$-nested set.
\end{itemize}
For instance, in the building set~$\building_\circ$ of \cref{fig:exmNested}\,(left), the blocks $B = 14$ and $B' = 25$ are simultaneously compatible and exchangeable. They are exchangeable with parent~$12345$ and pivots~$(1,2)$ or with parent~$12456$ and pivots~$(4,5)$.
\end{remark}

\begin{remark}
Observe also that it follows from the definitions that
\begin{itemize}
\item it suffices to check the condition of \cref{prop:exchangeablePairsNestedFan} for~$C$ and~$C'$ elementary blocks of~$\building$,
\item if~$B$ and~$B'$ are exchangeable, then~$B \not\subseteq B'$ and~$B' \not\subseteq B$,
% \item if~$B$ and~$B'$ are exchangeable and admit~$P$ as a parent, then they admit any~$P' \in \building$ with ${B \cup B' \subseteq P' \subseteq P}$ as a parent (using the same pivots),
\item if $(B, B', P)$ is an exchange frame and  ${B \cup B' \subseteq P' \subseteq P}$, then $(B, B', P')$ is also an exchange frame (using the same pivots),
% \item if~$B$ and~$B'$ are exchangeable they admit~$B \cup B'$ as a parent as soon as~$B \cup B'$ is a block (in particular if~$B \cap B' \ne \varnothing$).
\item if~$B$ and~$B'$ are exchangeable and~$B \cup B'$ is a block (in particular if~$B \cap B' \ne \varnothing$), then~$(B, B', B \cup B')$ is an exchange frame.
\end{itemize}
\end{remark}

We now apply \cref{prop:exchangeablePairsNestedFan} to identify some exchange frames that will play an important role in the description of the type cone of the $\building$-nested fan.

\begin{proposition}
\label{prop:maximalExchanges}
If~$B, B', P \in \building$ are such that~$B$ and~$B'$ are two distinct blocks of~$\building$ strictly contained in~$P$ and inclusion maximal inside~$P$, then~$(B, B', P)$ is an exchange frame.
\end{proposition}

\begin{proof}
Consider~$C \in \building$ such that~$\leaving{B}{C}{P}$.
Since~$B \cap C \ne \varnothing$, we have~$B \cup C \in \building$. Since~$C \subseteq P$ and~$C \not\subseteq B$, we have~$B \subsetneq B \cup C \subseteq P$.
By maximality of~$B$ in~$P$, we obtain that~$B \cup C = P$. 
Hence, $\leaving{B}{C}{P}$ implies~$B' \ssm B \subseteq C$ and similarly~$\leaving{B'}{C'}{P}$ implies~$B \ssm B' \subseteq C'$.
Therefore, choosing any~$v \in B \ssm B'$ and~$v' \in B' \ssm B$, we obtain that $B, B', P, v, v'$ satisfy the conditions of \cref{prop:exchangeablePairsNestedFan}, and thus~$(B, B', P)$ is an exchange frame.
\end{proof}

We call \defn{maximal exchange frames} the exchange frames defined by \cref{prop:maximalExchanges}.
For~$P \in \building$, we will denote by~$\maximalBlocks(P)$ the maximal blocks of~$\building$ strictly contained in~$P$.

%%%%%%%%%%%

\subsection{$\b{g}$-vector dependences}
\label{subsec:gvectorDependences}

We now describe the exchange relations in the $\building$-nested fan~$\nestedFan[\building]$.
We first need to observe that certain building blocks are forced to belong to any two adjacent maximal nested sets with a given frame, generalizing \cref{prop:exchangeablePairsGraphicalNestedFan}\,(ii).

\begin{proposition}
\label{prop:forcedBlocks}
For two adjacent maximal $\building$-nested sets~$\nested$ and~$\nested'$ with~$\nested \ssm \{B\} = \nested' \ssm \{B'\}$ and parent~$P$, all connected components of~$\connectedComponents(B \cap B')$ and of~$\connectedComponents \big(P \ssm (B \cup B'))$ belong to~$\nested \cap \nested'$.
\end{proposition}

\begin{proof}
Even if we discuss separately the elements of $\connectedComponents(B \cap B')$ from that of~$\connectedComponents(P \ssm (B \cup B'))$, the reader will see a lot of similarities in the arguments below.

We first consider~$K \in \connectedComponents(B \cap B')$ and prove that~$\nested \cup \{K\}$ is a $\building$-nested set, which proves that~$K \in \nested$ by maximality of~$\nested$.
Indeed, let us consider a subset~$\nested[X]$ of~$\nested \cup \{K\}$ whose union~$\bigcup \nested[X]$ is in~$\building$, and prove that~$\bigcup \nested[X]$ is in~$\nested[X]$.
We assume that~$K \in \nested[X]$, since otherwise~$\nested[X] \subseteq \nested$ so that~$\bigcup \nested[X]$ is in~$\nested[X]$ as~$\nested$ is a $\building$-nested set.
Assume now that~$B \in \nested[X]$ and define~$\nested[Y] \eqdef \nested[X] \ssm \{K\}$.
Since~$K \subseteq B \in \nested[X]$, we have~$\bigcup \nested[Y] = \bigcup \nested[X]$ in~$\building$, thus in~$\nested[Y] \subset \nested[X]$ since~$\nested[Y] \subseteq \nested$ and~$\nested$ is a $\building$-nested set.
It remains to consider the case when~$\nested[X] \subseteq (\nested \cap \nested') \cup \{K\}$.
Assume now that~$\bigcup \nested[X] \not\subseteq B$ and define~$\nested[Y] \eqdef \{B\} \cup \nested[X] \ssm \{K\}$.
Since~$K \subseteq B$, we have $\bigcup \nested[Y] = B \cup \bigcup \nested[X]$ which belongs to~$\building$ since~$B$ and~$\bigcup \nested[X]$ both belong to~$\building$ and intersect~$K$.
Hence, $\bigcup \nested[Y]$ is in~$\nested[Y]$ since~$\nested[Y] \subseteq \nested$ and~$\nested$ is a $\building$-nested set.
Note that~$\bigcup \nested[Y] \ne B$ by our assumption that~$\bigcup \nested[X] \not\subseteq B$.
Therefore, $\bigcup \nested[Y]$ is in~$\nested[X]$, and thus~$\bigcup \nested[X] = \bigcup \nested[Y]$ is in~$\nested[X]$.
By symmetry, we obtain that~$\bigcup \nested[X]$ is in~$\nested[X]$ if~$\bigcup \nested[X] \not\subseteq B'$.
Assume finally that~$\bigcup \nested[X] \subseteq B \cap B'$.
Then all the elements of~$\nested[X]$ are in~$B \cap B'$.
Since~$K \in \nested[X]$ is a connected component of~$B \cap B'$ and~$\bigcup \nested[X]$ is in~$\building$, this implies that~$\bigcup \nested[X] = K \in \nested[X]$.

We now consider~$K \in \connectedComponents(P \ssm (B \cup B'))$ and prove that~$\nested \cup \{K\}$ is a $\building$-nested set, which proves that~$K \in \nested$ by maximality of~$\nested$.
Indeed, let us consider a subset~$\nested[X]$ of~$\nested \cup \{K\}$ whose union~$\bigcup \nested[X]$ is in~$\building$, and prove that~$\bigcup \nested[X]$ is in~$\nested[X]$.
We assume that~$K \in \nested[X]$, since otherwise~$\nested[X] \subseteq \nested$ so that~$\bigcup \nested[X]$ is in~$\nested[X]$ as~$\nested$ is a $\building$-nested set.
Assume now that~$\bigcup \nested[X] \not\subseteq P$ and define~$\nested[Y] \eqdef \{P\} \cup \nested[X] \ssm \{K\}$.
Since~$K \subseteq P$, we have $\bigcup \nested[Y] = P \cup \bigcup \nested[X]$ which belongs to~$\building$ since~$P$ and~$\bigcup \nested[X]$ both belong to~$\building$ and intersect~$K$.
Hence, $\bigcup \nested[Y]$ is in~$\nested[Y]$ since~$\nested[Y] \subseteq \nested$ and~$\nested$ is a \mbox{$\building$-nested} set.
Note that~$\bigcup \nested[Y] \ne P$ by our assumption that~$\bigcup \nested[X] \not\subseteq P$.
Therefore, $\bigcup \nested[Y]$ is in~$\nested[X]$, and thus~$\bigcup \nested[X] = \bigcup \nested[Y]$ is in~$\nested[X]$.
Assume now that~${\bigcup \nested[X] \subseteq P \ssm (B \cup B')}$.
Then all elements of~$\nested[X]$ are in~$P \ssm (B \cap B')$.
Since~$K$ is a connected component of~$P \ssm (B \cap B')$ and~$\bigcup \nested[X]$ is in~$\building$, this implies that~${\bigcup \nested[X] = K \in \nested[X]}$.
Assume finally that~$\bigcup \nested[X]$ is contained in~$P$ and intersects~$B$ or~$B'$.
If~$B' \cap \bigcup \nested[X] \ne \varnothing$, then~${\leaving{B'}{\bigcup \nested[X]}{P}}$, thus~$v \eqdef \rootset{P}{\nested} \in \bigcup \nested[X]$ by \cref{prop:exchangeablePairsNestedFan}.
Hence in both cases~$B \cap \bigcup \nested[X] \ne \varnothing$, thus~${\leaving{B}{\bigcup \nested[X]}{P}}$, and thus~$v' \eqdef \rootset{P}{\nested} \in \bigcup \nested[X]$ by \cref{prop:exchangeablePairsNestedFan}.
Therefore, there is~$C \in \nested[X] \ssm \{K\} \subseteq \nested$ containing~$v'$.
Since~$v' = \rootset{P}{\nested}$, we obtain that~$P \subseteq C$, and hence $P=C$ because $C\subseteq \bigcup \nested[X] \subseteq P$.
Thus~$K \subseteq C$ and~$\bigcup \nested[X] = \bigcup \nested[Y]$ where~$\nested[Y] \eqdef \nested[X] \ssm \{K\}$.
Hence, $\bigcup \nested[Y]$ is in~$\nested[Y]$ since~$\nested[Y] \subseteq \nested$ and~$\nested$ is a $\building$-nested set.
We conclude that~$\bigcup \nested[X] = \bigcup \nested[Y]$ is in~$\nested[X]$.

We obtained that all blocks of~$\connectedComponents(B \cap B')$ and of~$\connectedComponents \big(P \ssm (B \cup B'))$ belong to~$\nested$, and thus also to~$\nested'$ by symmetry.
\end{proof}

We are now ready to describe the exchange relations in the $\building$-nested fan.
The main message here is that these relations only depend on the frames of the exchanges, generalizing \cref{prop:exchangeablePairsGraphicalNestedFan}\,(iii).

\begin{proposition}
\label{prop:exchangeRelation}
For two adjacent maximal $\building$-nested sets~$\nested$ and~$\nested'$ with~$\nested \ssm \{B\} = \nested' \ssm \{B'\}$ and parent~$P$, the unique (up to rescaling) linear dependence between the $\b{g}$-vectors of~$\nested \cup \nested'$ is
\begin{equation}
\label{eq:exchangeRelation}
\gvector{B} + \gvector{B'} + \sum_{K \in \connectedComponents(P \ssm (B \cup B'))} \gvector{K} = \gvector{P} + \sum_{K \in \connectedComponents(B \cap B')} \gvector{K}.
\end{equation}
In particular, the $\b{g}$-vector dependence only depends on the exchange frame~$(B,B',P)$.
\end{proposition}

\begin{proof}
\cref{eq:exchangeRelation} is a valid linear dependence since it holds at the level of characteristic vectors, and~$\gvector{C} \eqdef \pi \big( \sum_{v \in C} \b{e}_v \big)$ where~$\pi$ is the orthogonal projection from~$\R^\ground$ to~$\HH$.
Since all building blocks involved in \cref{eq:exchangeRelation} belong to~$\nested \cup \nested'$ by \cref{prop:forcedBlocks}, we conclude that \cref{eq:exchangeRelation} is the unique (up to rescaling) linear dependence between the $\b{g}$-vectors of~$\nested \cup \nested'$.
\end{proof}

\begin{remark}
\label{rem:differencesGraphicalExchangeRelation}
For the graphical nested fans studied in \cref{subsec:exchangeableTubes}, the parent of the exchange of~$B$ and~$B'$ is always~${B \cup B'}$ and we recover the $\b{g}$-vector relation of \cref{prop:exchangeablePairsGraphicalNestedFan}\,(iii).
In contrast to the graphical case, for an arbitrary building set,
\begin{itemize}
\item the sum on the left of \cref{eq:exchangeRelation} is empty only when~$P = B \cup B'$,
\item \cref{eq:exchangeRelation} depends on the exchange frame~$(B, B', P)$, not only on the exchangeable building blocks~$B$ and~$B'$.
\end{itemize}
For instance, the $\b{g}$-vector relation of the exchange between the two adjacent maximal $\building_\circ$-nested sets~$\nested_\circ$ and~$\nested'_\circ$ represented in \cref{fig:exmNested}\,(middle) is $\b{g}_{14} + \b{g}_{25} + \b{g}_3 = \b{g}_{12345}$.
Another $\b{g}$-vector relation for the same exchangeable blocks~$B = 14$ and~$B' = 25$ is~$\b{g}_{14} + \b{g}_{25} + \b{g}_6 = \b{g}_{12456}$.
\end{remark}

\begin{remark}
\label{rem:Zelevinsky}
The $\b{g}$-vector dependences were already studied in~\cite{Zelevinsky}.
Namely, our \cref{prop:forcedBlocks,eq:exchangeRelation} are essentially Proposition 4.5 and Equation~(6.6) of~\cite{Zelevinsky}.
Our versions are however more precise since we obtained in \cref{prop:exchangeablePairsNestedFan} a complete characterization of the exchangeable building blocks of~$\building$, which was surprisingly missing in the literature.
\end{remark}

Note that while the $\b{g}$-vector dependence only depends on the exchange frame, different frames may lead to the same $\b{g}$-vector dependence.
In the next two statements, we describe which of the maximal exchange frames lead to the same $\b{g}$-vector dependence.
Remember that we denote by~$\maximalBlocks(P)$ the maximal blocks of~$\building$ strictly contained in a block~$P \in \building$.

\begin{proposition}
\label{prop:mutualizedExchange1}
%Consider an elementary block~$P \in \elementary(\building)$ and let~$B_1, \dots, B_\ell$ denote the maximal blocks of~$\building$ strictly contained in~$P$.
%Then all exchange frames~$(B_i, B_j, P)$ for~$\{i,j\} \in \binom{[\ell]}{2}$ lead to the same $\b{g}$-vector dependence~$\sum_{i \in [\ell]} \gvector{B_i} = \gvector{P}$.
For an elementary block~$P \in \elementary(\building)$, all exchange frames~$(B, B', P)$ for~$B \ne B'$ in~$\maximalBlocks(P)$ lead to the same $\b{g}$-vector dependence~$\sum_{B \in \maximalBlocks(P)} \gvector{B} = \gvector{P}$.
\end{proposition}

\begin{proof}
Observe first that~$(B, B', P)$ is indeed an exchange frame by \cref{prop:maximalExchanges}.
We thus apply \cref{prop:exchangeRelation} to describe the corresponding $\b{g}$-vector dependence.
Observe first that the sum on the right of \cref{eq:exchangeRelation} is empty because~$B \cap B' = \varnothing$ by \cref{rem:elementary} since~$P$ is elementary and~${B, B' \in \maximalBlocks(P)}$.
The result thus follows from the observation that~$\connectedComponents(P \ssm (B \cup B')) = \maximalBlocks(P) \ssm \{B, B'\}$ which we prove next.

Let us consider~${K \in \connectedComponents(P \ssm (B \cup B'))}$ and prove that~$K \in \maximalBlocks(P) \ssm \{B, B'\}$.
Consider~$L \in \building$ such that~$K \subseteq L \subsetneq P$.
If~${L \cap B \ne \varnothing}$, then~$L \cup B \in \building$ and~$B \subsetneq L \cup B \subseteq P$, so that~$L \cup B = P$ by maximality of~$B$, contradicting the elementarity of~$P$.
Hence, $L \subseteq P \ssm (B \cup B')$, so that~$K = L$ by maximality of~$K$ in~$P \ssm (B \cup B')$.
We conclude that~$K \in \maximalBlocks(P) \ssm \{B, B'\}$.

Conversely, let us consider~$C \in \maximalBlocks(P) \ssm \{B, B'\}$ and prove that~$C \in \connectedComponents(P \ssm (B \cup B'))$.
Since~$P$ is elementary and~$B, B', C \in \maximalBlocks(P)$, the block $C$ is disjoint from~$B$ and~$B'$ by \cref{rem:elementary}.
Hence, ${C \subseteq P \ssm (B \cup B')}$ and thus~$C \in \connectedComponents(P \ssm (B \cup B'))$ by maximality~of~$C$.
\end{proof}

\begin{proposition}
\label{prop:mutualizedExchange2}
If~$(B_1, B'_1, P)$ and~$(B_2, B'_2, P)$ are two distinct maximal exchange frames with the same $\b{g}$-vector dependence, then $P$ is elementary.
\end{proposition}

\begin{proof}
Since the exchange relations given by \cref{eq:exchangeRelation} for the exchange frames~$(B_1, B'_1, P)$ and~$(B_2, B'_2, P)$ coincide, $B_2$ and~$B'_2$ belong to~$\{B_1, B'_1\} \cup \connectedComponents(P \ssm (B_1 \cup B'_1))$.
Since~$(B_1, B'_1, P)$ and~$(B_2, B'_2, P)$ are distinct exchange frames, we can assume for instance that~$B_2$ does not belong to~$\{B_1, B'_1\}$.
Hence, $B_2$ belongs to~$\connectedComponents(P \ssm (B_1 \cup B'_1))$, thus~$B_1 \cap B_2 = \varnothing$, and therefore~$P$ is elementary by \cref{rem:elementary} since it contains two disjoint maximal blocks.
\end{proof}

%%%%%%%%%%%

\subsection{Type cone of nested fans}
\label{subsec:typeConeNestedFans}

As a consequence of \cref{prop:exchangeRelation}, we obtain the following redundant description of the type cone of the nested fan~$\nestedFan[\building]$.

\begin{corollary}
\label{coro:typeConeNestedFan}
For any building set~$\building$, the type cone of the nested fan~$\nestedFan[\building]$ is given by
\[
\typeCone(\nestedFan[\building]) = \set{\b{h} \in \R^{\tubes}}{\begin{array}{l} \b{h}_{B} = 0 \text{ for } B \in \connectedComponents(\building) \text{ and for any exchange frame } (B, B', P) \\ \b{h}_B + \b{h}_{B'} + \sum_{K \in \connectedComponents(P \ssm (B \cup B'))} \b{h}_K > \b{h}_P + \sum_{K \in \connectedComponents(B \cap B')} \b{h}_K \end{array}}.
\]
\end{corollary}

We denote by~$\b{f}_B$ for~$B \in \building$ the canonical basis of~$\R^{\building}$ and by 
\[
\b{n}(B,B',P) \eqdef \Big( \b{f}_B + \b{f}_{B'} + \sum_{K \in \connectedComponents(P \ssm (B \cup B'))} \b{f}_K \Big) - \Big( \b{f}_P + \sum_{K \in \connectedComponents(B \cap B')} \b{f}_K \Big)
\]
the inner normal vector of the inequality of the type cone~$\typeCone(\nestedFan[\building])$ corresponding to an exchange frame~$(B, B', P)$ of~$\building$. Thus $\b{h} \in \typeCone(\nestedFan[\building])$ if and only if $\dotprod{\b{n}(B, B', P)}{\b{h}} > 0$ for all exchange frames~$(B, B', P)$ of~$\building$.

\pagebreak

\begin{example}
\label{exm:typeConeNestedFan}
Consider the nested fans illustrated in \cref{fig:nestedFans}.
The type cone of the left fan lives in~$\R^8$, has a lineality space of dimension~$3$ and $5$ facet-defining inequalities (given below). In particular, it is simplicial.
Note that as in \cref{fig:graphicalNestedFans}, we express the $\b{g}$-vectors in the basis given by the maximal tubing containing the first three tubes below.

\medskip
\centerline{$
\begin{array}{r|c@{\qquad}c@{\qquad}c@{\qquad}c@{\qquad}c@{\qquad}c@{\qquad}c@{\qquad}c}
\text{blocks} & \raisebox{-.4cm}{\includegraphics[scale=.6]{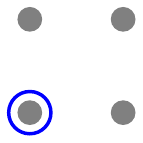}} & \raisebox{-.4cm}{\includegraphics[scale=.6]{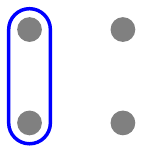}} & \raisebox{-.4cm}{\includegraphics[scale=.6]{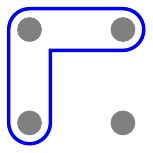}} & \raisebox{-.4cm}{\includegraphics[scale=.6]{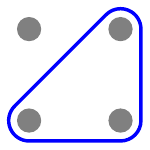}} & \raisebox{-.4cm}{\includegraphics[scale=.6]{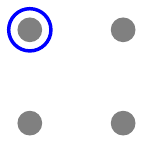}} & \raisebox{-.4cm}{\includegraphics[scale=.6]{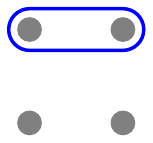}} & \raisebox{-.4cm}{\includegraphics[scale=.6]{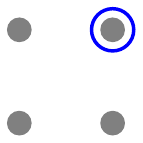}} & \raisebox{-.4cm}{\includegraphics[scale=.6]{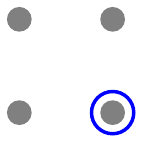}} \\[.6cm]
\text{$\b{g}$-vectors} & \compactVectorT{1}{0}{0} & \compactVectorT{0}{1}{0} & \compactVectorT{0}{0}{1} & \compactVectorT{1}{-1}{0} & \compactVectorT{-1}{1}{0} & \compactVectorT{-1}{0}{1} & \compactVectorT{0}{-1}{1} & \compactVectorT{0}{0}{-1} \\[.6cm]
\text{facet}		& 1 & -1 & 0 & 0 & 1 & 0 & 0 & 0 \\
\text{defining}		& 0 & 0 & 0 & 0 & 1 & -1 & 1 & 0 \\
\text{inequalities}	& 1 & 0 & 0 & -1 & 0 & 0 & 1 & 1 \\
				& 0 & 1 & -1 & 0 & -1 & 1 & 0 & 0 \\
				& -1 & 0 & 1 & 1 & 0 & 0 & -1 & 0 \\[.2cm]
\end{array}
$}

\medskip
\noindent
The type cone of the right fan lives in~$\R^8$, has a lineality space of dimension~$3$ and $7$ facet-defining inequalities (given below). In particular, it is not simplicial.

\[
\begin{array}{r|c@{\qquad}c@{\qquad}c@{\qquad}c@{\qquad}c@{\qquad}c@{\qquad}c@{\qquad}c}
\text{blocks} & \raisebox{-.4cm}{\includegraphics[scale=.6]{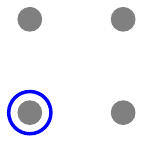}} & \raisebox{-.4cm}{\includegraphics[scale=.6]{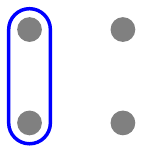}} & \raisebox{-.4cm}{\includegraphics[scale=.6]{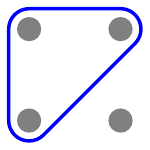}} & \raisebox{-.4cm}{\includegraphics[scale=.6]{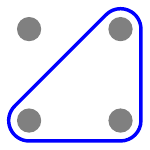}} & \raisebox{-.4cm}{\includegraphics[scale=.6]{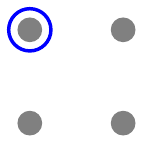}} & \raisebox{-.4cm}{\includegraphics[scale=.6]{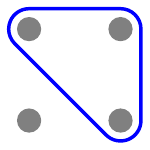}} & \raisebox{-.4cm}{\includegraphics[scale=.6]{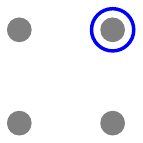}} & \raisebox{-.4cm}{\includegraphics[scale=.6]{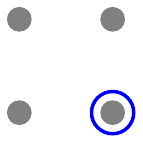}} \\[.6cm]
\text{$\b{g}$-vectors} & \compactVectorT{1}{0}{0} & \compactVectorT{0}{1}{0} & \compactVectorT{0}{0}{1} & \compactVectorT{1}{-1}{0} & \compactVectorT{-1}{1}{0} & \compactVectorT{-1}{0}{0} & \compactVectorT{0}{-1}{1} & \compactVectorT{0}{0}{-1} \\[.6cm]
\text{facet}		& 1 & -1 & 0 & 0 & 1 & 0 & 0 & 0 \\
\text{defining}		& 0 & 1 & -1 & 0 & 0 & 0 & 1 & 0 \\
\text{inequalities}	& 1 & 0 & 0 & -1 & 0 & 0 & 1 & 1 \\
				& 0 & 0 & 0 & 0 & 1 & -1 & 1 & 1 \\
				& -1 & 0 & 1 & 1 & 0 & 0 & -1 & 0 \\
				& 0 & 0 & 1 & 0 & -1 & 1 & -1 & 0 \\
				& 0 & 0 & 0 & 1 & 0 & 1 & -1 & -1 \\[.2cm]
\end{array}
\]
\medskip
\end{example}

\begin{example}
\label{exm:constructionsNestedFan}
We can exploit \cref{coro:typeConeNestedFan} to show that certain height functions belong to the type cone of~$\nestedFan[\building]$ and recover some classical constructions of the nestohedron, generalizing \cref{exm:constructionsGraphicalNestedFan}.

\begin{enumerate}[(i)]
\item Consider the height function~$\b{h} \in \R^\building$ given by~$\b{h}_B \eqdef -3^{|B|}$. Then for any exchange frame~$(B, B', P)$ of~$\building$, we have
\begin{align*}
\qquad\qquad
\dotprod{\b{n}(B, B')}{\b{h}}
& = - 3^{|B|} - 3^{|B'|} - \sum_{K \in \connectedComponents(P \ssm (B \cup B'))} 3^{|K|} + 3^{|P|} + \sum_{K \in \connectedComponents(B \cap B')} 3^{|K|} \\
& \ge - 2 \cdot 3^{|B \cup B'|-1} - 3^{|P \ssm (B \cup B'))|} + 3^{|P|} > 0.
\end{align*}
Therefore, the height function~$\b{h}$ belongs to the type cone~$\typeCone(\nestedFan[\building])$.
The corresponding polytope~$P_\b{h} \eqdef \bigset{\b{x} \in \R^\ground}{\dotprod{\gvector{B}}{\b{x}} \le \b{h}_B \text{ for } B \in \building}$ was constructed in~\cite{Pilaud-removahedra}, generalizing the graph associahedra of~\cite{Devadoss}.
\item Consider the height function~$\b{h} \in \R^\building$ given by~$\b{h}_B \eqdef -|\bigset{C \in \building}{C \subseteq B}|$. Then for any exchange frame~$(B, B', P)$ of~$\building$, we have
\[
\qquad\quad\;
\dotprod{\b{n}(B, B', P)}{\b{h}} = |\set{C \in \building}{C \not\subseteq B, \; C \not\subseteq B' \text{ and } C \not\subseteq P \ssm (B \cup B') \text{ but } C \subseteq P}| > 0
\]
since~$P$ fulfills the conditions on~$C$.
Therefore, the height function~$\b{h}$ belongs to the type cone~$\typeCone(\nestedFan[\building])$.
The corresponding polytope~${P_\b{h} \eqdef \bigset{\b{x} \in \R^\ground}{\dotprod{\gvector{B}}{\b{x}} \le \b{h}_B \text{ for } B \in \building}}$ is the nestohedron constructed by A.~Postnikov's in~\cite{Postnikov}.
\end{enumerate}
\end{example}

Note that many inequalities of \cref{coro:typeConeNestedFan} are redundant.
In the remaining of this section, we describe the facet-defining inequalities of~$\typeCone(\nestedFan[\building])$.
We say that an exchange frame~$(B, B', P)$~is
\begin{itemize}
\item \defn{extremal} if its corresponding inequality in \cref{coro:typeConeNestedFan} defines a facet of~$\typeCone(\nestedFan[\building])$,
\item \defn{maximal} if~$B$ and~$B'$ are both maximal building blocks in~$P$ as in \cref{prop:maximalExchanges}.
\end{itemize}
We can now state our main result on nested complexes, generalizing \cref{thm:extremalExchangeablePairsGraphicalNestedFan}.

\begin{theorem}
\label{thm:extremalExchangeFramesNestedFan}
An exchange frame is extremal if and only if it is maximal.
\end{theorem}

\begin{proof}
We treat separately the two implications:

\para{Extremal $\implies$ maximal}
Consider an exchange frame~$(B, B', P)$ of~$\building$, and fix pivot vertices~$v,v'$ satisfying the conditions of \cref{prop:exchangeablePairsNestedFan}.
We assume that this frame is not maximal, and prove that it is not extremal by showing that the normal vector~$\b{n}(B, B', P)$ of the corresponding inequality of the type cone~$\typeCone(\nestedFan[\building])$ is a positive linear combination of normal vectors of some other exchange frames.
By symmetry, we can assume that there is~$M \in \building$ such that~$B \subsetneq M \subsetneq P$ and we can assume that~$M$ is maximal for this property.
We decompose the proof into two cases, depending on whether~$B' \subseteq M$ or $B' \not\subseteq M$.

\medskip
\para{Case~1: $B' \subseteq M$}
Observe first that:
\begin{itemize}
\item $(B, B', M)$ is an exchange frame, since $(B, B', P)$ is an exchange frame and ${B \cup B' \subseteq M \subseteq P}$,
\item $(M, W, P)$ is an exchange frame for any connected component~$W$ of~$P \ssm (B \cup B')$ containing a vertex~$w \in P \ssm M$.
Indeed, we just check the conditions of \cref{prop:exchangeablePairsNestedFan} for~$v \in M \ssm W$ and~$w \in W \ssm M$:
	\begin{itemize}
	\item for any~$\leaving{M}{C}{P}$, we have~$w \in P \ssm M \subseteq C$ by maximality of~$M$.
	\item for any~$\leaving{W}{C'}{P}$, we have~$C' \subseteq P$ and~$C' \not\subseteq W$, hence~$C' \cap (B \cup B') \ne \varnothing$ since~$W$ is a connected component of~$P \ssm (B \cup B')$. Assume for instance that~$C' \cap B \ne \varnothing$ (the proof for~$C' \cap B' \ne \varnothing$ is symmetric). Since~$C' \cap W \ne \varnothing$, we obtain that~$\leaving{B}{C'}{P}$ and thus $v' \in C'$ by \cref{prop:exchangeablePairsNestedFan}. We therefore obtain that~$\leaving{B'}{C'}{P}$ and thus~$v \in C'$ by \cref{prop:exchangeablePairsNestedFan} again.
	\end{itemize}
\end{itemize}
We claim that these two exchange frames enable us to write 
\[
\b{n}(B, B', P) = \b{n}(B, B', M) + \b{n}(M, W, P).
\]
Proving this identity amounts to check that
\begin{equation}
\label{eq:connectedComponents1}
\connectedComponents(P \ssm (B \cup B')) \sqcup \connectedComponents(M \cap W) =  \connectedComponents(M \ssm (B \cup B')) \sqcup \connectedComponents(P \ssm (M \cup W)) \sqcup \{W\}.
\end{equation}
For this, we distinguish two subcases, depending on whether or not~$M$ and~$W$ intersect. 

\para{Subcase 1.1: $M \cap W = \varnothing$}
See \cref{fig:extremalImpliesMaximal}\,(left).
First, we claim that either~$C \cap M = \varnothing$ or~$C \subseteq M$ for any~$C \in \building$ with~$C \subseteq P \ssm (B \cup B')$.
Indeed, if~$C \cap M \ne \varnothing$, then~$C \cap W = \varnothing$ since~$M \cap W = \varnothing$ and~$W$ is a connected component of~$P \ssm (B \cup B')$.
Hence~$C \cup M \in \building$ and $B \subsetneq C \cup M \subsetneq P$, and thus~$C \subseteq M$ by maximality of~$M$.
We therefore obtain that
\[
\connectedComponents(P \ssm (B \cup B')) = \connectedComponents(M \ssm (B \cup B')) \sqcup \connectedComponents(P \ssm (M \cup W)) \sqcup \{W\}.
\]
This shows \cref{eq:connectedComponents1} since~$M \cap W = \varnothing$.

\para{Subcase 1.2: $M \cap W \ne \varnothing$}
See \cref{fig:extremalImpliesMaximal}\,(middle).
As~$M \cup W \in \building$ and~$B \subsetneq B \cup W \subseteq P$ and~$W \not\subseteq M$, we have~${P = M \cup W}$ by maximality of~$M$.
Since~$W \in \connectedComponents(P \ssm (B \cup B'))$, we have
\[
\connectedComponents(P \ssm (B \cup B')) = \connectedComponents(P \ssm (B \cup B' \cup W)) \sqcup \{W\} = \connectedComponents(M \ssm (B \cup B' \cup W)) \sqcup \{W\} \\
\]
Moreover, by maximality of~$W$, we obtain that there is no block of~$\building$ contained in~$M \ssm (B \cup B')$ and meeting both~$M \cap W$ and~$M \ssm (B \cup B' \cup W)$.
%~$C \in \building$ with~$C \subseteq M \ssm (B \cup B')$ and~$C \cap (M \ssm (B \cup B' \cup W)) \ne \varnothing \ne C \cap (M \cap W)$.
Hence
\[
\connectedComponents(M \cap W) \sqcup \connectedComponents(M \ssm (B \cup B' \cup W)) = \connectedComponents(M \ssm (B \cup B')).
\]
Combining these two identities proves \cref{eq:connectedComponents1} since~$P = M \cup W$.

\begin{figure}[t]
	\centerline{
		\begin{tabular}{c@{\qquad}c@{\qquad}c}
			\begin{tikzpicture}[every path/.style={draw, rounded corners=1ex, thick}]
				\draw[green] (0, 2.5) rectangle (4, 0) node[pos=0, xshift=1.5ex, yshift=-1.5ex] {$P$};
				\draw[purple] (.1, .1) rectangle (2.9, 1.9) node[pos=0, xshift=1.6ex, yshift=1.5ex] {$M$};
				\draw[red] (.2, 1.8) rectangle (1.7, .8) node[pos=0, xshift=1.5ex, yshift=-1.5ex] {$B$};
				\draw[blue] (1.3, 1.2) rectangle (2.8, .2) node[pos=1, xshift=-1.5ex, yshift=1.5ex] {$B'$};
				\draw[orange] (3, 1.3) rectangle (3.9, 2.4) node[pos=1, xshift=-1.6ex, yshift=-1.5ex] {$W$};
				\node[label={[label distance=-1.7ex] 180:{\red $v$}}] at (1.2, 1.3) {\red $\star$};
				\node[label={[label distance=-1.5ex] 0:{\blue $v\smash{'}$}}] at (1.8, .7) {\blue $\star$};
				\node[label={[label distance=-1.7ex] 180:{\orange $w$}}] at (3.6, 1.5) {\orange $\star$};
			\end{tikzpicture}
			&
			\begin{tikzpicture}[every path/.style={draw, rounded corners=1ex, thick}]
				\draw[green] (0, 2) -- (0, 0) -- (3, 0) -- (3, 1.2) -- (4, 1.2) -- (4, 2.5) -- (2.1, 2.5) -- (2.1, 2) -- cycle;  \node[green, right] at (0, 2.2) {$P = M \cup W$};
				\draw[purple] (.1, .1) rectangle (2.9, 1.9) node[pos=0, xshift=1.6ex, yshift=1.5ex] {$M$};
				\draw[red] (.2, 1.8) rectangle (1.7, .8) node[pos=0, xshift=1.5ex, yshift=-1.5ex] {$B$};
				\draw[blue] (1.3, 1.2) rectangle (2.8, .2) node[pos=1, xshift=-1.5ex, yshift=1.5ex] {$B'$};
				\draw[orange] (2.2, 1.3) rectangle (3.9, 2.4) node[pos=1, xshift=-1.6ex, yshift=-1.5ex] {$W$};
				\node[label={[label distance=-1.7ex] 180:{\red $v$}}] at (1.2, 1.3) {\red $\star$};
				\node[label={[label distance=-1.5ex] 0:{\blue $v\smash{'}$}}] at (1.8, .7) {\blue $\star$};
				\node[label={[label distance=-1.7ex] 180:{\orange $w$}}] at (3.6, 1.5) {\orange $\star$};
			\end{tikzpicture}
			&
			\begin{tikzpicture}[every path/.style={draw, rounded corners=1ex, thick}]
				\draw[green] (0, 2) -- (0, 0) -- (4, 0) -- (4, 1.3) -- (3, 1.3) -- (3, 2) -- cycle;  \node[green, right] at (0, 2.2) {$P = M \cup B'$};
				\draw[purple] (.1, .1) rectangle (2.9, 1.9) node[pos=0, xshift=1.6ex, yshift=1.5ex] {$M$};
				\draw[red] (.2, 1.8) rectangle (1.7, .8) node[pos=0, xshift=1.5ex, yshift=-1.5ex] {$B$};
				\draw[blue] (1.3, 1.2) rectangle (3.9, .2) node[pos=1, xshift=-1.5ex, yshift=1.5ex] {$B'$};
				\draw[orange] (1.5, 1) rectangle (2.8, .4) node[pos=1, xshift=-1.5ex, yshift=2ex] {$W$};
				\node[label={[label distance=-1.7ex] 180:{\red $v$}}] at (1.2, 1.3) {\red $\star$};
				\node[label={[label distance=-1.5ex] 0:{\blue $v\smash{'}$}}] at (1.8, .7) {\blue $\star$};
				\node[label={[label distance=-1.7ex] 180:{\blue $w$}}] at (3.6, 1) {\blue $\star$};
			\end{tikzpicture}
			\\[.1cm]
			Case 1.1
			&
			Case 1.2
			&
			Case 2
		\end{tabular}
		}
	\caption{Illustrations for the case analysis of the proof of \cref{thm:extremalExchangeFramesNestedFan}.}
	\label{fig:extremalImpliesMaximal}
\end{figure}
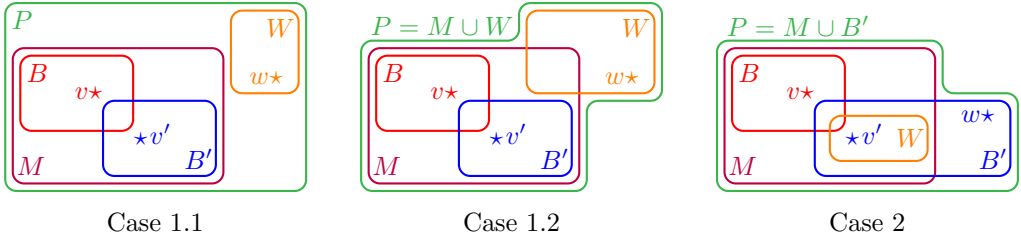

\medskip
\para{Case~2: $B' \not\subseteq M$}
See \cref{fig:extremalImpliesMaximal}\,(right).
Observe that:
\begin{itemize}
\item $(M, B', P)$ is an exchange frame. Indeed, we just check the conditions of \cref{prop:exchangeablePairsNestedFan} for~$v \in M \ssm B'$ and an arbitrary~$w \in B' \ssm M$:
	\begin{itemize}
	\item for any~$\leaving{M}{C}{P}$, we have~$w \in P \ssm M \subseteq C$ by maximality of~$M$.
	\item for any~$\leaving{B'}{C'}{P}$, we have~$v \in C'$ by \cref{prop:exchangeablePairsNestedFan}.
	\end{itemize}
\item $(B, W, M)$ is an exchange frame for the connected component~$W$ of~$M \cap B'$ containing~$v'$. Indeed, we just check the conditions of \cref{prop:exchangeablePairsNestedFan} for~$v \in B \ssm W$ and~$v' \in W \ssm B$:
	\begin{itemize}
	\item for any~$\leaving{B}{C}{M}$, we have~$\leaving{B}{C}{P}$ and thus $v' \in C$ by \cref{prop:exchangeablePairsNestedFan}.
	\item for any~$\leaving{W}{C'}{M}$, we have~$\leaving{B'}{C'}{P}$ and thus~$v \in C'$ by \cref{prop:exchangeablePairsNestedFan}.
	\end{itemize}
\end{itemize}
We claim that these two exchange frames enable to write 
\[
\b{n}(B, B', P) = \b{n}(M, B', P) + \b{n}(B, W, M).
\]
Proving this identity amounts to check that
\begin{equation}
\label{eq:connectedComponents2}
\connectedComponents(P \ssm (B \cup B')) \sqcup \connectedComponents(M \cap B')  \sqcup \connectedComponents(B \cap W) =  \connectedComponents(B \cap B') \sqcup \connectedComponents(P \ssm (M \cup B')) \sqcup \connectedComponents(M \ssm (B \cup W)) \sqcup \{W\}.
\end{equation}
To prove this, we observe that:
\begin{itemize}
\item Since~$W$ contains~$v'$, \cref{prop:exchangeablePairsNestedFan} ensures that there is no block of~$\building$ contained in~$M \cap B'$ and meeting both~$B$ and~$B' \ssm (B \cup W)$. Since~$W \in \connectedComponents(M \cap B')$, we thus obtain
\[
\connectedComponents(M \cap B') = \connectedComponents((M \cap B')  \ssm (B \cup W)) \sqcup \connectedComponents(B \cap B' \ssm W) \sqcup \{W\}.
\]
\item As~$W \in \connectedComponents(M \cap B')$, there is no block of~$\building$ contained in~$B \cap B'$ and meeting both~$B \cap W$ and~$B \cap B' \ssm W$, hence%~$C \in \building$ with~$C \subseteq M \cap B'$ and~${C \cap W \ne \varnothing \ne C \cap (B' \ssm W)}$,~thus
\[
\connectedComponents(B \cap W) \sqcup \connectedComponents(B \cap B' \ssm W) = \connectedComponents(B \cap B').
\]
\item There is no block of~$\building$ contained in~$M \ssm (B \cup W)$ and meeting both~$M \ssm (B \cup B')$ and~$(M \cap B') \ssm (B \cup W)$ (such a block~$C$ would satisfy~$\leaving{B'}{C}{P}$ and~$v \notin C$, contradicting \cref{prop:exchangeablePairsNestedFan}). Hence
\[
\connectedComponents(M \ssm (B \cup B')) \sqcup \connectedComponents((M \cap B')  \ssm (B \cup W)) = \connectedComponents(M \ssm (B \cup W)).
\]
\end{itemize}
Combining these three identities proves~\eqref{eq:connectedComponents2} since $P = M \cup B'$ by maximality of~$M$.

%This concludes the proof that extremal implies maximal.

\para{Maximal $\implies$ extremal}
Let~$(B, B', P)$ be a maximal exchange frame.
To prove that $(B, B', P)$ is extremal, we will construct a vector $\b{w} \in \R^{\building}$ such that ${\dotprod{\b{n}(B, B', P)}{\b{w}} < 0}$, but ${\dotprod{\b{n}(\tilde B, \tilde B', \tilde P)}{\b{w}} > 0}$ for any maximal exchange frame~$(\tilde B, \tilde B', \tilde P)$ with~$\b{n}(B, B', P) \ne \b{n}(\tilde B, \tilde B', \tilde P)$.
This will show that the inequality induced by~$(B, B', P)$ is not redundant.
Remember from \cref{prop:mutualizedExchange1,prop:mutualizedExchange2} that, as~$(B, B', P)$ and~$(\tilde B, \tilde B', \tilde P)$ are maximal exchange frames, $\b{n}(B, B', P) \ne \b{n}(\tilde B, \tilde B', \tilde P)$ if and only if~$P \ne \tilde P$, or~$P = \tilde P$ is not an elementary block.

Define $\alpha(B, B', P) \eqdef \set{C \in \building}{C \not\subseteq B, \; C \not\subseteq B' \text{ and } C \not\subseteq P \ssm (B \cup B') \text{ but } C \subseteq P}$.
Define three vectors~$\b{x}, \b{y}, \b{z} \in \R^{\building}$ by
\begin{align*}
\b{x}_C & \eqdef -|\bigset{D \in \building \ssm \alpha(B, B', P)}{D \subseteq C}|,
\\
\b{y}_C & \eqdef -|\bigset{D \in \alpha(B, B', P)}{D \subseteq C}|,
\\
\b{z}_C & \eqdef \begin{cases}
	-1 & \text{if } B \subseteq C \text{ or } B' \subseteq C, \\
	0 & \text{otherwise},
\end{cases}
\end{align*}
for each bock~$C \in \building$.

We will prove below that their scalar products with~$\b{n}(\tilde B, \tilde B', \tilde P)$ for any maximal exchange frame $(\tilde B, \tilde B', \tilde P)$ satisfy the following inequalities
\[
\renewcommand{\arraystretch}{1.3}
\begin{array}{l|ccc}
& \dotprod{\b{n}(\tilde B, \tilde B', \tilde P)}{\b{x}} & \dotprod{\b{n}(\tilde B, \tilde B', \tilde P)}{\b{y}} & \dotprod{\b{n}(\tilde B, \tilde B', \tilde P)}{\b{z}} \\
\hline
\text{if } \b{n}(B, B', P) = \b{n}(\tilde B, \tilde B', \tilde P) & = 0 & = |\alpha(B, B', P)| & = -1 \\
\text{if } \alpha(\tilde B, \tilde B', \tilde P) \not\subseteq \alpha(B, B', P) & \ge 1 & \ge 0 & \ge -1 \\
\text{otherwise} & \ge 0 & \ge 1 & \ge 0
\end{array}
\renewcommand{\arraystretch}{1}
\]
It immediately follows from this table that the vector $\b{w} \eqdef \b{x} + \delta \b{y} + \varepsilon \b{z}$ fulfills the desired properties for any~$\delta, \varepsilon$ such that $0 < \delta \cdot |\alpha(B, B', P)| < \varepsilon < 1$.

The equalities of the table are immediate.
To prove the inequalities, observe that for any maximal exchange frame~$(\tilde B, \tilde B', \tilde P)$,
\begin{itemize}
\item $\dotprod{\b{n}(\tilde B, \tilde B', \tilde P)}{\b{x}} \ge |\alpha(\tilde B, \tilde B', \tilde P) \ssm \alpha(B, B', P)|$,
\item $\dotprod{\b{n}(\tilde B, \tilde B', \tilde P)}{\b{y}} \ge |\alpha(\tilde B, \tilde B', \tilde P) \cap \alpha(B, B', P)|$,
\item $\dotprod{\b{n}(\tilde B, \tilde B', \tilde P)}{\b{z}} \ge -1$. Indeed, observe that~$\b{z}_{\tilde P} = -1$ as soon as~$\b{z}_{\tilde K} = -1$ for some ${\tilde K \in \{\tilde B, \tilde B'\} \sqcup \connectedComponents(\tilde P \ssm (\tilde B \cup \tilde B'))}$. This already implies that~$\dotprod{\b{n}(\tilde B, \tilde B', \tilde P)}{\b{z}} \ge -1$ except if~$\b{z}_{\tilde K} = \b{z}_{\tilde K'} = \b{z}_{\tilde K''} = -1$ for three distinct ${\tilde K, \tilde K', \tilde K'' \in \{\tilde B, \tilde B'\} \sqcup \connectedComponents(\tilde P \ssm (\tilde B \cup \tilde B'))}$. But since~$\tilde B$ and~$\tilde B'$ are the only intersecting blocks among~$\{\tilde B, \tilde B'\} \sqcup \connectedComponents(\tilde P \ssm (\tilde B \cup \tilde B'))$, the only option (up to permutation) is that~$\tilde K = \tilde B$ and~$\tilde K' = \tilde B'$ both contain~$B$ (resp.~$B'$), $K''$ contains~$B'$ (resp.~$B$), while none of the other blocks of~$\{\tilde B, \tilde B'\} \sqcup \connectedComponents(\tilde P \ssm (\tilde B \cup \tilde B'))$ meets~${B \cup B'}$. This implies that~$\b{z}_{\tilde P} = -1 = \b{z}_{L}$ for some~$L \in \connectedComponents(\tilde B \cap \tilde B')$, and thus~$\dotprod{\b{n}(\tilde B, \tilde B', \tilde P)}{\b{z}} \ge -1$.
\item $\dotprod{\b{n}(\tilde B, \tilde B', \tilde P)}{\b{z}} \ge 0$ when~$\b{n}(B, B', P) \ne \b{n}(\tilde B, \tilde B', \tilde P)$ but $\alpha(\tilde B, \tilde B', \tilde P) \subseteq \alpha(B, B', P)$. Indeed, $\alpha(\tilde B, \tilde B', \tilde P) \subseteq \alpha(B, B', P)$ implies that~$\tilde P \subseteq P$. Let ${\tilde K \in \{\tilde B, \tilde B'\} \sqcup \connectedComponents(\tilde P \ssm (\tilde B \cup \tilde B'))}$. If~$B \subseteq \tilde K$, then~$B \subseteq \tilde K \subsetneq \tilde P \subseteq P$ which implies that~$B =  \tilde K$ and~$P = \tilde P$ by maximality of~$B$ in~$P$. Similarly, $B' \subseteq \tilde K$ implies~$B' = \tilde K$ and~$P = \tilde P$. Hence, if~${\b{z}_{\tilde K} = -1}$, then by definition~$B \subseteq \tilde K$ or~$B' \subseteq \tilde K$, which implies that~$\tilde K \in \{B, B'\}$. Hence, if~${\tilde K \ne \tilde K'}$ are two distinct blocks of~$\{\tilde B, \tilde B'\} \sqcup \connectedComponents(\tilde P \ssm (\tilde B \cup \tilde B'))$  such that $\b{z}_{\tilde K} = -1 = \b{z}_{\tilde K'}$, then~$(B, B', P) = (\tilde K, \tilde K', \tilde P)$ and moreover either~$\{B, B'\} = \{\tilde K, \tilde K'\}$, or $\tilde K \cap \tilde K' = \varnothing$, so that~$P$ is elementary by \cref{rem:elementary} since it has two disjoint maximal blocks. In both cases, we obtain~${\b{n}(B, B', P) = \b{n}(\tilde B, \tilde B', \tilde P)}$ by \cref{prop:mutualizedExchange2}, contradicting our assumption. Therefore, at most one of~$\b{z}_{\tilde K}$ for ${\tilde K \in \{\tilde B, \tilde B'\} \sqcup \connectedComponents(\tilde P \ssm (\tilde B \cup \tilde B'))}$ equals to~$-1$, and if exactly one does, then~$\b{z}_{\tilde P} = -1$. We conclude that~$\dotprod{\b{n}(\tilde B, \tilde B', \tilde P)}{\b{z}} \ge 0$.
\qedhere
\end{itemize}
\end{proof}

We derive from \cref{thm:extremalExchangeFramesNestedFan} the facet description of the type cone~$\typeCone(\nestedFan[\building])$.
Remember that we denote by~$\maximalBlocks(P)$ the maximal blocks of~$\building$ strictly contained in a block~$P \in \building$.

\begin{corollary}
\label{coro:facetDescriptionTypeConeNestedFan}
The inequalities
\begin{itemize}
\item $\sum_{B \in\maximalBlocks(P)} \b{h}_{B} > \b{h}_P$ for any elementary block~$P$ of~$\building$,
\item $\b{h}_B + \b{h}_{B'} + \sum_{K \in \connectedComponents(P \ssm (B \cup B'))} \b{h}_K > \b{h}_P + \sum_{K \in \connectedComponents(B \cap B')} \b{h}_K$ for any block~$P$ of~$\building$ neither singleton nor elementary, and any two blocks~$B \ne B'$ in~$\maximalBlocks(P)$,
\end{itemize}
provide an irredundant facet description of the type cone~$\typeCone(\nestedFan[\building])$.
\end{corollary}

\begin{corollary}
\label{coro:numberFacetsTypeConeNestedFan}
The number of facets of the the type cone~$\typeCone(\nestedFan[\building])$ is
\[
|\elementary(\building)| + \sum_P \binom{\maximalBlocks(P)}{2}
\]
where the sum runs over all blocks~$P$ of~$\building$ which are neither singletons nor elementary blocks.
\end{corollary}

%%%%%%%%%%%

\pagebreak
\subsection{Simplicial type cones and interval building sets}
\label{subsec:simplicialTypeConeNestedFans}

To conclude the paper, we characterize the building sets~$\building$ whose nested fan~$\nestedFan[\building]$ has a simplicial type cone and study in more details a specific family of such building sets.

\begin{proposition}
\label{prop:simplicialTypeConeNestedFan}
The type cone~$\typeCone(\nestedFan[\building])$ is simplicial if and only if all blocks of~$\building$ with at least three distinct maximal strict subblocks are elementary.
\end{proposition}

\begin{proof}
Recall that the nested fan~$\nestedFan[\building]$ has dimension~$|\ground| - |\connectedComponents(\building)|$ and has~$|\building| - |\connectedComponents(\building)|$ rays.
Hence, the type cone~$\typeCone(\nestedFan[\building])$ is simplicial if and only if it has~$|\building| - |\ground|$ facets.
The statement thus immediately follows from \cref{coro:numberFacetsTypeConeNestedFan}.
\end{proof}

We conclude the paper by focussing on the following special family of building sets which fulfills \cref{prop:simplicialTypeConeNestedFan} and is illustrated~in~\cref{fig:intervalNestedFans}.

\begin{figure}
	\capstart
	\centerline{\includegraphics[scale=.55]{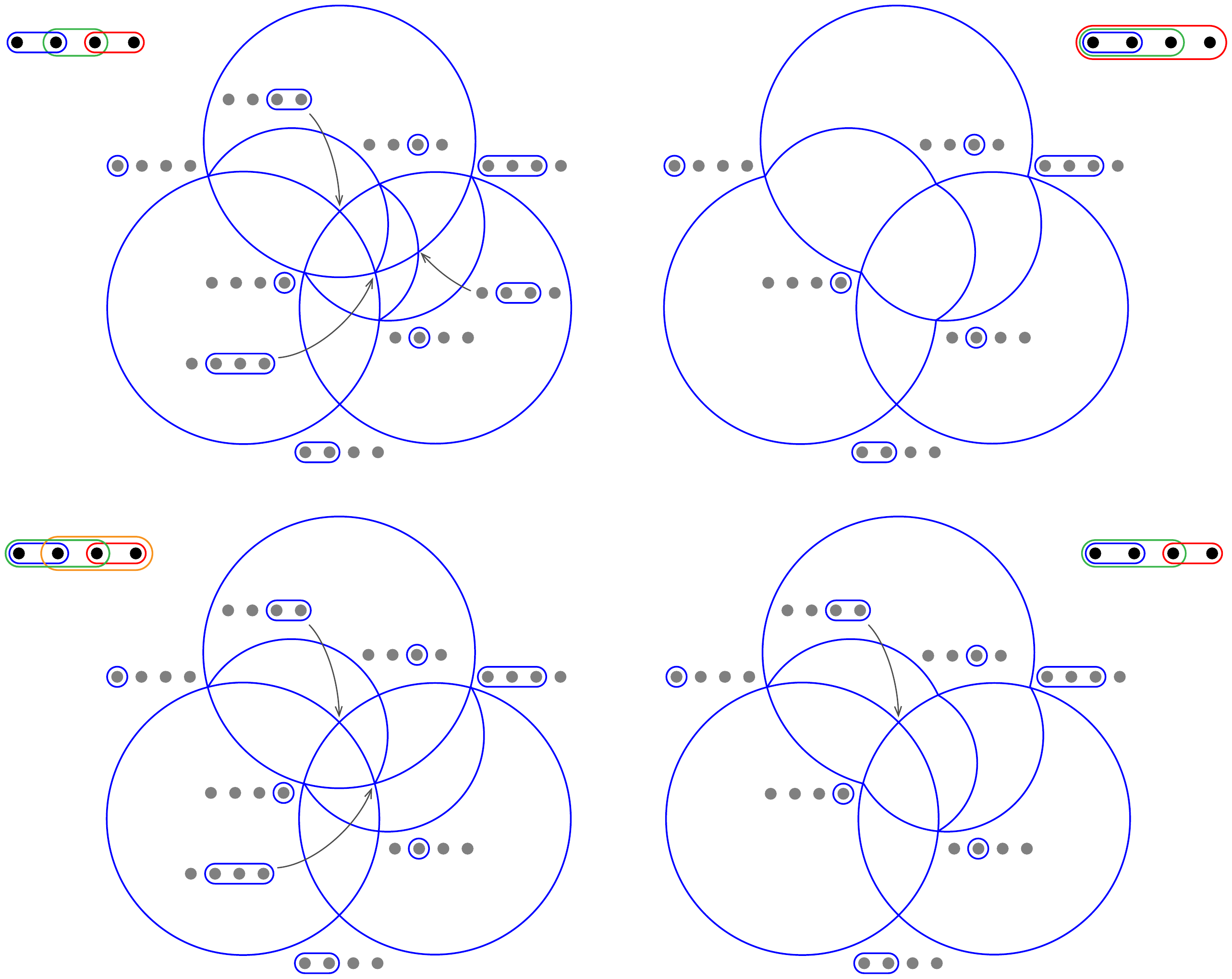}}
	\caption{Four interval nested fans. The top left one is the sylvester fan, the top right one is the Pitman-Stanley fan, the bottom left one is the freehedron fan, and only the top right one is a fertilotope fan. The rays are labeled by the corresponding blocks. As the fans are $3$-dimensional, we intersect them with the sphere and stereographically project them from the direction~$(-1,-1,-1)$.}
	\label{fig:intervalNestedFans}
\end{figure}

\begin{definition}
\label{def:intervalBuildingSet}
%A building set~$\building$ on~$\ground$ is an \defn{interval building set} if there is a total order~$\preccurlyeq$ on~$\ground$ such that all blocks of~$\building$ are intervals of~$\preccurlyeq$.
%For~$u,v \in \ground$ with~$u \preccurlyeq v$, we denote by~$[u,v] \eqdef \set{w \in \ground}{u \preccurlyeq w \preccurlyeq v}$ the corresponding interval of~$\preccurlyeq$.
An \defn{interval building set} is a building set on~$[n] \eqdef \{1, \dots, n\}$ whose blocks are some intervals.
We call \defn{interval nested fan} and \defn{interval nestohedron} the nested fan and nestohedron of an interval building set.
\end{definition}

\begin{example}
Particularly relevant examples of interval nestohedra include:
\begin{itemize}
\item the classical associahedron of~\cite{ShniderSternberg,Loday} for the building set with all intervals of~$[n]$,
\item the Pitman-Stanley polytope of~\cite{PitmanStanley} for the building set with all singletons~$\{i\}$ and all initial intervals~$[i]$~for~${i \in [n]}$,
\item the freehedron of~\cite{Saneblidze-freehedron} for the building set with all singletons~$\{i\}$, all initial intervals~$[i]$~for~${i \in [n]}$, and all final intervals~$[n] \ssm [i]$~for~${i \in [n-1]}$,
\item the fertilotopes of~\cite{Defant-fertilitopes} for the binary building sets defined as the interval building sets where any two intervals are either nested or disjoint.
\end{itemize}
Note that, by definition, any interval nested fan coarsens the associahedron nested fan.
\end{example}

\begin{proposition}
\label{prop:simplicialTypeConeInterval}
For any interval building set~$\building$, the type cone~$\typeCone(\nestedFan[\building])$ is simplicial.
\end{proposition}

\begin{proof}
Assume that~$\building$ has a non-elementary block~$[i,j]$, with at least three distinct maximal strict subblocks~$[a,b]$, $[c,d]$ and~$[e,f]$.
Since~$[a,b]$, $[c,d]$ and~$[e,f]$ are pairwise non nested, we can assume up to permutation that~$a < c < e$ and~$b < d < f$.
Since~$[i,j]$ is not elementary, $[a,b] \cap [c,d] \ne \varnothing$ and thus~$[a,b] \cup [c,d] = [a,d]$ is a block of~$\building$.
This contradicts the maximality of~$[a,b]$ since ${[a,b] \subsetneq [a,d] \subsetneq [i,j]}$ as~$b < d < f \le j$.
\end{proof}

\begin{remark}
Note that there are building sets~$\building$ for which the type cone~$\typeCone(\nestedFan[\building])$ is simplicial, but which are not (isomorphic to) interval building sets.
See \eg \cref{fig:nestedFans}\,(left).
\end{remark}

We now translate the facet description of \cref{coro:facetDescriptionTypeConeNestedFan} to the specific case of interval building sets.
We need a few additional notations.
Consider an interval building set~$\building$ on~$[n]$.
For~${1 \le i < j \le n}$, define
\[
\ell(i,j) \eqdef \min\nolimits \set{k \in [i+1,j]}{[k,j] \in \building}
\qquad\text{and}\qquad
r(i,j) \eqdef \max\nolimits \set{k \in [i,j-1]}{[i,k] \in \building}.
\]
Note that~$\ell(i,j)$ and~$r(i,j)$ are well-defined since~$\building$ contain all singletons.
Observe that~${[i, r(i,j)]}$ and~$[\ell(i,j), j]$ are maximal strict subblocs of~$[i,j]$.
Therefore,
\begin{itemize}
\item if~$[i,j] \in \building$ is elementary, then we have~$r(i,j) < \ell(i,j)$ and the maximal strict subblocks of~$[i,j]$ are the intervals~$[s_{k-1}(i,j), s_k(i,j)-1]$ for~${k \in [p]}$ where the sequence ${s_0(i,j) < s_1(i,j) < \dots < s_p(i,j)}$ is defined by the boundary conditions~$s_0(i,j) \eqdef i$ and ${s_1(i,j) = r(i,j)+1}$ and ${s_p(i,j) \eqdef j+1}$, and the induction~$s_k(i,j) \eqdef r(s_{k-1}(i,j), j+1)+1$.
\item if~$[i,j] \in \building$ is not elementary, we have~$\ell(i,j) \le r(i,j)$ so that 
\[
\qquad
[i, r(i,j)] \cup [\ell(i,j), j] = [i,j]
\qquad\text{and}\qquad
[i, r(i,j)] \cap [\ell(i,j), j] = [\ell(i,j), r(i,j)]. 
\]
Thus $[i, r(i,j)]$ and~$[\ell(i,j), j]$ are the only maximal strict subblocks of~$[i,j]$.
Moreover, the connected components of~$[i, r(i,j)] \cap [\ell(i,j), j] = [\ell(i,j), r(i,j)]$ are the intervals ${[t_{k-1}(i,j), t_k(i,j)-1]}$ for~${k \in [q]}$~where the sequence ${t_0(i,j) < t_1(i,j) < \dots < t_q(i,j)}$ is defined by the boundary conditions~$t_0(i,j) \eqdef \ell(i,j)$ and~$t_q(i,j) \eqdef r(i,j)+1$, and the induction~${t_k(i,j) \eqdef r(t_{k-1}(i,j), r(i,j)+1) + 1}$.
\end{itemize}
Using these notations, the following statement is just a translation of \cref{coro:facetDescriptionTypeConeNestedFan}.

\begin{proposition}
\label{prop:facetDescriptionTypeConeNestedFanInterval}
Consider an interval building set~$\building$ on~$[n]$ and let~$\building^\star \eqdef \building \ssm \set{\{i\}}{i \in [n]}$ denote the blocks which are not singletons.
Then the inequalities
\begin{itemize}
\item $\sum_{k \in [p]} \b{h}_{[s_{k-1}(i,j), s_k(i,j)-1]} > \b{h}_{[i,j]}$ for all~$[i,j] \in \building^\star$ with~$r(i,j) < \ell(i,j)$,
\item $\b{h}_{[i, r(i,j)]} + \b{h}_{[\ell(i,j), j]} > \b{h}_{[i,j]} + \sum_{k \in [q]} \b{h}_{[t_{k-1}(i,j), t_k(i,j)-1]}$ for all~$[i,j] \in \building^\star$ with~${\ell(i,j) \le r(i,j)}$,
\end{itemize}
provide an irredundant facet description of the type cone~$\typeCone(\nestedFan[\building])$.
\end{proposition}

\begin{example}
For instance
\begin{itemize}
\item for the building set containing all intervals of~$[n]$, we have~$\ell(i,j) = i+1$ and~$r(i,j) = j-1$, so that the facet defining inequalities of the type cone are~$\b{h}_{[i,j-1]} + \b{h}_{[i+1,j]} > \b{h}_{[i,j]} + \b{h}_{[i+1,j-1]}$ for all~$1 \le i < j \le n$ (with the convention that~$\b{h}_{[i+1,j-1]} = 0$ for~$i+1=j$),
\item for the building set containing all singletons~$\{i\}$ and all intervals~$[i]$ for~$i \in [n]$, we have $r(1,j) = j-1 < j = \ell(1,j)$, so that the facet defining inequalities of the type cone are~$\b{h}_{[j-1]} + \b{h}_{\{j\}} > \b{h}_{[j]}$ for all~$1 < j \le n$.
\end{itemize}
\end{example}

Generalizing \cref{prop:kinematicAssociahedra}, we finally combine \cref{prop:simplicialTypeCone,prop:facetDescriptionTypeConeNestedFanInterval} to define \defn{kinematic nestohedra} for interval building sets, similar to the constructions of~\cite{ArkaniHamedBaiHeYan, BazierMatteDouvilleMousavandThomasYildirim, PadrolPaluPilaudPlamondon} for associahedra, cluster associahedra and gentle associahedra.
Again, these polytopes are just affinely equivalent to the realizations in~$\R^n$, but they should be more natural from a mathematical physics perspective.

\begin{proposition}
\label{prop:kinematicNestohedraInterval}
Consider an interval building set~$\building$ on~$[n]$ and let~$\building^\star \eqdef \building \ssm \set{\{i\}}{i \in [n]}$ denote the blocks which are not singletons.
Then for any~$\b{p} \in \R_{>0}^{\building^\star}$, the polytope~$R_\b{p}(\building) \subseteq \R^\building$ defined as the intersection of the positive orthant~$ \set{\b{z} \in \R^{\building}}{\b{z} \ge 0}$ with the hyperplanes
\begin{itemize}
\item $\b{z}_{K} = 0$ for~$K \in \connectedComponents(\building)$,
\item $\sum_{k \in [p]} \b{z}_{[s_{k-1}(i,j), s_k(i,j)-1]} - \b{z}_{[i,j]} = \b{p}_{[i,j]}$~for~$[i,j] \in \building^\star$ with~$r(i,j) < \ell(i,j)$,
\item $\b{z}_{[i, r(i,j)]} + \b{z}_{[\ell(i,j), j]} - \b{z}_{[i,j]} - \sum_{k \in [q]} \b{z}_{[t_{k-1}(i,j), t_k(i,j)-1]} = \b{p}_{[i,j]}$ for~$[i,j] \in \building^\star$ with~${\ell(i,j) \le r(i,j)}$,
\end{itemize}
%\[
%R_\b{p}(\building) \eqdef \Bigset{\b{z} \in \R^{\building}}{\begin{array}{l} \b{z} \ge 0 \quad\text{ and }\quad \b{z}_{K} = 0 \text{ for all } K \in \connectedComponents(\building) \\ \sum_{k \in [p]} \b{z}_{[s_{k-1}(i,j), s_k(i,j)-1]} - \b{z}_{[i,j]} = \b{p}_{[i,j]} \text{ for all } [i,j] \in \building \text{ with } r(i,j) < \ell(i,j) \\ \b{z}_{[i, r(i,j)]} + \b{z}_{[\ell(i,j), j]} - \b{z}_{[i,j]} - \sum_{k \in [q]} \b{z}_{[t_{k-1}(i,j)+1, t_k(i,j)]} = \b{p}_{[i,j]} \text{ for all } [i,j] \in \building \text{ with } \ell(i,j) \le r(i,j) \end{array}}
%\]
is a nestohedron whose normal fan is the nested fan~$\nestedFan[\building]$.
Moreover, the polytopes~$R_\b{p}(\building)$ for ${\b{p} \in \R_{>0}^{\building^\star}}$ describe all polytopal realizations of~$\nestedFan[\building]$ (up to translations).
\end{proposition}

%%%%%%%%%%%%%%%%%%%%%%%%%%%%%%%%%%%%%%%

\addtocontents{toc}{\vspace{.1cm}}
\section*{Acknowledgments}

The graphical part of the present paper appeared in a preliminary version of our paper with Yann Palu and Pierre-Guy Plamondon~\cite{PadrolPaluPilaudPlamondon} as an illustration of the limits of our method based on the simpliciality of the type cone.
We later realized during the master project of Germain Poullot that, although many complications appear, the main results and techniques leading to the description of the type cone can be extended from graph associahedra to arbitrary nestohedra.
We are grateful to Yann Palu and Pierre-Guy Plamondon for encouraging us to separate this part from~\cite{PadrolPaluPilaudPlamondon}  to write the present paper.
We also thank two anonymous referees for helpful suggestions on a preliminary version of this paper.

%%%%%%%%%%%%%%%%%%%%%%%%%%%%%%%%%%%%%%%

\bibliographystyle{alpha}
\bibliography{typeConeNestohedra}
\label{sec:biblio}

\end{document}